\documentclass[11pt, leqno]{amsart}

\usepackage{
  amsmath,
  amsthm,
  amssymb,
  tikz,
  fancyhdr,
  enumitem,
  setspace,
  multirow,
  mathdots,
  upgreek,
  yhmath,
  ytableau,
  mathrsfs,
}

\usepackage[all, knot]{xy}
\xyoption{arc}

\usetikzlibrary{arrows,calc,matrix,shapes}

\usepackage[colorlinks=true, pdfstartview=FitV,linkcolor=black,citecolor=black,urlcolor=black]{hyperref}

\usepackage[vcentermath,enableskew]{youngtab}
\usepackage{lscape}
\usepackage[center,small,sc]{caption}

\usepackage{multicol}

\theoremstyle{plain}
\newtheorem{theorem}{Theorem}[section]
\newtheorem{lemma}[theorem]{Lemma}
\newtheorem{proposition}[theorem]{Proposition}
\newtheorem{corollary}[theorem]{Corollary}
\newtheorem{conjecture}[theorem]{Conjecture}

\theoremstyle{definition}
\newtheorem{definition}[theorem]{Definition}
\newtheorem{example}[theorem]{Example}

\newtheorem{remark}[theorem]{Remark}
\newtheorem{problem}[theorem]{Problem}

\numberwithin{equation}{section}
\newcommand{\xx}{\mathbf{x}}
\newcommand{\g}{\mathfrak{g}}

\newcommand{\iso}{\cong} % Isomorphic
\newcommand{\fw}{\varpi} % Fundamental weight
\newcommand{\tfw}{\widetilde{\fw}} % Tilde fundamental weight

\newcommand{\seteq}{\mathbin{:=}}

\newcommand{\ntau}[1]{\btau^{(#1)}}

\newcommand{\cn}[1]{\mathbf{#1}}

\newcommand{\clr}{\mathrm{c}}
\newcommand{\blam}{\uplambda}
\newcommand{\btau}{\uptau}
\newcommand{\bmu}{\upmu}
\newcommand{\ASP}{\mathrm{ASP}}
\newcommand*\gn[1]{\tikz[baseline=(char.base)]{%
            \node[shape=circle,fill=gray!40,draw,inner sep=1.3pt] (char) {#1};}}
\newcommand*\nn[1]{\tikz[baseline=(char.base)]{%
            \node[shape=circle,draw,inner sep=1.3pt] (char) {#1};}}

\newcommand{\Cat}{\mathcal{C}} % Catalan numbers (to distinguish from type C_n)
\newcommand{\CatCR}{\widetilde{\Cat}} % Carlitz-Riordan Catalan numbers
\newcommand{\Mot}{\mathcal{M}} % Motzkin numbers
\newcommand{\Rior}{\mathcal{R}} % Riordan numbers
\newcommand{\Dyck}{\mathcal{D}} % Set of Dyck paths
\newcommand{\qbinom}[3]{\left[ \begin{matrix} #1 \\ #2 \end{matrix} \right]_{#3}} % q-binomal
\newcommand{\qPdf}[1]{\langle q \rangle_{#1}} % q-Pochhammer "double factorial"
\newcommand{\nilp}{\mathbf{P}} % Non-intersecting lattice paths

\newcommand{\Z}{\mathbb{Z}}
\newcommand{\B}{\mathfrak{B}}

\newcommand{\mcL}{\mathcal{L}}

\newcommand{\tens}{\mathop\otimes}
\newcommand{\virtual}[1]{\widehat{#1}}  % virtualization map
\newcommand{\wt}{\operatorname{wt}} % weight
\newcommand{\sgn}{\operatorname{sgn}} % sign
\newcommand{\sig}{\operatorname{sig}} % signature
\newcommand{\ch}{\operatorname{ch}} % character
\newcommand{\ps}{\operatorname{ps}} % principal specialization
\newcommand{\nps}{\widetilde{\ps}} % normalized principal specialization
\newcommand{\Des}{\operatorname{Des}} % Descent set
\newcommand{\maj}{\operatorname{maj}} % major index

\newcommand{\Sh}{\mathsf{Sh}}

\def\Q{\mathbb Q}

\def\ZZ{\mathcal Z}

\newcommand{\ofour}{\overline{4}}
\newcommand{\othree}{\overline{3}}
\newcommand{\otwo}{\overline{2}}
\newcommand{\one}{\overline{1}}
\newcommand{\on}{\overline{n}}
\newcommand{\ok}{\overline{k}}

\newcommand{\dyckgrid}[1]{
\foreach \i in {0, ..., #1} {\draw[gray!40, very thin] (\i,\i) -- (#1, \i); \draw[gray!40, very thin] (\i,\i) -- (\i, 0);}
}

\usepackage{listings}
\lstdefinelanguage{Sage}[]{Python}
{morekeywords={False,True},sensitive=true}
\definecolor{darkred}{rgb}{0.7,0,0} % darkred color
\newcommand{\defn}[1]{{\color{darkred}\emph{#1}}} % emphasis of a definition

\usepackage[colorinlistoftodos]{todonotes}

\title{Identities from representation theory}
\author[S.-j.~Oh]{Se-jin Oh}
\address[S.-j.~Oh]{Ewha Womans University Seoul, 52 Ewhayeodae-gil, Daehyeon-dong, Seodaemun-gu, Seoul, South Korea}
\email{sejin092@gmail.com}
\urladdr{https://sites.google.com/site/mathsejinoh/}

\author[T.~Scrimshaw]{Travis Scrimshaw}
\address[T. Scrimshaw]{School of Mathematics and Physics, The University of Queensland, St. Lucia, QLD 4072, Australia}
\email{tcscrims@gmail.com}
\urladdr{https://sites.google.com/view/tscrim/home}

\thanks{SjO was partially supported by the National Research Foundation of Korea(NRF) Grant funded by the Korea government(MSIP) (NRF-2016R1C1B2013135).
TS was partially supported by the National Science Foundation RTG grant NSF/DMS-1148634 and the Australian Research Council DP170102648.}

\begin{document}

\begin{abstract}
We give a new Jacobi--Trudi-type formula for characters of finite-dimensional irreducible representations in type $C_n$ using characters of the fundamental representations and non-intersecting lattice paths.
We give equivalent determinant formulas for the decomposition multiplicities for tensor powers of the spin representation in type $B_n$ and the exterior representation in type $C_n$.
This gives a combinatorial proof of an identity of Katz and equates such a multiplicity with the dimension of an irreducible representation in type $C_n$.
By taking certain specializations, we obtain identities for $q$-Catalan triangle numbers, the $q,t$-Catalan number of Stump, $q$-triangle versions of Motzkin and Riordan numbers, and generalizations of Touchard's identity.
We use (spin) rigid tableaux and crystal base theory to show some formulas relating Catalan, Motzkin, and Riordan triangle numbers.
\end{abstract}

\maketitle
\tableofcontents

%%%%%%%%%%%%%%%%%%%%%%%%%%%%%%%%%%%%%%%%
\section{Introduction}

The interaction between combinatorics and representation theory has a long history. For example, by enumerating bases and modules by combinatorial objects, we can translate problems in representation theory into computations using combinatorial rules that are often simple. On the other hand, the algebraic structures frequently imply certain identities or positivity results.

For example, the only known proof that the $q,t$-Catalan number $\CatCR_n(q,t)$ of Garsia and Haiman~\cite{GH96} is symmetric in $q$ and $t$, \textit{i.e.} $\CatCR_n(q,t) = \CatCR_n(t,q)$, is by showing that $\CatCR_n(q,t)$ is the bi-graded Hilbert series of a certain representation that is naturally symmetric in $q$ and $t$~\cite{GH01,GH02}. It is a famous open problem in combinatorics to prove this bijectively.
Another problem in combinatorial group theory (and geometry) was the positivity of Kazhdan--Lusztig polynomials and was proven by using the structure of Soergel bimodules~\cite{EW14,Soergel07}.
Representation theory has also been applied to show certain polynomials from combinatorics are symmetric and unimodel (see, \textit{e.g.},~\cite{Stanley80} and references therein).

One important relationship is given by constructing the dimension of a $\mathfrak{gl}_n$-representation $V(\lambda)$ by using semistandard Young tableaux of shape $\lambda$. The character of $V(\lambda)$ is then a Schur function $s_{\lambda}$, which also is well-defined in the stable limit as $n \to \infty$. One can show that combinatorially that Schur functions are a basis for the ring of symmetric functions, which implies that finite-dimensional $\mathfrak{gl}_n$ modules are fully reducible and $V(\lambda) \iso V(\mu)$ if and only if $\lambda = \mu$. In turn, the fact that $s_{\lambda} s_{\mu}$ corresponds to $V(\lambda) \otimes V(\mu)$ implies that the Littlewood--Richardson coefficients are non-negative. Furthermore, the Lindstr\"om--Gessel--Viennot (LGV) lemma~\cite{Lindstrom73,GV85} can be applied to give a bijective proof of the Jacobi--Trudi identity to compute Schur functions. For more information on Schur functions, see, \textit{e.g.},~\cite[Ch.~7]{ECII}.

For the other classical Lie algebras, the problem is more intricate. There have been a number of different tableaux to enumerate the bases of a finite-dimensional irreducible representation of $\mathfrak{so}_n$~\cite{KW93}, including by King and El-Sharkaway~\cite{KElS83}, King and Welsh~\cite{KW93}, Koike and Terada~\cite{KT90}, Proctor~\cite{Proctor90III,Proctor94}, and Sundaram~\cite{Sundaram90}. For $\mathfrak{sp}_{2n}$, the representations can also be indexed by certain tableaux~\cite{Berele86}, which include those by De Concini~\cite{dC79}, King~\cite{King76}, King and El-Sharkaway~\cite{KElS83}, Sundaram~\cite{Sundaram86}.
In~\cite{KT87}, Koike and Terada give two new bases for the ring of symmetric functions that corresponds to stable limits of $\mathfrak{sp}_{2n}$ and $\mathfrak{so}_n$ characters.
Jacobi--Trudi-type formulas and bijective proofs using the LGV lemma for the characters of $\mathfrak{sp}_{2n}$ and $\mathfrak{so}_n$ were given in~\cite{FK97,SV16}.
%Chari and Kleber constructed representations of type $BCD$ whose decomposition multiplicities are Littlewood--Richardson coefficients and conjecture they correspond to of the minimal affinization the highest weight $U_q(\g)$-module $V(\lambda)$ specialized at $q = 1$~\cite{CK02}.

One of the major advancements in both combinatorics and representation theory was Kashiwara's crystal bases~\cite{K90,K91}, which are the $q \to 0$ limit of bases of representations of a Drinfel'd--Jimbo quantum group $U_q(\g)$. Crystal give representation theoretic interpretations of and interconnected a number of combinatorial constructions such as coplactic operators~\cite{Lothaire02}, evacuation~\cite{Lenart07}, promotion~\cite{Shimozono02}, and charge~\cite{NY97}. Crystals also yielded a new tableaux model for $V(\lambda)$~\cite{KN94}, Levi branching rules, and a combinatorial method to compute the decompositions of tensor products in a unified framework.

Our main result (Theorem~\ref{thm:general_det_formula_Cn}) is a Jacobi--Trudi-type formula for characters of irreducible representations in terms of representations corresponding to fundamental weights (\textit{i.e.} given by single column tableaux).
In particular, our formula is distinct from~\cite[Eq.~(3.9),~(3.10) and~(3.11)]{FK97}, which uses the reflection principle and can involve multiple terms in each entry of the matrix.
Moreover, it is different than~\cite[Thm.~3.2]{SV16}, which is a ``dual'' version of our results as it uses single row tableaux representations (\textit{i.e.,}, symmetric powers of the natural representation or homogeneous symmetric functions).
Our proof uses the LGV lemma on partial Dyck paths and King tableaux, which yields a determinant formula for the dimension using Catalan triangle numbers.
It is also distinct from~\cite{KT87}, which uses characters of single column tableaux in type $A$ (\textit{i.e.}, elementary symmetric functions).

Our second main result are determinant formulas for the decomposition multiplicities of a tensor power of the spin representation in type $B_n$ (Theorem~\ref{thm:triangular_Catalan_det}) and the exterior algebra of $\mathfrak{sp}_{2n}$ (Theorem~\ref{thm:triangular_Catalan_det_type_C}). Again, our proof is applying the LGV lemma to partial Dyck paths. In particular, when we consider the multiplicity of the trivial representation, we have a natural identification of the two lattice paths as proper Dyck paths. This gives a direct bijective proof of a result of Katz~\cite[Thm~1.4]{Katz16}. By using our interpretation using crystals, we can see that this identification essentially corresponds to the virtualization map of type $C_n$ to $B_n$ (see, {\it e.g.},~\cite{K96,OSS03III,OSS03II,SchillingS15}). Thus, we also provide a representation theoretic proof~\cite[Thm.~1.4]{Katz16}, answering a problem posed by Katz~\cite[Sec.~4]{Katz16}.

The remainder of our results are about various identities that arise from representation theory. We show that the (natural) $q$-Catalan number $\Cat_n(q)$ from MacMahon~\cite{MacMahon15} is equal to the principle specialization of a type $C_{n-1}$ character (up to a power of $q$) (Theorem~\ref{thm:q_catalan_paths}) by giving a representation theoretic formulation of the $q,t$-Catalan number given by Stump~\cite{Stump08} (Theorem~\ref{thm:qt_polynomial}). Using this as a guide, we give a new $q,t$-Catalan triangle numbers using specializations of type $C_{n-1}$ characters. 
Furthermore, we extend this to a $q,t$-binomial coefficient that specializes to the usual $q$-binomial coefficients at $t = q^{-1}$ (up to a power of $q$) (Theorem~\ref{thm:q_binomial_paths}).

Additionally, we prove the classical Touchard identity using representations in type $C_n$, extend this to the Catalan triangle numbers, and give a new $q,t$-analog that specializes to the principle specialization of a type $C_n$ character.
We also give representation theoretic proofs of some classical identities between Catalan numbers and Motzkin/Riordan numbers. 
We provide tableau models for highest weight crystals in types $B_n$ and $D_n$ using semistandard (spin) rigid tableaux of~\cite{KLO17} and are equinumerous to Motzkin/Riordan triangle numbers in various ways.
%We note that Motzkin triangle numbers have been shown to have applications in biology~\cite{Nkwanta97}.
Furthermore, we give new (recursive) $q$-analogs of Motzkin and Riordan (triangle) numbers, which we conjecture to have a combinatorial interpretation.

This paper is organized as follows.
In Section~\ref{sec:background}, we provide the necessary background.
In Section~\ref{sec:JT_type_C}, we give a Jacobi--Trudi-type formula for characters in type $C_n$.
In Section~\ref{sec:decomposition_multiplicities}, we compute the decomposition multiplicities of tensor powers of the type $B_n$ spin and type $C_n$ exterior representation.
In Section~\ref{sec:repr_identities}, we provide a representation theoretic proof of several combinatorial identities.
In Section~\ref{sec:rigid_tableaux}, we use semistandard rigid tableaux for representations in type $B_n$ to give $(s+1)$-many
distinct tableaux models which are equinumerous to the Motzkin triangle number $\Mot_{(m,s)}$.
In Section~\ref{sec:spin_rigid_tableaux}, we show identities involving Riordan triangle numbers using semistandard spin rigid tableaux and type $D_n$ representations.
In Section~\ref{sec:problems}, we conclude with a number of open problems.

After completion of this manuscript, the authors were informed of the preprint~\cite{Okada89} of Soichi Okada that also proves Theorem~\ref{thm:general_det_formula_Cn} using the LGV lemma and King tableaux. The authors additionally learned of~\cite{Okada09}, which provides an alternative proof of Corollary~\ref{cor: Cn determinatal r omega_n} and Proposition~\ref{prop:determinant_double_spin_B}.

%%%%%%%%%%%%%%%%%%%%%%%%%%%%%%%%%%%%%%%%
\section{Background}
\label{sec:background}

In this section, we give the necessary background.

% ========
\subsection{Catalan triangle numbers}

The Catalan numbers are a well-studied sequence of numbers whose history spans over the past three centuries, where the earliest recorded discovery is by the Chinese mathematician Antu Ming around 1730~\cite{Larcombe99,Luo88}. The $n$-th Catalan number is
\[
\Cat_n \seteq \frac{1}{2n+1} \binom{2n}{n}
\]
and also satisfies the recursion
\begin{equation}
\label{eq:Catalan_recurrence}
\Cat_n = \sum_{k=1}^n \Cat_{k-1} \Cat_{n-k}
\end{equation}
starting with $\Cat_0 = 1$. Stanley has given a list of over 200 objects count the Catalan numbers~\cite{ECII,Stanley15}, and this list continues to grow, {\it e.g.},~\cite{Reynolds15}.

One such interpretation is given by the set of all Dyck paths in an $n \times n$ grid that stay weakly below the diagonal $y = x$, which we denote by $\Dyck_n$. To be precise, we consider a directed graph, known as the \defn{Catalan graph}, with vertices $V = \{(i,j) \mid 0 \leq j \leq i \leq n\}$ and (directed) edges
\[
N = (i,j) \to (i,j+1),
\qquad\qquad
E = (i,j) \to (i+1,j),
\]
which we call North steps and East steps respectively. A \defn{Dyck path} is a path from $(0, 0)$ to $(n, n)$ in the Catalan graph. Indeed, we have $\Cat_n = \lvert \Dyck_n \rvert$. We call the corresponding word in the alphabet $\{ N, E \}$ a \defn{Dyck word}.

We will also require the following generalization of the Catalan numbers that appeared as early as 1800~\cite{Arbogast}. The \defn{$(n,k)$-th Catalan triangle number} is defined by
\[
\Cat_{(n,k)} \seteq \binom{n+k}{k}-\binom{n+k}{k-1} = \frac{(n+k)! (n-k+1)}{k! (n+1)!} = \Cat_{(n,k-1)} + \Cat_{(n-1,k)}
\]
and can be described combinatorially the set of all paths in the Catalan graph that start at $(0,0)$ and end at $(n,k)$~\cite{Bailey96}.
Thus, we denote the set of such paths by $\Dyck_{(n,k)}$, which is equivalent be the set of partial Dyck words with $k$ occurrences of the letter $N$ and $n$ occurrences of the letter $E$.
We consider $\Cat_{(n,k)} = 0$ whenever $k < 0$ or $k > n$.
Note that $\Cat_{(n,n)} = \Cat_{(n,n-1)} = \Cat_n$.

% ========
\subsection{Motzkin triangle numbers}

Another well-studied sequence of numbers are the Motzkin numbers that were introduced by Motzkin~\cite{Motzkin48}.
The \defn{$n$-th Motzkin number} is
\[
\Mot_n \seteq \sum_{k=0}^{\lfloor n/2 \rfloor} \Cat_k \binom{n}{2k}
\]
and satisfies the recurrence
\[
\Mot_n = \Mot_{n-1} + \sum_{k=0}^{n-2} \Mot_k \Mot_{n-2-k} = \frac{2n+1}{n+2} \Mot_{n-1} + \frac{3n-3}{n+2} \Mot_{n-2}.
\]
The $n$-th Motzkin number has an interpretation of the number of ways of drawing non-intersecting chords on $n$ points on a circle. They also have an interpretation similar to Dyck paths, called \defn{Motzkin paths}, using steps
\[
U = (i,j) \to (i+1,j+1),
\qquad
H = (i, j) \to (i+1, j),
\qquad
D = (i, j) \to (i+1, j-1),
\]
that starts at $(0, 0)$, goes to $(n, 0)$, and stays weakly above the $y = 0$ horizontal line. Note that a Dyck path is a Motzkin path without a horizontal step $H$.
In~\cite{Callen17}, it was shown that $\Mot_{n-1}$ is the number of Dyck paths in $\Dyck_n$ whose peaks all occur at odd height.

There is a triangle version of the Motzkin numbers~\cite{Lando03} (see also~\cite[A026300]{OEIS}). 
The \defn{$(n,k)$-th Motzkin triangle number} is defined by
\begin{align*}
\Mot_{(n,k)} & \seteq \sum_{i=0}^{\lfloor (n-k)/2 \rfloor} \binom{n}{2i+k} \left[ \binom{2i+k}{i} - \binom{2i+k}{i-1} \right]
\\ & = \Mot_{(n-1,k)} + \Mot_{(n-1,k-1)} + \Mot_{(n-1,k+1)}.
\end{align*}
These count the number of Motzkin paths that end at $(n,k)$, and so $\Mot_{(n,0)} = \Mot_n$ and $\Mot_{(n,n)} = 1$. Note that this differs from our convention for Catalan triangle numbers.

% ========
\subsection{Riordan triangle numbers}

The Riordan numbers are another sequence of numbers closely related to Motzkin numbers that were first introduced by Riordan~\cite{Riordan75}, where the name was coined by Bernhart~\cite{Bernhart97}. The \defn{$n$-th Riordan number} is defined by
\[
\Rior_n \seteq \frac{1}{n+1} \sum_{k=1}^{\lfloor n/2 \rfloor} \binom{n+1}{k} \binom{n-k-1}{k-1} = \frac{n-1}{n+1} \left(2 \Rior_{n-1} + 3 \Rior_{n-2} \right)
\]
with $\Rior_0 = 1$ and $\Rior_1 = 0$. The $n$-th Riordan number has an interpretation as the number of Motzkin paths from $(0,0)$ to $(n,0)$ that does not have any horizontal steps on the $y=0$ line. Such paths are called \defn{Riordan paths}. Another interpretation is the number of Dyck paths in $\Dyck_n$ whose peaks all occur at even height~\cite{Callen17}.

Let $\Rior_{(n,k)}$ denote the \defn{$(n,k)$-th Riordan triangle number} defined by
\[
\Rior_{(n,k)} \seteq \sum_{i=0}^{n-k} (-1)^i (\Mot_{(n+1-i,k)} + \Mot_{(n+1-i,k-1)}).
\]
The $(n,k)$-th Riordan triangle number has an interpretation as the Riordan paths that end at $(n,k)$, and thus, we have $\Rior_{(n,0)} = \Rior_n$ and $\Rior_{(n,n)} = 1$. Note that this differs from our convention for Catalan triangle numbers.

One of the earliest known (to the authors) occurrences of the Riordan triangle numbers is in the paper~\cite{Bernhart97} by taking a difference of trinomial coefficients (see~\cite[Fig.~22(b)]{Bernhart97}). The Riordan triangle numbers also appeared in the same year in a paper by~\cite[Fig.~5]{MRSV97}, but under a different interpretation. To convert the notation/paths of~\cite{MRSV97} to the standard Riordan paths given here, use the map
\[
e \mapsto U,
\qquad\qquad\qquad
ne \mapsto H,
\qquad\qquad\qquad
n^2e \mapsto D.
\]
We leave it as an exercise to the reader to show this is indeed a bijection. 

\begin{lemma} \label{lem:Mot_Rio_relation} \cite[Lemma~4.10]{KLO17} For $m,s \ge 1$, we have
\[
\Rior_{(m,s)} + \Rior_{(m-1,s)}=\Mot_{(m-1,s)} + \Mot_{(m-1,s-1)}.
\]
\end{lemma}

% ========
\subsection{Lindstr\"om--Gessel--Viennot lemma}

We state one of our main computational tools, the \defn{Lindstr\"om--Gessel--Viennot (LGV) lemma}.

Let $G$ be a finite directed acyclic graph with edge weights $w \colon E(G) \to R$ for some commutative ring $R$. Let
\[
e(u, v) = \sum_{P \colon u \to v} \prod_{e \in P} w(e),
\]
where the sum is over all paths $P$ from $u$ to $v$. Next, fix some initial vertices $\mathbf{s} = (s_1, \dotsc, s_n)$ and terminal vertices $\mathbf{t} = (t_1, \dotsc, t_n)$. A \defn{family of non-intersecting lattice paths} $\nilp = (P^{(1)}, \dotsc, P^{(n)})$ from $\mathbf{s}$ to $\mathbf{t}$ are paths $P^{(i)} \colon s_i \to t_{\sigma(i)}$, for some fixed $\sigma \in \mathfrak{S}_n$, such that $P^{(i)} \cap P^{(j)} = \emptyset$ for all $i \neq j$. The \defn{sign} of $\nilp$ is defined by $\sgn(\nilp) = \sgn(\sigma)$.

\begin{lemma}[{\cite{Lindstrom73,GV85}}]
\label{lemma:LGV}
We have
\[
\det \begin{bmatrix}
e(s_1, t_1) & e(s_1, t_2) & \cdots & e(s_1, t_n) \\
e(s_2, t_1) & e(s_2, t_2) & \cdots & e(s_2, t_n) \\
\vdots & \vdots & \ddots & \vdots \\
e(s_n, t_1) & e(s_n, t_2) & \cdots & e(s_n, t_n)
\end{bmatrix}
= \sum_{\nilp} \sgn(\nilp) \prod_{i=1}^n w(P^{(i)}),
\]
where the sum is over all non-intersecting lattices paths $\nilp = (P^{(1)}, \dotsc, P^{(n)})$ from $\mathbf{s}$ to $\mathbf{t}$.
\end{lemma}

We will simply refer to Lemma~\ref{lemma:LGV} as the LGV lemma.

% ========
\subsection{Crystals}
\label{sec:crystals}

Let $\g$ be a simple Lie algebra with index set $I$, Cartan matrix $(A_{ij})_{i,j \in I}$, simple roots $(\alpha_i)_{i \in I}$, fundamental weights $(\fw_i)_{i \in I}$, weight lattice $P_{\Z}$, simple coroots $(\alpha_i^{\vee})_{i \in I}$, and canonical pairing $\langle\ ,\ \rangle \colon P_{\Z}^{\vee} \times P_{\Z} \to \Z$ given by $\langle \alpha_i^{\vee}, \alpha_j \rangle = A_{ij}$.
We will use the standard identification of the weight lattice as a sublattice of $\bigoplus_{i=1}^n \Q\epsilon_i$ (see, \textit{e.g.},~\cite{BS17}).
Let $U_q(\g)$ denote the corresponding (Drinfel'd--Jimbo) quantum group.
%Define $t_i^{\vee} \seteq \max(c_i^{\vee}/c_i, c_0)$, where $c_i$ and $c_i^{\vee}$ are the Kac and dual Kac labels respectively~\cite[Table Aff1-3]{kac90}.
%We write $i \sim j$ if $A_{ij} \neq 0$ and $i \neq j$.
We denote
\[
\tfw_i = \begin{cases}
2\fw_n & \text{if $\g = B_n$ and $i = n$},
\\ \fw_n + \fw_{n+1} & \text{if $\g = D_{n+1}$ and $i = n$},
\\ 2\fw_i & \text{if $\g = D_{n+1}$ and $i = n+1$},
\\ \fw_i & \text{otherwise}.
\end{cases}
\]

A \defn{$U_q(\g)$-crystal} is a set $B$ with operations $e_i, f_i \colon B \to B \sqcup \{0\}$, for $i \in I$, and weight function $\wt \colon B \to P_{\Z}$ such that
\begin{enumerate}
\item $\langle \wt(b), \alpha_i^{\vee} \rangle + \varepsilon_i(b) = \varphi_i(b)$, where
\begin{itemize}
\item $\varepsilon_i(b) = \max \{ k \in \Z_{\geq 0} \mid e_i^k b \neq 0 \}$,
\item $\varphi_i(b) = \max \{ k \in \Z_{\geq 0} \mid f_i^k b \neq 0 \}$,
\end{itemize}
\item $\wt(e_i b) = \wt(b) + \alpha_i$ if $e_i b \neq 0$,
\item $e_i b = b'$ if and only if $b = f_i b'$ for all $b, b' \in B$,
\item is a crystal basis, in the sense of Kashiwara~\cite{K90,K91}, of a $U_q(\g)$-module.
\end{enumerate}

\begin{remark}
Our definition of a crystal definition includes the property sometimes called \defn{seminormal} or \defn{regular}.
\end{remark}

In~\cite{K90}, Kashiwara showed that all highest weight modules $V(\lambda)$ for $\lambda \in P_{\Z}^+$ admit a crystal basis. Let $B(\lambda)$ denote the crystal basis of $V(\lambda)$, and let $u_{\lambda}$ be the \defn{highest weight element} of $B(\lambda)$, the (unique) element such that $e_i u_{\lambda} = 0$ for all $i \in I$.

Kashiwara showed that $U_q(\g)$-crystals form a tensor category~\cite{K91}. We define the \defn{tensor product} of $U_q(\g)$-crystals $B_1$ and $B_2$ as the crystal $B_2 \otimes B_1$ with elements being the Cartesian product $B_2 \times B_1$ with the crystal structure
\begin{align*}
e_i(b_2 \otimes b_1) & = \begin{cases}
e_i b_2 \otimes b_1 & \text{if } \varepsilon_i(b_2) > \varphi_i(b_1), \\
b_2 \otimes e_i b_1 & \text{if } \varepsilon_i(b_2) \leq \varphi_i(b_1),
\end{cases}
\\ f_i(b_2 \otimes b_1) & = \begin{cases}
f_i b_2 \otimes b_1 & \text{if } \varepsilon_i(b_2) \geq \varphi_i(b_1), \\
b_2 \otimes f_i b_1 & \text{if } \varepsilon_i(b_2) < \varphi_i(b_1),
\end{cases}
\\ \varepsilon_i(b_2 \otimes b_1) & = \max(\varepsilon_i(b_1), \varepsilon_i(b_2) - \langle h_i, \wt(b_1) \rangle),
\\ \varphi_i(b_2 \otimes b_1) & = \max(\varphi_i(b_2), \varphi_i(b_1) + \langle h_i, \wt(b_2) \rangle),
\\ \wt(b_2 \otimes b_1) & = \wt(b_2) + \wt(b_1).
\end{align*}

\begin{remark}
Our tensor product convention follows~\cite{BS17}, which is opposite of Kashiwara~\cite{K91}.
\end{remark}

We can simplify the tensor product rule on the tensor product of $U_q(\g)$-crystals $B = B_L \otimes \dotsm \otimes B_1$ by using the \defn{signature rule}. Let $b = b_L \otimes \cdots \otimes b_2 \otimes b_1 \in B$, and for $i \in I$, we write
\[
\underbrace{-\cdots-}_{\varphi_i(b_L)}\
\underbrace{+\cdots+}_{\varepsilon_i(b_L)}\
\cdots\
\underbrace{-\cdots-}_{\varphi_i(b_1)}\
\underbrace{+\cdots+}_{\varepsilon_i(b_1)}\ .
\]
Then by successively deleting any $(+-)$-pairs (in that order) in the above sequence, we obtain a sequence
\[
\sig_i(b) \seteq
\underbrace{-\cdots-}_{\varphi_i(b)}\
\underbrace{+\cdots+}_{\varepsilon_i(b)}
\]
called the \defn{reduced signature}. Suppose $1 \leq j_-, j_+ \leq L$ are such that $b_{j_-}$ contributes the rightmost $-$ in $\sig_i(b)$ and $b_{j_+}$ contributes the leftmost $+$ in $\sig_i(b)$.
Then, we have
\begin{align*}
e_i b &= b_L \otimes \cdots \otimes b_{j_++1} \otimes e_ib_{j_+} \otimes b_{j_+-1} \otimes \cdots \otimes b_1, \\
f_i b &= b_L \otimes \cdots \otimes b_{j_-+1} \otimes f_ib_{j_-} \otimes b_{j_--1} \otimes \cdots \otimes b_1.
\end{align*}

Let $B_1$ and $B_2$ be two $U_q(\g)$-crystals.
A \defn{crystal morphism} $\psi \colon B_1 \to B_2$ is a map $B_1 \sqcup \{0\} \to B_2 \sqcup \{0\}$ with $\psi(0) = 0$ such that the following properties hold for all $b \in B_1$ and $i \in I$:
\begin{itemize}
\item[(1)] If $\psi(b) \in B_2$, then $\wt\bigl(\psi(b)\bigr) = \wt(b)$, $\varepsilon_i\bigl(\psi(b)\bigr) = \varepsilon_i(b)$, and $\varphi_i\bigl(\psi(b)\bigr) = \varphi_i(b)$.
\item[(2)] We have $\psi(e_i b) = e_i \psi(b)$ if $\psi(e_i b) \neq 0$ and $e_i \psi(b) \neq 0$.
\item[(3)] We have $\psi(f_i b) = f_i \psi(b)$ if $\psi(f_i b) \neq 0$ and $f_i \psi(b) \neq 0$.
\end{itemize}
An \defn{embedding} (resp.~\defn{isomorphism}) is a crystal morphism such that the induced map $B_1 \sqcup \{0\} \to B_2 \sqcup \{0\}$ is an embedding (resp.~bijection).
%A crystal morphism is \defn{strict} if it commutes with all crystal operators.

For a $U_q(\g)$-representation $V$, let $\ch V$ denote the character.
Recall that for a representation $V$ with a crystal basis $B$, we have
\[
\dim V = \lvert B \rvert,
\qquad\qquad
\ch V = \sum_{b \in B} x^{\wt(b)},
\]
where for $\lambda = \sum_{i=1}^n c_i \epsilon_i$, we write $x^{\lambda} = \prod_{i=1}^n x_i^{c_i}$. Furthermore, recall that the character is invariant under the action of the Weyl group corresponding to $\g$.
To simplify our notation, we denote $\ch(\fw) \seteq \ch V(\fw)$.

% ========

\begin{figure}
\[
\begin{array}{|rl|}\hline
A_n: &
\begin{tikzpicture}[xscale=1.9,baseline=-4]
\node (1) at (0,0) {$\young(1)$};
\node (2) at (1.5,0) {$\young(2)$};
\node (d) at (3.0,0) {$\cdots$};
\node (n-1) at (4.5,0) {$\young(n)$};
\node (n) at (6,0) {$\boxed{n+1}$};
\draw[->,red] (1) to node[above]{\tiny$1$} (2);
\draw[->,purple] (2) to node[above]{\tiny$2$} (d);
\draw[->,brown] (d) to node[above]{\tiny$n-1$} (n-1);
\draw[->,blue] (n-1) to node[above]{\tiny$n$} (n);
\end{tikzpicture}\\
B_n: &
\begin{tikzpicture}[xscale=1.3,baseline=-4]
\node (1) at (0,0) {$\young(1)$};
\node (d1) at (1.5,0) {$\cdots$};
\node (n) at (3,0) {$\young(n)$};
\node (0) at (4.5,0) {$\young(0)$};
\node (bn) at (6,0) {$\young(\on)$};
\node (d2) at (7.5,0) {$\cdots$};
\node (b1) at (9,0) {$\young(\one)$};
\draw[->,red] (1) to node[above]{\tiny$1$} (d1);
\draw[->,brown] (d1) to node[above]{\tiny$n-1$} (n);
\draw[->,blue] (n) to node[above]{\tiny$n$} (0);
\draw[->,blue] (0) to node[above]{\tiny$n$} (bn);
\draw[->,brown] (bn) to node[above]{\tiny$n-1$} (d2);
\draw[->,red] (d2) to node[above]{\tiny$1$} (b1);
\end{tikzpicture}
\\[10pt]
C_n: &
\begin{tikzpicture}[xscale=1.3,baseline=-4]
\node (1) at (0,0) {$\young(1)$};
\node (d1) at (1.8,0) {$\cdots$};
\node (n) at (3.6,0) {$\young(n)$};
\node (bn) at (5.4,0) {$\young(\on)$};
\node (d2) at (7.2,0) {$\cdots$};
\node (b1) at (9,0) {$\young(\one)$};
\draw[->,red] (1) to node[above]{\tiny$1$} (d1);
\draw[->,brown] (d1) to node[above]{\tiny$n-1$} (n);
\draw[->,blue] (n) to node[above]{\tiny$n$} (bn);
\draw[->,brown] (bn) to node[above]{\tiny$n-1$} (d2);
\draw[->,red] (d2) to node[above]{\tiny$1$} (b1);
\end{tikzpicture}
\\[10pt]
D_{n+1}: &
\begin{tikzpicture}[xscale=1.3,baseline=-4]
\node (1) at (0,0) {$\young(1)$};
\node (d1) at (1.5,0) {$\cdots$};
\node (n-1) at (3,0) {$\boxed{n-1}$};
\node (n) at (4.5,.75) {$\young(n)$};
\node (bn) at (4.5,-.75) {$\young(\on)$};
\node (bn-1) at (6,0) {$\boxed{\overline{n-1}}$};
\node (d2) at (7.5,0) {$\cdots$};
\node (b1) at (9,0) {$\young(\one)$};
\draw[->,red] (1) to node[above]{\tiny$1$} (d1);
\draw[->,magenta] (d1) to node[above]{\tiny$n-2$} (n-1);
\draw[->,brown] (n-1) to node[above,sloped]{\tiny$n-1$} (n);
\draw[->,blue] (n-1) to node[below,sloped]{\tiny$n$} (bn);
\draw[->,blue] (n) to node[above,sloped]{\tiny$n$} (bn-1);
\draw[->,brown] (bn) to node[below,sloped]{\tiny$n-1$} (bn-1);
\draw[->,magenta] (bn-1) to node[above]{\tiny$n-2$} (d2);
\draw[->,red] (d2) to node[above]{\tiny$1$} (b1);
\end{tikzpicture}
%\\[30pt]
%G_2: &
%\begin{tikzpicture}[xscale=1.3,baseline=-4]
%\node (1) at (0,0) {$\young(1)$};
%\node (2) at (1.5,0) {$\young(2)$};
%\node (3) at (3,0) {$\young(3)$};
%\node (0) at (4.5,0) {$\young(0)$};
%\node (b3) at (6,0) {$\young(\othree)$};
%\node (b2) at (7.5,0) {$\young(\otwo)$};
%\node (b1) at (9,0) {$\young(\one)$};
%\path[->,font=\tiny]
% (1) edge node[above]{$1$} (2)
% (2) edge node[above]{$2$} (3)
% (3) edge node[above]{$1$} (0)
% (0) edge node[above]{$1$} (b3)
% (b3) edge node[above]{$2$} (b2)
% (b2) edge node[above]{$1$} (b1);
%\end{tikzpicture}
\\\hline
\end{array}
\]
\caption{Crystals of the vector representation $B(\omega_1)$ of classical types.}
\label{fig:vec_repr}
\end{figure}
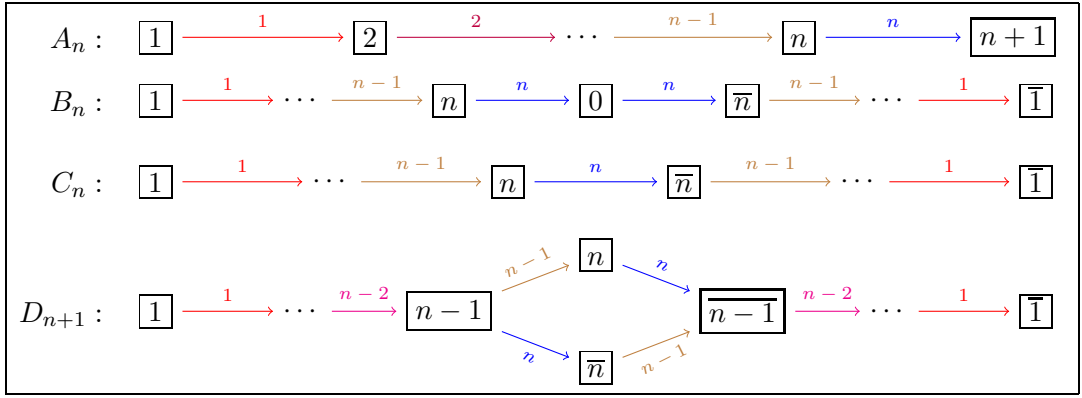

% ========
\subsection{Tableaux for finite-dimensional representations}

We use English convention for our partitions and tableaux.

We recall the definition of \defn{Kashiwara--Nakashima (KN) tableaux}~\cite{KN94}, which give a model for highest weight crystals in types $A_n$, $B_n$, $C_n$, and $D_n$. We first identify $\lambda \in P^+$ with a partition by having a height $h$ column correspond to $\tfw_h$. Additionally, for type $B_n$ and $D_n$, we allow a half width spin column corresponding to an odd coefficient of $\fw_n$. For type $D_n$, we also allow the $n$-th row to have negative length. In other words, we generally have $\langle \alpha_i, \lambda \rangle$ being the multiplicity of $i$ in the transpose of the partition associated to $\lambda$. Furthermore, if $\lambda = \sum_{i=1}^n \lambda_i \epsilon_i$, then the partition corresponding to $\lambda$ is given by $(\lambda_1, \lambda_2, \dotsc, \lambda_n)$.

Now consider the crystal of the vector representation, which is given in Figure~\ref{fig:vec_repr} and corresponds to $B(\fw_1)$. When $\lambda$ does not contain a half width column, we construct the elements of $B(\lambda)$ as semistandard tableaux of shape $\lambda$ with letters in $B(\fw_1)$,\footnote{We consider $B(\fw_1)$ as the Hasse diagram of a poset with \boxed{$1$} being the smallest element.} which we identity with an element in $B(\fw_1)^{\otimes |\lambda|}$ by taking the reverse Far-Eastern reading word and use this to define the crystal structure. In particular, we consider the closure of the tableaux whose $i$-th row is filled with $i$ except for type $D_n$ when the $n$-th row has negative length, where we fill it with $\on$.

This does not generate all highest weight crystals in types $B_n$ and $D_n$. To do this, we need the spin representation $B(\fw_n)$, as well as $B(\fw_{n-1})$ for type $D_n$, which are given by elements in $\{+, -\}^n$. In type $D_n$, for $(s_1, \dotsc, s_n) \in B(\fw_k)$, we require that $\prod_{i=1}^n s_i = -,+$ if $k = n-1,n$ respectively. We define the crystal structure by
\begin{align*}
e_i(s_1, \dotsc, s_n) & = \begin{cases}
(\dotsc, s_{i-1}, +, -, s_{i+2}, \dotsc) & \text{if $i < n$ and } (s_i, s_{i+1}) = (-, +), \\
(\dotsc, s_{n-1}, +) & \text{if $i = n$, type $B_n$ and } s_n = -, \\
(\dotsc, s_{n-2}, +, +) & \text{if $i = n$, type $D_n$ and } (s_{n-1}, s_n) = (-, -), \\
0 & \text{otherwise},
\end{cases}
\\ f_i(s_1, \dotsc, s_n) & = \begin{cases}
(\dotsc, s_{i-1}, -, +, s_{i+2}, \dotsc) & \text{if $i < n$ and } (s_i, s_{i+1}) = (+, -), \\
(\dotsc, s_{n-1}, -) & \text{if $i = n$, type $B_n$ and } s_n = +, \\
(\dotsc, s_{n-2}, -, -) & \text{if $i = n$, type $D_n$ and } (s_{n-1}, s_n) = (+, +), \\
0 & \text{otherwise}.
\end{cases}
\\ \wt(s_1, \dotsc, s_n) & = \frac{1}{2} \left( s_1 \epsilon_1 + s_2 \epsilon_2 + \cdots + s_n \epsilon_n \right),
\end{align*}
We realize the spin elements as a tableaux whose shape is a half-width column of height $n$ and $B(\fw_{n-1})$ in type $D_n$ as having a negative half-width box at height $n$. Note that this is consistent with the identification of $P_{\Z}^+$ with partitions given above. Following English convention for tableaux, the entry in the $i$-th row counted from the top of the corresponding tableau is $s_i$ for an element $(s_1, \dotsc, s_n)$.

The decompositions for $\g$ of type $B_n$
\begin{equation}
\label{eq:2tens_B}
B(\fw_n)^{\otimes 2} \iso \bigoplus_{i=0}^{n} B(\tfw_i)
\end{equation}
and for $\g$ of type $D_{n+1}$
\begin{subequations}
\label{eq:2tens}
\begin{align}
\label{eq:2tens_Dpp}
B(\fw_{n+1})^{\otimes 2} & \iso \bigoplus_{i=0}^{\lfloor (n+1)/2 \rfloor} B(\tfw_{n+1-2i}),
\\ \label{eq:2tens_Dmm}
B(\fw_n)^{\otimes 2} & \iso B(2\fw_n) \oplus \bigoplus_{i=0}^{\lfloor (n-1)/2 \rfloor} B(\tfw_{n-2i}),
\\ \label{eq:2tens_Dpm}
B(\fw_{n+1}) \otimes B(\fw_{n}) \iso B(\fw_{n}) \otimes B(\fw_{n+1}) & \iso \bigoplus_{i=0}^{\lfloor n/2 \rfloor} B(\tfw_{n-2i}),
\end{align}
\end{subequations}
are well-known and an easy computation from the signature rule with the spin representations.

However, different sets of tableaux have been used to count the weight space decomposition of $V(\lambda)$.
In the sequel, we will use tableaux equivalent to those given by King~\cite{King76}. A \defn{King tableau} is a semistandard tableau under the alphabet
\[
1 \prec \one \prec 2 \prec \otwo \prec \cdots \prec n \prec \on
\]
such that the minimum entry in row $k$ is $k$. Note that this equivalent to forbidding $\ok$ in row $k+1$ (as the semistandard condition then prohibits anything else).
Let $\wt(T)$ is defined as for KN tableaux ($\ok$ contributes $-\epsilon_k$).
We note that our tableaux are equivalent to the tableaux in~\cite{King76} by interchanging $k \leftrightarrow \ok$ and using the order $\one \prec' 1 \prec' \cdots \prec' n$. Let $K(\lambda)$ denote the set of King tableaux of shape $\lambda$.

\begin{theorem}[{\cite{King76}}]
\label{thm:king_decomposition}
Let $\g$ be of type $C_n$ and $\lambda \in P_{\Z}^+$. We have
\[
\ch(\lambda) = \sum_{T \in K(\lambda)} x^{\wt(T)}.
\]
\end{theorem}

% ========
\subsection{Principal specializations}

We want to consider the \defn{principal specialization} of a type $B_n$, $C_n$, or $D_n$ character $\chi(x_1^{\pm 1}, \dotsc, x_n^{\pm 1})$, which is defined as
\[
\chi(q, q^2, \dotsc, q^n) \seteq \chi(x_1^{\pm 1}, \dotsc, x_n^{\pm 1}) \bigr|_{x_1^{\pm 1}=q^{\pm 1},\ldots,x_n^{\pm 1}=q^{\pm n}}.
\]
For type $A_n$, the principal specialization is given by $\chi(1,q,q^2,\dotsc,q^{n-1})$.
We write $\ps(\lambda)$ for the principal specialization of the character of $V(\lambda)$.
Let $\nps(\lambda)$ denote \defn{normalized principal specialization}, where we define
\[
\nps(\lambda) \seteq q^{-\eta_{\lambda}} \ps(\lambda),
\]
where $\eta_{\lambda}$ is the valuation of $\ps(\lambda)$, or the lowest power of $q$ appearing in $\ps(\lambda)$.
Note that $\nps(\lambda)$ is a polynomial in $q$ with a non-zero constant term.

We recall some descriptions of the principal specializations that follow from the Weyl character formula; we follow~\cite{BKW16}. Recall the identification between dominant weights and partitions from Section~\ref{sec:crystals}. Hence, we denote $|\lambda| = \sum_{i \geq 1} \lambda_i$ and $n(\lambda) = \sum_{i \geq 1} (i-1)\lambda_i$.

\begin{proposition}[{\cite[Eq.~(3.21)]{BKW16}}]
\label{prop:Bn_ps}
In type $B_n$, we have
\[
\ps(\lambda) = q^{n(\lambda) - n|\lambda|} \prod_{i=1}^n \dfrac{1-q^{2\lambda_i+2n-2i+1}}{1-q^{2n-2i+1}} \prod_{1 \leq i < j \leq n} \dfrac{1-q^{\lambda_i-\lambda_j+j-i}}{1-q^{j-i}} \cdot \dfrac{1-q^{\lambda_i+\lambda_j+2n-i-j+1}}{1-q^{2n-i-j+1}}.
\]
\end{proposition}

\begin{proposition}[{\cite[Eq.~(3.30a)]{BKW16}}]
\label{prop:Cn_ps}
In type $C_n$, we have
\[
\ps(\lambda) = q^{n(\lambda) - n|\lambda|} \prod_{i=1}^n \dfrac{1-q^{2(\lambda_i+n-i+1)}}{1-q^{2(n-i+1)}} \prod_{1 \leq i < j \leq n} \dfrac{1-q^{\lambda_i-\lambda_j+j-i}}{1-q^{j-i}} \cdot \dfrac{1-q^{\lambda_i+\lambda_j+2n-i-j+2}}{1-q^{2n-i-j+2}}.
\]
\end{proposition}

% ========
\subsection{$q$-analogs and $q,t$-Catalan numbers}

We recall the standard $q$-analogs of integers, factorials, binomial coefficients, and $q$-Pochhammer symbols. The $q$-analog of an integer $n$ by
\[
[n]_q \seteq 1 + q + \cdots + q^{n-1} = \frac{1-q^n}{1-q},
\]
factorials by $[n]!_q \seteq [1]_q \cdots [n]_q$, and binomial coefficients by
\[
\qbinom{n}{k}{q} \seteq \frac{[n]!_q}{[k]!_q [n-k]!_q}.
\]
We recall that $q$-binomial coefficients are positive polynomials in $q$, \textit{i.e.}, $\qbinom{n}{k}{q} \in \Z_{\geq 0}[q]$, and there are two versions of the Pascal's identity:
\begin{subequations}
\label{eq:q_pascal}
\begin{align}
\label{eq:pascal1} \qbinom{n+m}{n}{q} & = q^n \qbinom{n+m-1}{n}{q} + \qbinom{n+m-1}{n-1}{q}
\\ \label{eq:pascal2} & = \qbinom{n+m-1}{n}{q} + q^m \qbinom{n+m-1}{n-1}{q}.
\end{align}
\end{subequations}
The $q$-Pochhammer symbol is defined as
\[
(a; q)_n \seteq \prod_{k=1}^n (1 - aq^{k-1}).
\]
We will use the shorthand $(q)_n \seteq (q;q)_n$.

MacMahon defined a $q$-definition of the Catalan numbers, which are known as the \defn{Mahonian $q$-Catalan numbers}, during his study of permutations~\cite{MacMahon15}. The Mahonian $q$-Catalan numbers are defined by the natural $q$-analog of the closed form of the Catalan numbers:
\[
\Cat_n(q) \seteq \frac{1}{[n+1]_q} \qbinom{2n}{n}{q}.
\]
However, the Mahonian $q$-Catalan numbers are not known to satisfy a $q$-analog of the Catalan recursion relation~\eqref{eq:Catalan_recurrence}.

We also have another $q$-analog of the Catalan numbers that satisfy a $q$-analog of~\eqref{eq:Catalan_recurrence}:
\[
\CatCR_{n+1}(q) = \sum_{k=0}^n q^k \CatCR_k(q) \CatCR_{n-k}(q),
\qquad
\qquad
\CatCR_0(q) = 1.
\]
The values $\CatCR_n(q)$ are called the \defn{Carlitz--Riordan $q$-Catalan numbers}~\cite{CR64}, and there is no known simple formula closed for $\CatCR_n(q)$. However, there does exist the following combinatorial interpretation of $\CatCR_n(q)$.
The \defn{area} of a Dyck path $a \colon \Dyck_n \to \Z$ is given by $a(D) = \sum_{i=1}^n j_i$, where $(j_i, i)$ is the endpoint after the $i$-th North step. It is a standard combinatorial exercise to show
\[
\CatCR_n(q) = \sum_{D \in \Dyck_n} q^{a(D)}.
\]

Another statistic called \defn{bounce} $b \colon \Dyck_n \to \Z$ is used to define the \defn{$q,t$-Catalan numbers}:
\[
\CatCR_n(q,t) = \sum_{D \in \Dyck_n} q^{a(D)} t^{b(D)}.
\]
We do not need the definition of bounce and instead refer the reader to~\cite{GH01,GH02}. Furthermore, Garsia and Haglund showed that $\CatCR_n(q,t) = \CatCR_n(t,q)$ by showing $\CatCR_n(q,t)$ is the bi-graded Hilbert series of the alternating elements subspace of the space of diagonal harmonics~\cite{GH01,GH02}. It is still an open problem to prove combinatorially that $\CatCR_n(q,t) = \CatCR_n(t,q)$.

Another $q,t$-Catalan number was introduced by Stump~\cite{Stump08} by using a refinement of the major index statistic given by F\"urlinger and Hofbauer~\cite[Sec.~5]{FH85} (see also~\cite[St000027,St000947,St001161]{FindStat}). Let $D$ be a Dyck path. Denote $\Des(D) \seteq \{i \mid D_i = N, D_{i+1} = E \}$ be the \defn{descent set} of $D = D_1 \cdots D_{2n}$, which corresponds to the positions of the valleys of $D$. Define statistics
% 0 <-> E and 1 <-> N for [Stump] <-> [us]
\[
\maj_N(D) \seteq \sum_{i \in \Des(D)} \lvert \{ j \leq i \mid D_j = N \} \rvert,
\qquad
\maj_E(D) \seteq \sum_{i \in \Des(D)} \lvert \{ j \leq i \mid D_j = E \} \rvert,
\]
and another $q,t$-analog of the Catalan numbers by
\[
\Cat_n(q,t) = \sum_{D \in \Dyck_n} q^{\maj_N(D)} \, t^{\binom{n}{2} - \maj_E(D)}.
\]
Note that the usual major index is precisely $\maj_N(D) + \maj_E(D)$, which yields the following.

\begin{proposition}[{\cite{Stump08}}]
\label{prop:specialized_stump}
We have
\[
q^{\binom{n}{2}} \Cat_n(q,q^{-1}) = \Cat_n(q).
\]
\end{proposition}

We will also consider two forms of $q$-Catalan triangle numbers.
The first is a natural $q$-version of the Catalan triangle numbers:
\[
\Cat_{(n,k)}(q) \seteq \frac{[n+k]_q! [n-k+1]_q}{[k]_q! [n+1]_q!} = q^{-k} \left( \qbinom{n+k}{k}{q} - \qbinom{n+k}{k-1}{q} \right).
\]
Note the right equality follows from
% Full computation
%\begin{align*}
%\qbinom{n+k}{k}{q} - \qbinom{n+k}{k-1}{q} & = \frac{[n+k]_q !}{[k]_q! [n]_q!} - \frac{[n+k]_q !}{[k-1]_q! [n+1]_q!}
%\\ & = \frac{[n+k]_q! [n+1]_q}{[k]_q! [n+1]_q!} - \frac{[n+k]_q! [k]_q}{[k]_q! [n+1]_q!}
%\\ & = \frac{[n+k]_q! [n+1]_q - [n+k]_q! [k]_q}{[k]_q! [n+1]_q!}
%\\ & = \frac{[n+k]_q! ([n+1]_q -[k]_q)}{[k]_q! [n+1]_q!}
%\end{align*}
\[
\qbinom{n+k}{k}{q} - \qbinom{n+k}{k-1}{q} = \frac{[n+k]_q! ([n+1]_q -[k]_q)}{[k]_q! [n+1]_q!}
\]
and that $[n+1]_q - [k]_q = q^k [n-k+1]_q$.

\begin{definition}
We define a new $q$-analog of the Catalan triangle numbers by $\Cat'_{(2n-i+1,i)}(q) \seteq \nps(\fw_i)$ in type $C_n$. 
\end{definition}

Note that $\Cat'_{(2n-i+1,i)}(1) = \Cat_{(2n-i+1,i)} = \dim V(\fw_i)$.
Additionally, note that $\Cat'_{(2n-i+1,i)}(q) \neq \Cat_{(2n-i+1,i)}(q)$ in general.
For example, we have
\begin{align*}
q^7 \Cat'_{(7,2)}(q) & = q^{14} + q^{13} + 2 q^{12} + q^{11} + 2 q^{10} + 2 q^9
\\ & \hspace{20pt} + 3 q^8 + 3 q^7 + 3 q^6 + 2 q^5 + 2 q^4 + q^3 + 2 q^2 + q + 1,
\\ \Cat_{(7,2)}(q) & = q^{12} + q^{11} + 2 q^{10} + 2 q^9 + 3 q^8 + 3 q^7
\\ & \hspace{20pt} + 3 q^6 + 3 q^5 + 3 q^4 + 2 q^3 + 2 q^2 + q + 1.
\end{align*}
However, they are related by a simple ratio.

\begin{proposition}
\label{prop:ratio}
For all $1 \leq i \leq n-1$, we have
\[
%\frac{\Cat_{(2n-i+1,i)}(q)}{\CatCR_{(2n-i+1,i)}(q)} = q^{\binom{n+1}{2} - \binom{n+1-i}{2}} \frac{q^{n+1-i} + 1}{q^{n+1}+1}.
\frac{\Cat_{(2n-i-1,i)}(q)}{\Cat'_{(2n-i-1,i)}(q)} = q^{\binom{n}{2} - \binom{n-i}{2}} \frac{q^{n-i} + 1}{q^{n}+1}.
\]
\end{proposition}

\begin{proof}
Use the single column in type $C_{n-1}$ given in, \textit{e.g.}, Proposition~\ref{prop:Cn_ps}.
Then it is straightforward computation.
\end{proof}

Based on Proposition~\ref{prop:ratio}, we can give the more general definition of a modified $q$-Catalan triangle number
\[
\Cat^{\dagger}_{(n,k)}(q) \seteq \frac{q^{\lfloor (n-k+1)/2 \rfloor} + 1}{q^{\lfloor (n+k+1)/2 \rfloor} + 1} \Cat_{(n,k)}(q).
\]
Note that $\Cat^{\dagger}_{(2n-i-1,i)}(q) = q^{\binom{n}{2} - \binom{n-i}{2}} \Cat'_{(2n-i-1,i)}(q)$. We will not use $\Cat^{\dagger}_{(n,k)}(q)$ in this paper.

%%%%%%%%%%%%%%%%%%%%%%%%%%%%%%%%%%%%%%%%
\section{Jacobi--Trudi-type formulas for type $C_n$ characters}
\label{sec:JT_type_C}

In this section, we give a determinant formula for type $C_n$ characters. To do so, we use the King tableaux description of the weight decomposition given by Theorem~\ref{thm:king_decomposition}. We start by recalling the bijection $\Xi_n \colon \Dyck_n \to K(\fw_{n-1})$ in type $C_{n-1}$, given by
\begin{enumerate}
\item\label{bij_step1} getting the complement partition,
\item\label{bij_step2} adding the staircase shape $\rho = n \dotsm 321$, and
\item\label{bij_step3} replacing $1, 2, \dotsc, 2n$ with $\on, n, \dotsc, 2, \one, 1$ (in that order) in a (sorted) column.
\end{enumerate}
The bijection $\Xi_n$ is a simple modification of the bijection given in~\cite[Exercise~6.19(t)]{ECII} (equivalently~\cite[Item~79]{Stanley15}), which shows that $\lvert B(\fw_{n-1}) \rvert = \dim V(\fw_{n-1}) = \Cat_n$.

\begin{example}
\label{ex:dyck_tableau_bij}
Let $n = 5$. The steps under $\Xi_5$ to go from the Dyck path $EENENNEENN$ to a King tableaux are:
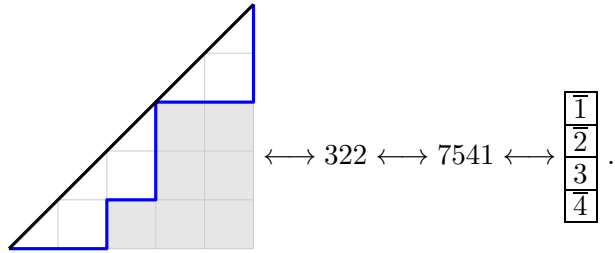
\[
\begin{tikzpicture}[scale=0.65, baseline=32]
\fill[gray!20] (3,3) rectangle (5,0);
\fill[gray!20] (2,0) rectangle (3,1);
\dyckgrid{5}
% Draw the path
\draw[-, blue, very thick, line join=round] (0,0) -- (2,0) -- (2,1) -- (3,1) -- (3,3) -- (5,3) -- (5,5);
% Draw stuff over it
\draw[-,very thick] (0,0) -- (5,5);
\end{tikzpicture}
\longleftrightarrow
322
\longleftrightarrow
7541
\longleftrightarrow
\young(\one,\otwo,3,\ofour)\ .
\]
\end{example}

By unrolling the bijection $\Xi_{n+1}$, we can construct a weighting $w(N;D)$ on the first $n$ steps $N$ of $D$ such that
\[
\sum_{D \in \Dyck_{n+1}} \prod_N x_{w(N;D)} = \ch(\fw_n)
\]
via the bijection $\Xi_{n+1}$, where we write $x_{\overline{\imath}} \seteq x_i^{-1}$.\footnote{Intrinsically, the weighting is $w(N;D) = x_i^{\pm 1}$, but we will use the (barred) integer valued weighting in the sequel when we take the principal specialization.} Indeed, we define $w(N;D)$ by
\[
w(N, D) \seteq \begin{cases}
(N_X + N_Y + 1) / 2 & \text{if $N_X + N_Y$ odd}, \\
\overline{(N_X + N_Y) / 2} & \text{if $N_X + N_Y$ even},
\end{cases}
\]
where the initial position of $N$ in $D$ is $(N_X, N_Y)$. We denote the product by $x^{\wt(D)} \seteq \prod_N x_{w(N;D)}$.
% We call a vertical step \defn{positive} (resp.\ \defn{negative}) if $N_x + N_y$ is odd (resp.\ even).

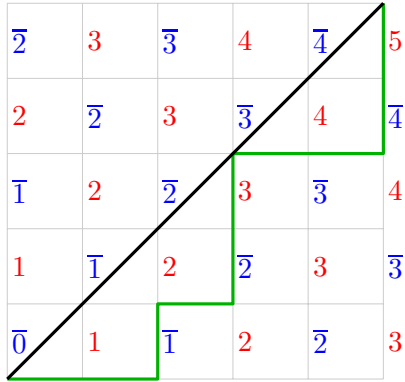
\begin{figure}
\begin{tikzpicture}
\draw[gray!40,very thin] (0,0) grid (5,5);
\draw (0,0.5) node[right=-2pt,blue] {$\overline{0}$};
\foreach \i in {0,...,2} \draw (2-\i,\i+0.5) node[right=-2pt,blue] {$\one$};
\foreach \i in {0,...,4} \draw (4-\i,\i+0.5) node[right=-2pt,blue] {$\otwo$};
\foreach \i in {0,...,3} \draw (5-\i,\i+1.5) node[right=-2pt,blue] {$\othree$};
\foreach \i in {0,...,1} \draw (5-\i,\i+3.5) node[right=-2pt,blue] {$\ofour$};
\foreach \i in {0,1} \draw (1-\i,\i+0.5) node[right=-2pt,red] {$1$};
\foreach \i in {0,...,3} \draw (3-\i,\i+0.5) node[right=-2pt,red] {$2$};
\foreach \i in {0,...,4} \draw (5-\i,\i+0.5) node[right=-2pt,red] {$3$};
\foreach \i in {0,...,2} \draw (5-\i,\i+2.5) node[right=-2pt,red] {$4$};
\foreach \i in {0} \draw (5-\i,\i+4.5) node[right=-2pt,red] {$5$};
\draw[-, very thick, black!30!green,line join=round] (0,0) -- (2,0) -- (2,1) -- (3,1) -- (3,3) -- (5,3) -- (5,5);
\draw[-, very thick] (0,0) -- (5,5);
\end{tikzpicture}
\caption{An example of the weighting $w$ extended to a full $5 \times 5$ grid and the lattice path from Example~\ref{ex:dyck_tableau_bij}.}
\label{fig:Dyck_statistic}
\end{figure}

\begin{remark}
\label{rem:counting_NE}
The values $N_X$ and $N_Y$ are the number of $E$ and $N$ steps respectively occurring before the particular $N$ step in the Dyck word.
\end{remark}

We will also require the weighting given by taking the conjugate of the partition obtained from Step~(\ref{bij_step1}) before doing Step~(\ref{bij_step2}). We denote the resulting bijection by $\Xi'_n$. Thus we obtain a weighting on the horizontal steps, which is given explicitly as
\[
w'(E;D) \seteq \begin{cases}
\overline{n - (E_X + E_Y + 1) / 2} & \text{if $E_X + E_Y$ odd}, \\
n - (E_X + E_Y) / 2 & \text{if $E_Y + E_Y$ even},
\end{cases}
\]
where the sum is over the last $n$ steps $E$ in $D$ and the initial position of $E$ is $(E_X, E_Y)$.
Similarly, we denote $x^{\wt'(D)} \seteq \prod_E x_{w'(E;D)}$.
% We call a horizontal step \defn{positive} (resp.\ \defn{negative}) if $E_X + E_Y$ is even (resp.\ odd).

\begin{example}
\label{ex:dyck_tableau_bij_transpose}
Using the Dyck path $D$ from Example~\ref{ex:dyck_tableau_bij}, applying $\Xi'_n$ yields
\[
D
\longleftrightarrow
331
\longleftrightarrow
7631
\longleftrightarrow
\young(\one,2,\othree,\ofour)\ .
\]
\end{example}

\begin{figure}
\begin{tikzpicture}[scale=1]
\draw[gray!40,very thin] (0,0) grid (5,5);
\draw (4.5,5) node[above=-3pt,blue] {$\overline{0}$};
\foreach \i in {0,...,2} \draw (2.5+\i,5-\i) node[above=-3pt,blue] {$\one$};
\foreach \i in {0,...,4} \draw (0.5+\i,5-\i) node[above=-3pt,blue] {$\otwo$};
\foreach \i in {0,...,3} \draw (0.5+\i,3-\i) node[above=-3pt,blue] {$\othree$};
\foreach \i in {0,...,1} \draw (0.5+\i,1-\i) node[above=-3pt,blue] {$\ofour$};
\foreach \i in {0,1} \draw (3.5+\i,5-\i) node[above=-3pt,red] {$1$};
\foreach \i in {0,...,3} \draw (1.5+\i,5-\i) node[above=-3pt,red] {$2$};
\foreach \i in {0,...,4} \draw (0.5+\i,4-\i) node[above=-3pt,red] {$3$};
\foreach \i in {0,...,2} \draw (0.5+\i,2-\i) node[above=-3pt,red] {$4$};
\draw (0.5,0) node[above=-3pt,red] {$5$};
\draw[-, very thick, black!30!green, line join=round] (0,0) -- (2,0) -- (2,1) -- (3,1) -- (3,3) -- (5,3) -- (5,5);
\draw[-, very thick] (0,0) -- (5,5);
\end{tikzpicture}
\caption{An example of the weighting $w'$ extended to a full $5 \times 5$ grid and the lattice path from Example~\ref{ex:dyck_tableau_bij_transpose}.}
\label{fig:Dyck_statistic_transpose}
\end{figure}
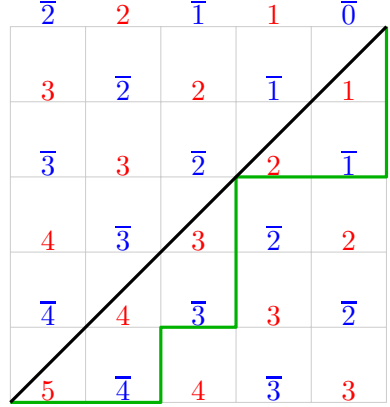

\begin{remark}
\label{rem:removed_fixed}
The last $N$ and the first $E$ step are fixed for any Dyck path. Hence, they would only contribute a constant factor of $x_{n+1}$ in $\sum_{D \in \Dyck_{n+1}} x^{\wt(D)}$ and $\sum_{D \in \Dyck_{n+1}} x^{\wt'(D)}$ if they were included in the computation of $x^{\wt(D)}$ and $x^{\wt'(D)}$. However $x_{n+1}$ does not occur the corresponding King tableaux, so we do not include them.
\end{remark}

It is straightforward to extend the bijection $\Xi_{n+1}$ to a bijection $\Xi_{(2n-i+1,i)} \colon \Dyck_{(2n-i+1,i)} \to K(\fw_i)$ in type $C_n$.
We explain the extension using our weighting on $N$ steps.
For a fixed partial Dyck word $D \in \Dyck_{(2n-i+1,i)}$ starting at $(0,0)$, we consider the $N$-steps $N^{(1)}, \dotsc, N^{(i)}$ and write the column
\[
\Xi_{(2n-i+1,i)}(D) = \begin{array}{|c|c|c|c|c|}
\hline
w(N^{(1)};D) & w(N^{(2)};D) & \cdots & w(N^{(i-1)};D) & w(N^{(i)};D)
\\\hline
\end{array}^{\,T}.
\]

Next, let $\Dyck'_{(i,2n-i+1)}$ denote the set of words with $i$ occurrences of the letter $E$ and $2n-i+1$ occurrences of the letter $N$ such that reversing the word and interchanging $E \longleftrightarrow N$ results in a partial Dyck word in $\Dyck_{(2n-i+1,i)}$.
We call the set $\Dyck'_{i,2n-i+1}$ the set of conjugate partial Dyck words.
We note that we can similarly extend $\Xi_{n+1}$ to a bijection $\Xi'_{(i,2n-i+1)} \colon \Dyck'_{(i,2n-i+1)} \to K(\fw_i)$.
For a fixed conjugate partial Dyck word $D' \in \Dyck'_{(i,2n-i+1)}$ ending at $(n,n)$, we consider the $E$-steps $E^{(1)}, \dotsc, E^{(i)}$ and write the column
\[
\Xi'_{(i,2n-i+1)}(D') = \begin{array}{|c|c|c|c|c|}
\hline
w'(E^{(i)};D') & w'(E^{(i-1)};D') & \cdots & w'(E^{(2)};D') & w'(E^{(1)};D')
\\\hline
\end{array}^{\,T}.
\]

\begin{proposition}
\label{prop:fundamental_char}
Fix some positive integer $n$. We have
\[
x^{\wt(D)} = x^{\wt\bigl( \Xi_{(2n-i+1,i)}(D) \bigr)},
\qquad\qquad
x^{\wt'(D')} = x^{\wt\bigl( \Xi'_{(i,2n-i+1)}(D') \bigr)},
\]
for any $D \in \Dyck_{(2n-i+1,i)}$ and $D' \in \Dyck'_{i,2n-i+1}$.
Moreover, in type $C_n$ we have
\[
\ch(\fw_i) = \sum_{D \in \Dyck_{(2n-i+1,i)}} x^{\wt(D)} = \sum_{D' \in \Dyck'_{(i,2n-i+1)}} x^{\wt'(D')}
\]
\end{proposition}

\begin{proof}
This is immediate given the construction of the bijections $\Xi_{2n-i+1,i}$ and $\Xi'_{i,2n-i+1}$ and from Theorem~\ref{thm:king_decomposition}
%We only consider $w$ as the case for $w'$ is similar.
%Consider an $N$ step with initial position $(N_X, N_Y)$. This contributes an $n - N_X$ to the partition under Step~(\ref{bij_step1}) of $\Xi_{1n-i+1,i}$, which then becomes $2n - 1 - N_X - N_Y$ under Step~(\ref{bij_step2}). Under Step~(\ref{bij_step3}), we obtain the letter $(N_X + N_Y + 1) / 2$ if $N_X + N_Y$ is odd and $\overline{(N_X + N_Y) / 2}$ otherwise. Hence, the first claim follows from the definition of weight of a King tableau.
%The second claim follows from Theorem~\ref{thm:king_decomposition}.
\end{proof}

\begin{remark}
\label{rem:catalan_extension}
We can replace both $\Dyck_{n+1,n}$ and $\Dyck'_{n,n+1}$ with $\Dyck_{n+1}$ to obtain
\[
\ch(\fw_n) = \sum_{D \in \Dyck_{n+1}} x^{\wt(D)} = \sum_{D \in \Dyck_{n+1}} x^{\wt'(D)}.
\]
We note that going from $\Dyck_{n+1}$ to $\Dyck_{n+1,n}$ (resp.~$\Dyck'_{n,n+1}$) removes the last $N$ step (resp.\ first $E$ step) as per Remark~\ref{rem:removed_fixed}.
\end{remark}

\begin{remark}
\label{remark:weight_change}
Since our weighting $w$ is encoding the character of a representation, we can extend the Weyl group action on the character to form a new weighting $\widetilde{w}$. Because the characters are Weyl group invariant, we have
\[
\sum_{D \in \Dyck_{n+1}} \prod_N x_{w(N;D)} = \sum_{D \in \Dyck_{n+1}} \prod_N x_{\widetilde{w}(N;D)}.
\]
A similar statement holds for $w'$.
\end{remark}

We can extend this to a Jacobi--Trudi formula for type $C_n$ characters.
Denote
\[
\chi_{(2m-i+1,i)} \seteq \ch(\fw_i) \qquad \text{in type $C_m$}.
\]

\begin{theorem}
\label{thm:general_det_formula_Cn}
For type $C_n$, let $\lambda = \sum_{i \in I} c_i \fw_i \in P_{\Z}^+$, and let $(\lambda'_1, \lambda'_2, \ldots, \lambda'_{\ell})$ denote the conjugate partition of $\lambda$.
Then, we have
\[
\det \left[ \chi_{(a(i,j), b(i,j))} \right]_{i,j=1}^{\ell} =  \ch(\lambda),
\]
where
\begin{align*}
a(i,j) & = 2\ell-i-j+2n+1-\lambda'_{\ell+1-j},
\\ b(i,j) & = j-i+\lambda'_{\ell+1-j}.
\end{align*}
\end{theorem}

\begin{proof}
Note that we can rewrite $a(i,j) = 2(\ell+n-i)-b(i,j)+1$, and so $\chi_{(a(i,j), b(i,j))} = \ch( \fw_{b(i,j)} )$ in type $C_{\ell+n-i}$.
We want to apply the LGV lemma, but we note that the edge weighting of the paths from $s_i \to t_j$ do not necessarily match. However, by the Weyl group invariance of the characters, we can reweight our edges and not change $\chi_{(2m-s+1,s)} = \ch(\fw_s)$. Specifically, for a path $s_i \to t_j$, we reweight by sending $k \mapsto k - (n-i) \mod{m}$ (and so $\overline{k} \mapsto \overline{k - (n-i) \mod{m}}$).\footnote{We write the results in $\{1, \dotsc, m\}$; in particular, we write $m$ for $0$.}
Thus, this weighting is coherent for all such choices of start and endpoints, and hence we can define it as an edge weighting on the Catalan graph.

Therefore, we can now apply the LGV lemma, and the only families of non-intersecting lattice paths are of the form $\nilp = (P^{(1)}, \dotsc, P^{(\ell)})$ with $P^{(i)} \colon s_i \to t_i$. Next, form a King tableau of shape $\lambda$ by defining the $i$-th column as $\Xi_{(2n-i+1,i)}(P^{(i)})$ using the permuted weighting above. Note that the resulting entries are in $\{ 1, \one, \dotsc, n, \on \}$, and the semistandard condition of a King tableau follows from the non-intersecting properly.\footnote{This is exactly analogous to proof of the usual Jacobi--Trudi formula.} This is clearly reversible, and so the claim follows.
\end{proof}

\begin{example}
The four grids with starting and terminal points of the paths marked in the computation of $\ch(2\fw_2)$ in type $C_2$, but reweighted as in the proof of Theorem~\ref{thm:general_det_formula_Cn}:
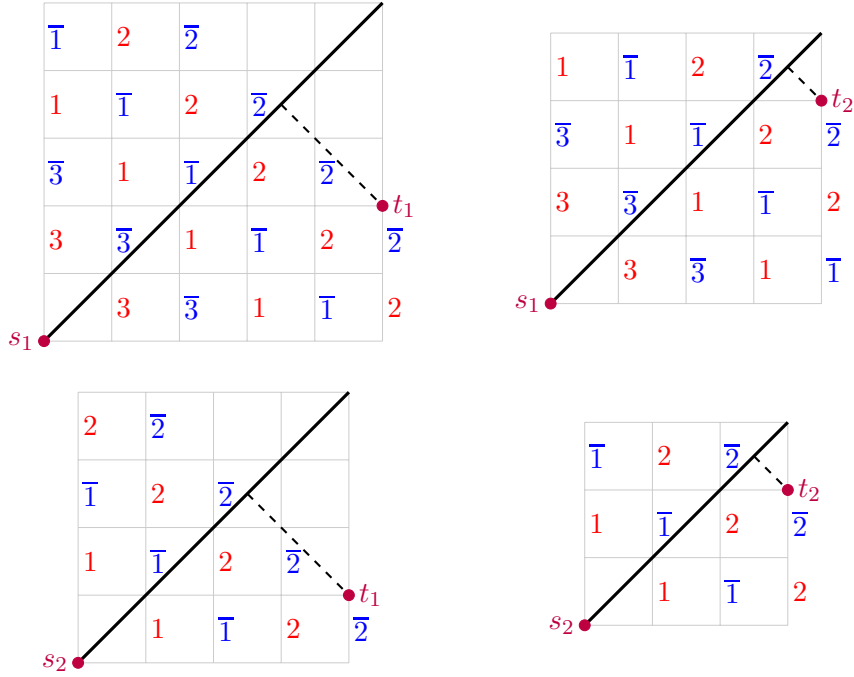
\[
\begin{array}{c@{\hspace{30pt}}c}
\begin{tikzpicture}[baseline=2.5cm,scale=0.9]
\draw[gray!40,very thin] (0,0) grid (5,5);
%\draw (0,0.5) node[right=-2pt,blue] {$\overline{0}$};
\foreach \i in {0,...,2} \draw (2-\i,\i+0.5) node[right=-2pt,blue] {$\othree$};
\foreach \i in {0,...,4} \draw (4-\i,\i+0.5) node[right=-2pt,blue] {$\one$};
\foreach \i in {0,...,3} \draw (5-\i,\i+1.5) node[right=-2pt,blue] {$\otwo$};
%\foreach \i in {0,...,1} \draw (5-\i,\i+3.5) node[right=-2pt,blue] {$\ofour$};
\foreach \i in {0,1} \draw (1-\i,\i+0.5) node[right=-2pt,red] {$3$};
\foreach \i in {0,...,3} \draw (3-\i,\i+0.5) node[right=-2pt,red] {$1$};
\foreach \i in {0,...,4} \draw (5-\i,\i+0.5) node[right=-2pt,red] {$2$};
%\foreach \i in {0,...,2} \draw (5-\i,\i+2.5) node[right=-2pt,red] {$4$};
\draw[-, very thick] (0,0) -- (5,5);
\draw[thick,dashed] (3.5,3.5) -- +(1.5, -1.5);
\draw[fill=red, color=purple] (0, 0) circle (0.08) node[anchor=east] {$s_1$};
\draw[fill=red, color=purple] (5, 2) circle (0.08) node[anchor=west] {$t_1$};
\end{tikzpicture}
&
\begin{tikzpicture}[baseline=2cm,scale=0.9]
\draw[gray!40,very thin] (0,0) grid (4,4);
%\draw (0,0.5) node[right=-2pt,blue] {$\overline{0}$};
\foreach \i in {0,...,2} \draw (2-\i,\i+0.5) node[right=-2pt,blue] {$\othree$};
\foreach \i in {0,...,3} \draw (4-\i,\i+0.5) node[right=-2pt,blue] {$\one$};
\foreach \i in {0,...,1} \draw (4-\i,\i+2.5) node[right=-2pt,blue] {$\otwo$};
\foreach \i in {0,1} \draw (1-\i,\i+0.5) node[right=-2pt,red] {$3$};
\foreach \i in {0,...,3} \draw (3-\i,\i+0.5) node[right=-2pt,red] {$1$};
\foreach \i in {0,...,2} \draw (4-\i,\i+1.5) node[right=-2pt,red] {$2$};
%\foreach \i in {0} \draw (4-\i,\i+3.5) node[right=-2pt,red] {$4$};
\draw[-, very thick] (0,0) -- (4,4);
\draw[thick,dashed] (3.5,3.5) -- +(0.5, -0.5);
\draw[fill=red, color=purple] (0, 0) circle (0.08) node[anchor=east] {$s_1$};
\draw[fill=red, color=purple] (4, 3) circle (0.08) node[anchor=west] {$t_2$};
\end{tikzpicture}
\allowdisplaybreaks \\[3cm]
\begin{tikzpicture}[baseline=2cm,scale=0.9]
\draw[gray!40,very thin] (0,0) grid (4,4);
%\draw (0,0.5) node[right=-2pt,blue] {$\overline{0}$};
\foreach \i in {0,...,2} \draw (2-\i,\i+0.5) node[right=-2pt,blue] {$\one$};
\foreach \i in {0,...,3} \draw (4-\i,\i+0.5) node[right=-2pt,blue] {$\otwo$};
%\foreach \i in {0,...,1} \draw (4-\i,\i+2.5) node[right=-2pt,blue] {$\othree$};
\foreach \i in {0,1} \draw (1-\i,\i+0.5) node[right=-2pt,red] {$1$};
\foreach \i in {0,...,3} \draw (3-\i,\i+0.5) node[right=-2pt,red] {$2$};
%\foreach \i in {0,...,2} \draw (4-\i,\i+1.5) node[right=-2pt,red] {$3$};
%\foreach \i in {0} \draw (4-\i,\i+3.5) node[right=-2pt,red] {$4$};
\draw[-, very thick] (0,0) -- (4,4);
\draw[thick,dashed] (2.5,2.5) -- +(1.5, -1.5);
\draw[fill=red, color=purple] (0, 0) circle (0.08) node[anchor=east] {$s_2$};
\draw[fill=red, color=purple] (4, 1) circle (0.08) node[anchor=west] {$t_1$};
\end{tikzpicture}
&
\begin{tikzpicture}[baseline=1.5cm,scale=0.9]
\draw[gray!40,very thin] (0,0) grid (3,3);
%\draw (0,0.5) node[right=-2pt,blue] {$\overline{0}$};
\foreach \i in {0,...,2} \draw (2-\i,\i+0.5) node[right=-2pt,blue] {$\one$};
\foreach \i in {0,...,1} \draw (3-\i,\i+1.5) node[right=-2pt,blue] {$\otwo$};
\foreach \i in {0,1} \draw (1-\i,\i+0.5) node[right=-2pt,red] {$1$};
\foreach \i in {0,...,2} \draw (3-\i,\i+0.5) node[right=-2pt,red] {$2$};
\draw[-, very thick] (0,0) -- (3,3);
\draw[thick,dashed] (2.5,2.5) -- +(0.5, -0.5);
\draw[fill=red, color=purple] (0, 0) circle (0.08) node[anchor=east] {$s_2$};
\draw[fill=red, color=purple] (3, 2) circle (0.08) node[anchor=west] {$t_2$};
\end{tikzpicture}
\end{array}
\]
\end{example}

\begin{example}
We consider $C_4$ and $\lambda = \fw_2 + 2\fw_3$.
Then the bijection between non-intersecting lattice paths and King tableaux is given by
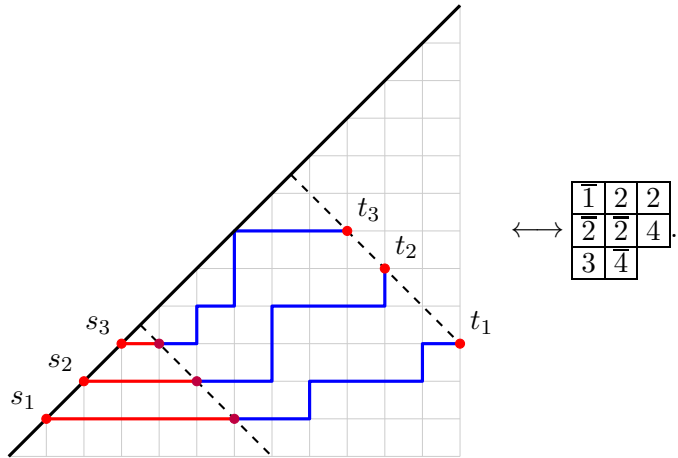
\[
\begin{tikzpicture}[scale=0.5, baseline=83]
\dyckgrid{12}
% Draw the paths
\draw[very thick, blue, line join=round] (4,3) -- (5,3) -- (5,4) -- (6,4) -- (6,6) -- (9,6);
\draw[very thick, blue, line join=round] (5,2) -- (7,2) -- (7,4) -- (10,4) -- (10,5);
\draw[very thick, blue, line join=round] (6,1) -- (8,1) -- (8,2) -- (11,2) -- (11,3) -- (12,3);
% Draw the stuff over it
\draw[very thick,-] (0,0) -- (12,12);
\draw[thick,dashed] (3.5,3.5) -- +(3.5, -3.5);
\draw[thick,dashed] (7.5,7.5) -- +(4.5, -4.5);
\foreach \i in {1,2,3}
  \draw[very thick, red] (\i, \i) -- +(7-2*\i, 0);
\foreach \i in {1,2,3}
  \draw[fill=purple, color=purple] (3+\i, 4-\i) circle (0.12);
\foreach \i in {1,2,3}
{
  \draw[fill=red, color=red] (\i, \i) circle (0.12);
  \draw (\i, \i) node[anchor=south east] {$s_{\i}$};
}
\foreach \i in {2,3,5} { \draw[fill=red, color=red] (7+\i, 8-\i) circle (0.12); }
\draw (8+1, 8-2) node[anchor=south west] {$t_3$};
\draw (8+2, 8-3) node[anchor=south west] {$t_2$};
\draw (8+4, 8-5) node[anchor=south west] {$t_1$};
\end{tikzpicture}
\longleftrightarrow
\young(\one22,\otwo\otwo4,3\ofour).
\]
\end{example}

\begin{example}
\label{ex:JT_full_rect}
We consider $C_3$ and $\lambda = 3\fw_3$.
Then the bijection between non-intersecting lattice paths and King tableaux is given by
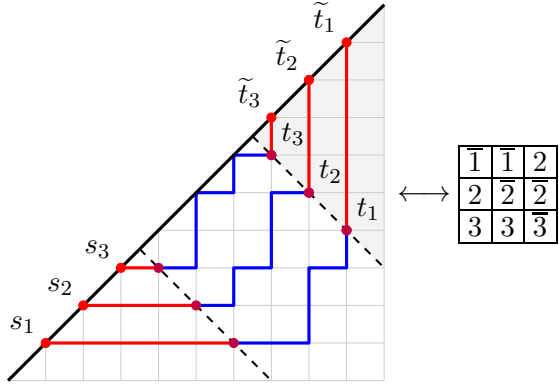
\[
\begin{tikzpicture}[scale=0.5, baseline=68]
\fill[gray!10] (6.5,6.5) -- (10,3) -- (10,10) -- cycle;
\dyckgrid{10}
% Draw the paths
\draw[very thick, blue, line join=round] (4,3) -- (5,3) -- (5, 5) -- (6,5) -- (6,6) -- (7,6);
\draw[very thick, blue, line join=round] (5,2) -- (6,2) -- (6,3) -- (7,3) -- (7,5) -- (8,5);
\draw[very thick, blue, line join=round] (6,1) -- (8,1) -- (8,3) -- (9,3) -- (9,4);
% Draw the stuff over it
\draw[very thick,-] (0,0) -- (10,10);
\foreach \i in {1,2,3}
{
  \draw[very thick, red] (\i, \i) -- +(7-2*\i, 0);
  \draw[very thick, red] (10-\i, 10-\i) -- +(0, -7+2*\i);
}
\draw[thick,dashed] (3.5,3.5) -- +(3.5, -3.5);
\draw[thick,dashed] (6.5,6.5) -- +(3.5, -3.5);
\foreach \i in {1,2,3}
{
  \draw[fill=red, color=red] (\i, \i) circle (0.12);
  \draw (\i, \i) node[anchor=south east] {$s_{\i}$};
  \draw[fill=red, color=red] (10-\i, 10-\i) circle (0.12);
  \draw (10-\i, 10-\i) node[anchor=south east] {$\widetilde{t}_{\i}$};
}
\foreach \i in {1,2,3}
{
  \draw[fill=purple, color=purple] (3+\i, 4-\i) circle (0.12);
  \draw[fill=purple, color=purple] (6+\i, 7-\i) circle (0.12);
}
\foreach \i in {1,2,3}
{
  \draw[fill=purple, color=purple] (6+\i, 7-\i) circle (0.12);
  %\fill[white] (10-\i+0.75, 3+\i+0.3) rectangle ++(-0.5,-0.6);
  \draw (10-\i, 3+\i) node[anchor=south west] {$t_{\i}$};
}
\end{tikzpicture}
\longleftrightarrow
\young(\one\one2,2\otwo\otwo,33\othree).
\]
We add the points $\widetilde{t}_i$ to denote the corresponding terminal points if the paths were extended to be Dyck paths in the shaded upper fixed portion.
\end{example}

\begin{remark}
\label{rem:other_forms}
We can write the determinant of Theorem~\ref{thm:general_det_formula_Cn} under the specialization $x_i = 1$ as
\[
\dim V(\ell \fw_k) = \det \left[ \Cat_{(i+j+2n-k+1,j-i+k)} \right]_{i,j=0}^{\ell-1} =\det \left[ \Cat_{(i+j+2n-k+1,i-j+k)} \right]_{i,j=0}^{\ell-1}.
\]
Furthermore, we have
\begin{equation}
\label{eq:extending_det}
\dim V(\ell \fw_k) = \det \left[ \Cat_{(a(i,j), b(i,j))} \right]_{i,j=1}^{\ell} = \det \left[ \Cat_{(a(i,j), b(i,j)+1)} \right]_{i,j=1}^{\ell},
\end{equation}
by noting that we can only extend the non-intersecting lattice paths corresponding to the first determinant by a $N$ step. See, \textit{e.g.}, Example~\ref{ex:JT_full_rect}.
\end{remark}

Note that the proof of Theorem~\ref{thm:general_det_formula_Cn} does not extend to a Jacobi--Trudi-type formula for $\ch(\ell \fw_n)$ for $\ell > 2$ by using $\det [\ch(\fw_{n+1+i+j})]_{i,j=0}^{\ell-1}$ (or even their principal specializations) in contrast to Remark~\ref{rem:catalan_extension}. This is because any path ending at $t_{\ell}$ has to have the final $N$ step with a weight of $0$ as per Remark~\ref{rem:removed_fixed} (in Example~\ref{ex:JT_full_rect}, this would be the edge $(t_3, \widetilde{t}_3)$). However, for paths ending at $t_1$, the weight of this $N$ step must be $n+1$. Therefore, we cannot apply the LGV lemma in this situation. To demonstrate that, consider the principal specialization, where we have
\begin{align*}
q^{13} \det [q^{-\binom{3+i+j}{2}} \Cat_{3+i+j}(q)]_{i,j=0}^1 & = q^{26} - q^{25} + 2q^{24} + 2q^{22} - q^{21} + 4q^{20} - 2q^{19} + 4q^{18} - 3q^{17}
\\ & \hspace{20pt} + 3q^{16} - 3q^{15} + 4q^{14} - 6q^{13} + 4q^{12} - 3q^{11} + 3q^{10} - 3q^{9}
\\ & \hspace{20pt} + 4q^{8} - 2q^{7} + 4q^{6} - q^{5} + 2q^{4} + 2q^{2} - q + 1.
\end{align*}
Also note that we also cannot use the Mahonian $q$-Catalan numbers:
\begin{align*}
\det [\Cat_{3+i+j}(q)]_{i,j=0}^1 & = q^{26} + q^{24} + 2q^{23} + 2q^{22} + 2q^{21} + 3q^{20} + 3q^{19} + 3q^{18} + 3q^{17}
\\ & \hspace{20pt} + q^{16} + 2q^{15} + q^{14} - 2q^{12} - 2q^{10} - q^{9} - 2q^{8} - q^{7} - q^{6} - q^{4}.
\end{align*}

Note that when we take the principal specialization (resp.\ specialize $x_i = 1$), we obtain the following determinant formula for the principal specializations (resp.\ dimensions) of type $C_n$ representations.
\[
\det [\Cat'_{(a(i,j), b(i,j))}(q) ]_{i,j=1}^{\ell} = \ps(\lambda),
\qquad\qquad
\det [\Cat_{(a(i,j), b(i,j))} ]_{i,j=1}^{\ell} = \dim V(\lambda).
\]
In particular, we note the following, which we give two other alternative proofs in Appendix~\ref{app:alt_proofs}.
Another alternative proof is given by a specialization of~\cite[Thm.~2.1]{Okada09}.

\begin{corollary}
\label{cor: Cn determinatal r omega_n}
For type $C_n$, we have
\[
\det [ \Cat_{n+1+i+j} ]_{i,j=0}^{r-1} = \dim V(r \fw_n).
\]
\end{corollary}

%%%%%%%%%%%%%%%%%%%%%%%%%%%%%%%%%%%%%%%%
\section{Determinant formulas for type $B_n$ and $C_n$ decomposition multiplicities}
\label{sec:decomposition_multiplicities}

The starting point for our interpretation of decomposition multiplicities in terms of lattice paths begins with a linear algebra definition. A \defn{Hankel matrix} is a matrix $H = (h_{ij})_{i,j=0}^{n-1}$ such that $h_{ij} = f(i+j)$ for some fixed function $f \colon \Z \to \Z$. We then will use the LGV lemma to give an interpretation of non-intersecting lattice paths, where $f$ has a lattice path interpretation.

\begin{theorem}
\label{thm:Hankel_Catalan_det}
The decomposition multiplicity of $B(0)$ in $B(\fw_n)^{\otimes 2m}$ in type $B_n$ is the determinant of the $n \times n$ Hankel matrix
\[
H^B_{n,m} \seteq \left[ \Cat_{2n-i-j+m} \right]_{i,j=1}^n.
\]
\end{theorem}

\begin{proof}
By~\cite[Thm.~1]{MW00}, the determinant of $H^B_{n,m}$ is given by the number of $n$ non-intersecting lattice paths that stay below the anti-diagonal from $n$ adjacent sites to $n$ adjacent sites on the anti-diagonal separated by $m$ places.
Thus, it is sufficient to define a bijection from the non-intersecting lattice paths and highest weight elements in $B(\fw_n)^{\otimes 2m}$.

We first note that if we run diagonals down from the upper-most initial and lower-most terminal point, everything not between these lines is fixed.
Consider a set of non-intersecting lattice paths $\nilp = (P^{(i)})_{i=1}^n$, and we truncate the paths to the non-fixed region. Hence, each path has length $2m$.
We label the paths/initial points from $1$ to $n$ starting from the lower left.
Let $\Psi(\nilp)$ denote the tensor product of spins, where the $i$-th row of the $a$-th factor is a $+$ (resp.~$-$) if the $a$-th step $P^{(i)}_a$ is an $E$ (resp.~$N$) step.

Note that $e_n b = 0$ if and only if we have the Yamanouchi condition on the $\pm$, which we can consider as the signature rule. We note that the Yamanouchi condition is equivalent to $p^{(n)}$ staying strictly below the diagonal.

Next, assume $i < n$. Note that the super antidiagonal of the $a$-th position of the $i$-th path corresponds to $i$ plus the coefficient of $\epsilon_i$ of the tensor product truncated to $a$ factors. Note that $P^{(i)}$ intersects $P^{(i+1)}$ at position $a$ if and only if $c_{i+1} - c_i < 0$ for
\[
c_1 \epsilon_1 + \cdots + c_n \epsilon_n = \wt(s_a \otimes \cdots \otimes s_1).
\]
Recall that $s_m \otimes \cdots \otimes s_1$ is highest weight if and only if $\wt(s_a \otimes \cdots \otimes s_1) \in P^+$ for all $1 \leq a \leq m$, and a weight $c_1 \epsilon_1 + \cdots + c_n \epsilon_n \in P^+$ if and only if $c_1 \geq c_2 \geq \cdots \geq c_n \geq 0$.
Therefore, the paths are non-intersecting if and only if the corresponding element is highest weight.
\end{proof}

\begin{example}
Consider type $B_3$, and let $\Phi$ denote the bijection given in the proof of Theorem~\ref{thm:Hankel_Catalan_det}. Under $\Phi$, we have
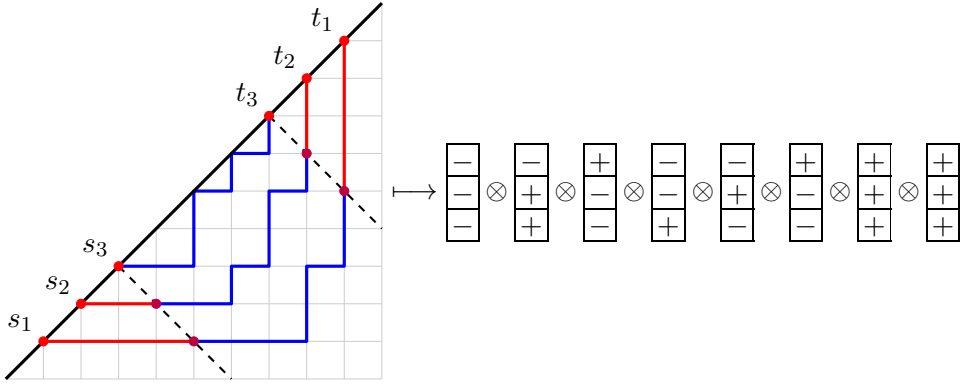
\[
\begin{tikzpicture}[baseline=68, scale=0.5]
\dyckgrid{10}
% Draw the paths
\draw[very thick, blue, line join=round] (3,3) -- (5,3) -- (5, 5) -- (6,5) -- (6,6) -- (7,6) -- (7,7);
\draw[very thick, blue, line join=round] (4,2) -- (6,2) -- (6,3) -- (7,3) -- (7,5) -- (8,5) -- (8,6);
\draw[very thick, blue, line join=round] (5,1) -- (8,1) -- (8,3) -- (9,3) -- (9,5);
% Draw the stuff over it
\draw[very thick,-] (0,0) -- (10,10);
\foreach \i in {1,2}
{
  \draw[very thick, red] (\i, \i) -- +(6-2*\i, 0);
  \draw[very thick, red] (10-\i, 10-\i) -- +(0, -6+2*\i);
}
\draw[thick,dashed] (3,3) -- +(3, -3);
\draw[thick,dashed] (7,7) -- +(3, -3);
\foreach \i in {1,2,3}
{
  \draw[fill=red, color=red] (\i, \i) circle (0.12);
  \draw (\i, \i) node[anchor=south east] {$s_{\i}$};
  \draw[fill=red, color=red] (10-\i, 10-\i) circle (0.12);
  \draw (10-\i, 10-\i) node[anchor=south east] {$t_{\i}$};
}
\foreach \i in {1,2}
{
  \draw[fill=purple, color=purple] (3+\i, 3-\i) circle (0.12);
  \draw[fill=purple, color=purple] (7+\i, 7-\i) circle (0.12);
}
\end{tikzpicture}
\longmapsto
\young(-,-,-) \otimes \young(-,+,+) \otimes \young(+,-,-) \otimes \young(-,-,+) \otimes \young(-,+,-) \otimes \young(+,-,-) \otimes \young(+,+,+) \otimes \young(+,+,+)\ .
\]
\end{example}

\begin{remark}
An equivalent definition of the Hankel matrix used in Theorem~\ref{thm:Hankel_Catalan_det} as $H^B_{n,m} = (\Cat_{i+j+m})_{i,j=0}^{n-1}$. However, this indexing does not give a direct correlation between the LGV lemma path $P^{(i)}$ and the weight $\epsilon_i$.
\end{remark}

We remark that the decomposition multiplicity of $B(0)$ in $V(\fw_n)^{\otimes 2m}$ in type $B_n$ is the $m$-th number of~\cite[A006149]{OEIS} for $n = 3$,~\cite[A006150]{OEIS} for $n=4$, and ~\cite[A006151]{OEIS} for $n=5$.

Note that the proof of~\cite[Thm.~1]{MW00} is essentially a direct application of LGV lemma and the associated combinatorics of the lattice paths being Catalan numbers. We can generalize this to obtain the multiplicity of $B(\lambda)$ for general $\lambda \in P_{\Z}^+$ as both certain non-intersecting lattice paths and as a determinant of Catalan triangle numbers.

As an motivating example, we can reformulate Theorem~\ref{thm:Hankel_Catalan_det} using (the determinant of) the matrix
\[
H^B_{(n,0)} =\left[ \Cat_{(2n-i-j+m, j-i+m)} \right]_{i,j=1}^n,
\]
and note that $0 \in P_{\Z}^+$.
Additionally, note that the only non-intersecting lattice paths must come from the identity permutation. The same idea holds more generally and yields the following.

\begin{theorem}
\label{thm:triangular_Catalan_det}
For type $B_n$, let $\lambda = \sum_{i \in I} c_i \epsilon_i \in P_{\Z}^+$ with $c_n \in \Z$.
The decomposition multiplicity of $B(\lambda)$ in $B(\fw_n)^{\otimes 2m}$ is the determinant of the $n \times n$ matrix
\[
H^B_{(n,\lambda)} = \left[ \Cat_{(a(i,j),b(i,j))} \right]_{i,j=1}^n,
\]
where
\begin{align*}
a(i,j) & = 2n - i - j + m + c_j,
\\ b(i,j) & = j - i + m - c_j.
\end{align*}
\end{theorem}

%sage: def tricat(n,k):
%....:     if k > n or k < 0:
%....:         return 0
%....:     return factorial(n+k) * (n-k+1) / (factorial(k) * factorial(n+1))
%sage: WCR = WeylCharacterRing(['B',3], style='coroots')
%sage: sp = WCR(0,0,1)
%sage: sp^8
%330*B3(0,0,0) + 924*B3(1,0,0) + 1540*B3(0,1,0) + 1925*B3(0,0,2) + 972*B3(2,0,0) + 2100*B3(1,1,0) + 2835*B3(1,0,2) + 1344*B3(0,2,0) + 2268*B3(0,1,2) + 980*B3(0,0,4) + 462*B3(3,0,0) + 1100*B3(2,1,0) + 1540*B3(2,0,2) + 924*B3(1,2,0) + 1617*B3(1,1,2) + 770*B3(1,0,4) + 300*B3(0,3,0) + 567*B3(0,2,2) + 350*B3(0,1,4) + 84*B3(0,0,6) + 84*B3(4,0,0) + 210*B3(3,1,0) + 300*B3(3,0,2) + 196*B3(2,2,0) + 350*B3(2,1,2) + 175*B3(2,0,4) + 84*B3(1,3,0) + 162*B3(1,2,2) + 105*B3(1,1,4) + 28*B3(1,0,6) + 14*B3(0,4,0) + 28*B3(0,3,2) + 20*B3(0,2,4) + 7*B3(0,1,6) + B3(0,0,8)
%sage: c = [0,0,1]
%sage: matrix([[tricat(i+(j+c[j])+4,i-(j+c[j])+4) for i in range(3)] for j in range(3)]).det()
%924
%sage: c = [0,1,1]
%sage: matrix([[tricat(i+(j+c[j])+4,i-(j+c[j])+4) for i in range(3)] for j in range(3)]).det()
%1540
%sage: c = [1,1,1]
%sage: matrix([[tricat(i+(j+c[j])+4,i-(j+c[j])+4) for i in range(3)] for j in range(3)]).det()
%1925

\begin{remark}
For a $2m+1$ tensor factor analog of Theorem~\ref{thm:triangular_Catalan_det} (and Theorem~\ref{thm:Hankel_Catalan_det}), we simply add $\fw_n$ to $\lambda$. However, this does not change the entries of the matrix or the dimensions of the representations.
\end{remark}

\begin{example}
We want to consider the multiplicity of $B(\tfw_1+\tfw_3)$ in $B(\omega_3)^{\otimes 8}$ in type $B_3$. Note that $\tfw_1 + \tfw_3 = 2\epsilon_1 + \epsilon_2 + \epsilon_3$. One such non-intersecting lattice path and its image under the bijection of Theorem~\ref{thm:triangular_Catalan_det}
 is
\[
\begin{tikzpicture}[scale=0.5, baseline=68]
\dyckgrid{11}
% Draw the paths
\draw[very thick, blue, line join=round] (3,3) -- (5,3) -- (5,5) -- (6,5) -- (6,6) -- (8,6);
\draw[very thick, blue, line join=round] (4,2) -- (6,2) -- (6,3) -- (9,3) -- (9,5);
\draw[very thick, blue, line join=round] (5,1) -- (8,1) -- (8,2) -- (10,2) -- (10,3) -- (11,3);
% Draw the stuff over it
\draw[very thick,-] (0,0) -- (11,11);
\draw[thick,dashed] (3,3) -- +(3, -3);
\draw[thick,dashed] (7,7) -- +(4, -4);
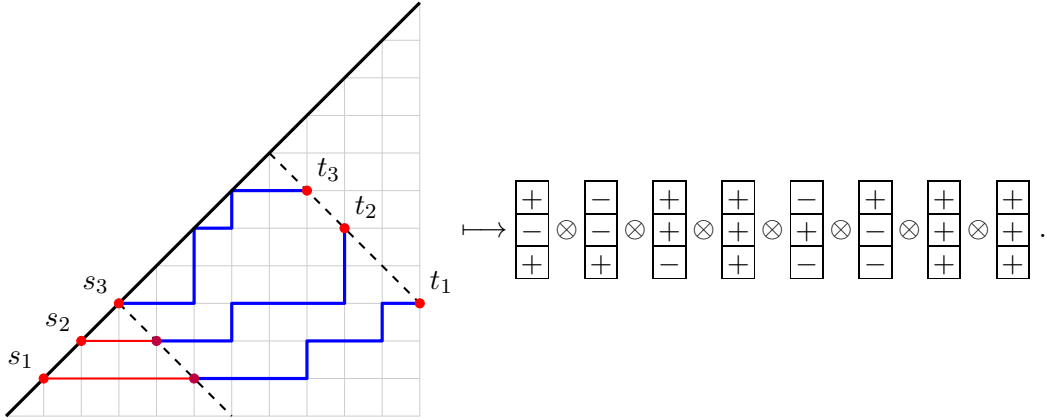
\foreach \i in {1,2}
{
  \draw[color=red,thick] (3-\i,3-\i) -- (3+\i,3-\i);
  \draw[fill=purple, color=purple] (3+\i, 3-\i) circle (0.12);
}
\foreach \i in {1,2,3}
{
  \draw[fill=red, color=red] (\i, \i) circle (0.12);
  \draw (\i, \i) node[anchor=south east] {$s_{\i}$};
}
\foreach \i in {1,2,4} { \draw[fill=red, color=red] (7+\i, 7-\i) circle (0.12); }
\draw (7+1, 7-1) node[anchor=south west] {$t_3$};
\draw (7+2, 7-2) node[anchor=south west] {$t_2$};
\draw (7+4, 7-4) node[anchor=south west] {$t_1$};
\end{tikzpicture}
\!\!\longmapsto
\young(+,-,+) \otimes \young(-,-,+) \otimes \young(+,+,-) \otimes \young(+,+,+) \otimes \young(-,+,-) \otimes \young(+,-,-) \otimes \young(+,+,+) \otimes \young(+,+,+)\ .
\]
\end{example}

Let $\overline{\lambda}$ denote the complement partition of $\lambda$ in an $\ell \times n$ box.
Specifically, we have $\overline{\lambda}_i = n - \lambda_{\ell+1-i}$.
We can reformulate Theorem~\ref{thm:general_det_formula_Cn} when specializing $q = 1$ as
\[
\dim V(\lambda) = \det \left[ \Cat_{(2n-1+i+j-\lambda_j, i-j+\lambda_j)} \right]_{i,j=1}^{\ell} = \det \left[ \Cat_{(\overline{a}(i,j), \overline{b}(i,j))} \right]_{i,j=1}^{\ell},
\]
where
\begin{align*}
\overline{a}(i,j) & = 2\ell-i-j+n+1+\overline{\lambda}_j,
\\ \overline{b}(i,j) & = j-i+n-\overline{\lambda}_j.
\end{align*}
We can compare this form for $\dim V(\lambda)$ with the result from Theorem~\ref{thm:triangular_Catalan_det}. Note that they almost agree under the substitutions:
\begin{center}
\begin{tabular}{c|ccc}
\hline
Theorem~\ref{thm:triangular_Catalan_det} & $n$ & $m$ & $c_j$
\\\hline
Theorem~\ref{thm:general_det_formula_Cn} & $\ell$ & $n+1$ & $\overline{\lambda}_j$
\\\hline
\end{tabular}
\end{center}
Yet, we can use Equation~\eqref{eq:extending_det} with $\lambda = n \fw_{m-1}$ in type $C_{m-1}$ so that they do agree, which yields the following.

\begin{corollary}
The decomposition multiplicity of $B(0)$ in $B(\fw_n)^{\otimes 2m}$ in type $B_n$ is equal to $\dim V(n \fw_{m-1})$ of type $C_{m-1}$.
\end{corollary}

Next, we give a type $C_n$ version of Theorem~\ref{thm:Hankel_Catalan_det}. In order to prove it, we first need to examine the crystal structure of $\bigwedge B(\omega_1)$ in type $C_n$. The elements of $\bigwedge^k B(\omega_1)$ are given by height $k$ tableaux since all elements must be distinct by semistandardness, and we choose representatives in the exterior product that are ordered (considering $B(\omega_1)$ as a poset). Similar to~\cite{S05}, removing all pairs $(i, \overline{\imath})$ from a column $T$ such that $k + 1 + p_i - p_{\overline{\imath}} > i$, where $p_x$ is the height of $x$ in $T$, gives a crystal isomorphism from $\bigwedge^k B(\omega_1)$ to KN tableaux.

\begin{remark}
\label{remark:KR_tableaux}
The tableaux for $\bigwedge B(\omega_1)$ can be considered as a tableaux model for the Kirillov--Reshetikhin crystal $\bigoplus_{k=1}^n B^{k,1}$ of type $A_{2n}^{(2)}$, where we consider $B^{0,1} = B(0)$,~\cite{OS08,FOS09} similar to the Kirillov--Reshetikhin tableaux  of~\cite{SchillingS15}.
\end{remark}

\begin{theorem}
\label{thm:Hankel_Catalan_type_C}
For type $C_n$, the multiplicity of $B(0)$ inside of $B = \left( \bigwedge B(\omega_1) \right)^{\otimes m}$ is the determinant of the $n \times n$ Hankel matrix
\[
H^C_{n,m+1} = \left[ \Cat_{m+1+i+j} \right]_{i,j=0}^{n-1}.
\]
\end{theorem}

\begin{proof}
We first ignore the first $(n-i) + 1$ and last steps in the path $P^{(i)} \colon s_i \to t_i$ (they are forced to be $E$ and $N$ respectively).
Thus, we get a bijection between highest weight elements of $B$ and lattice paths as follows. Consider the non-forced path from $s_i$ to $t_i$. We consider the $(2k, 2k+1)$ steps of the path to the $k$-th column tableau by an $E N$ in row $i$ being an $i, \overline{\imath}$ pair being added, a $E E$ being an $i$, and $N N$ being an $\overline{\imath}$ to the column (sorting the result as necessary). A $N E$ contributes nothing. Note that a column of height $h$ might not be valid in terms of KN tableaux, but it comes from a $B(\omega_a) \subseteq \bigwedge^h B(\omega_1)$.

It is clear that this process is reversible. The proof that the result is a highest weight element is similar to the proof of Theorem~\ref{thm:Hankel_Catalan_det}.
\end{proof}

\begin{example}
Consider type $C_3$, and let $\Omega$ denote the bijection given in the proof of Theorem~\ref{thm:Hankel_Catalan_type_C}. Under $\Omega$, we have
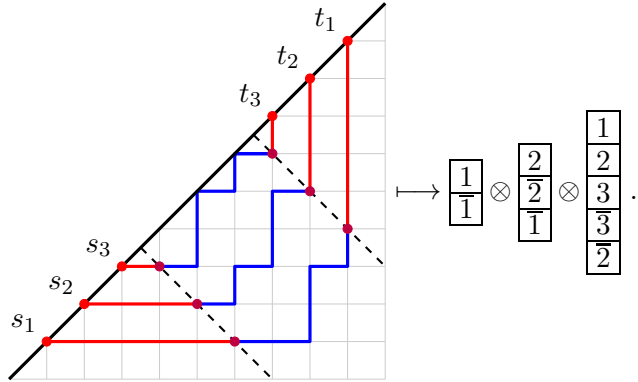
\[
\begin{tikzpicture}[scale=0.5, baseline=68]
\dyckgrid{10}
% Draw the paths
\draw[very thick, blue, line join=round] (4,3) -- (5,3) -- (5, 5) -- (6,5) -- (6,6) -- (7,6) -- (6,6);
\draw[very thick, blue, line join=round] (5,2) -- (6,2) -- (6,3) -- (7,3) -- (7,5) -- (8,5) -- (8,5);
\draw[very thick, blue, line join=round] (6,1) -- (8,1) -- (8,3) -- (9,3) -- (9,4);
% Draw the stuff over it
\draw[very thick,-] (0,0) -- (10,10);
\foreach \i in {1,2,3}
{
  \draw[very thick, red] (\i, \i) -- +(7-2*\i, 0);
  \draw[very thick, red] (10-\i, 10-\i) -- +(0, -7+2*\i);
}
\draw[thick,dashed] (3.5,3.5) -- +(3.5, -3.5);
\draw[thick,dashed] (6.5,6.5) -- +(3.5, -3.5);
\foreach \i in {1,2,3}
{
  \draw[fill=red, color=red] (\i, \i) circle (0.12);
  \draw (\i, \i) node[anchor=south east] {$s_{\i}$};
  \draw[fill=red, color=red] (10-\i, 10-\i) circle (0.12);
  \draw (10-\i, 10-\i) node[anchor=south east] {$t_{\i}$};
}
\foreach \i in {1,2,3}
{
  \draw[fill=purple, color=purple] (3+\i, 4-\i) circle (0.12);
  \draw[fill=purple, color=purple] (6+\i, 7-\i) circle (0.12);
}
\end{tikzpicture}
\longmapsto
\young(1,\one) \otimes \young(2,\otwo,\one) \otimes \young(1,2,3,\othree,\otwo) \ .
\]
\end{example}

We note that the proof of Theorem~\ref{thm:Hankel_Catalan_type_C} is related to the virtualization map $v \colon B(\fw_n) \to B(\fw_n)^{\otimes 2}$ (see, {\it e.g.},~\cite{K96,OSS03III,OSS03II,SchillingS15}) of $A_{2n}^{(2)}$ to $D_{n+1}^{(2)}$, where the scaling factors are $(\gamma_k)_{k=1}^n = (1, \dotsc, 1, 2)$. Indeed, using the wedge product tableaux from Remark~\ref{remark:KR_tableaux},
is given by the pairs in row $i$
\begin{align*}
++ & \longmapsto i, & -- & \longmapsto \overline{\imath}, \\
-+ & \longmapsto i\overline{\imath}, & +- & \longmapsto \emptyset,
\end{align*}
added to the column (and sorting as necessary).
%the isomorphism is given by
%\begin{align*}
%i & \longmapsto ++, & \overline{\imath} & \longmapsto --, \\
%i\overline{\imath} & \longmapsto -+, & \emptyset & \longmapsto +-,
%\end{align*}
%added to the corresponding pair of columns in row $i$.
Note that we also have to reverse the order of the tensor factors to connect it with our lattice path interpretation given in the proof. However, this could also be rectified by reversing the starting and terminal point labels.

Note that we can extend Theorem~\ref{thm:Hankel_Catalan_type_C} to compute the multiplicity of $B(\lambda)$ inside of $B$ in parallel to how we obtained Theorem~\ref{thm:triangular_Catalan_det} from Theorem~\ref{thm:Hankel_Catalan_det}.

\begin{theorem}
\label{thm:triangular_Catalan_det_type_C}
For type $C_n$, let $\lambda = \sum_{i \in I} c_i \epsilon_i \in P^+$.
The multiplicity of $B(\lambda)$ inside of $B = \left( \bigwedge B(\omega_1) \right)^{\otimes m}$ in type $C_n$ is the determinant of the $n \times n$ matrix
\[
H^C_{(n,\lambda)} = \left[ \Cat_{(a(i,j),b(i,j))} \right]_{i,j=0}^{n-1},
\]
where
\begin{align*}
a(i,j) & = 2n - i - j - 1 + m + c_j,
\\ b(i,j) & = j - i + m - c_j.
\end{align*}
\end{theorem}

\begin{example}
We want to consider the multiplicity of $B(\fw_1+\fw_3)$ in $B^{\otimes 4}$ in type $C_3$. Note that $\fw_1 + \fw_3 = 2\epsilon_1 + \epsilon_2 + \epsilon_3$. One such non-intersecting lattice path and its image under the bijection of Theorem~\ref{thm:triangular_Catalan_det_type_C} is
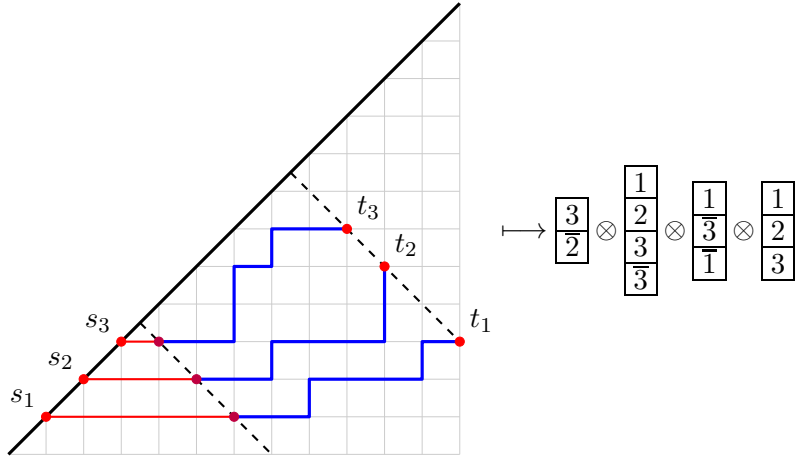
\[
\begin{tikzpicture}[scale=0.5, baseline=82]
\dyckgrid{12}
% Draw the paths
\draw[very thick, blue, line join=round] (4,3) -- (6,3) -- (6,5) -- (7,5) -- (7,6) -- (9,6);
\draw[very thick, blue, line join=round] (5,2) -- (7,2) -- (7,3) -- (10,3) -- (10,5);
\draw[very thick, blue, line join=round] (6,1) -- (8,1) -- (8,2) -- (11,2) -- (11,3) -- (12,3);
% Draw the stuff over it
\draw[very thick,-] (0,0) -- (12,12);
\draw[thick,dashed] (3.5,3.5) -- +(3.5, -3.5);
\draw[thick,dashed] (7.5,7.5) -- +(4.5, -4.5);
\foreach \i in {0,1,2}
{
  \draw[color=red,thick] (3-\i,3-\i) -- (3+\i+1,3-\i);
  \draw[fill=purple, color=purple] (4+\i, 3-\i) circle (0.12);
}
\foreach \i in {1,2,3}
{
  \draw[fill=red, color=red] (\i, \i) circle (0.12);
  \draw (\i, \i) node[anchor=south east] {$s_{\i}$};
}
\foreach \i in {2,3,5} { \draw[fill=red, color=red] (7+\i, 8-\i) circle (0.12); }
\draw (8+1, 8-2) node[anchor=south west] {$t_3$};
\draw (8+2, 8-3) node[anchor=south west] {$t_2$};
\draw (8+4, 8-5) node[anchor=south west] {$t_1$};
\end{tikzpicture}
\!\!\longmapsto
\young(3,\otwo) \otimes \young(1,2,3,\othree) \otimes \young(1,\othree,\one) \otimes \young(1,2,3)
\]
%EE = i
%NN = -i
%EN = i, -i
%NE = _
\end{example}

\begin{corollary}
We have that for any partition $\lambda$, the multiplicity of $B(\lambda)$ in $B^m$ in type $C_n$ is equal to $\dim V(\overline{\lambda}')$ in type $C_m$, where $\overline{\lambda}$ is the complement partition in an $\ell \times n$ box and $\mu'$ is the conjugate partition of $\mu$.
\end{corollary}

\begin{proof}
The claim follows by comparing Theorem~\ref{thm:general_det_formula_Cn} with the specialization $x_i = 1$ and Theorem~\ref{thm:triangular_Catalan_det_type_C} using
\begin{center}
\begin{tabular}{c|ccc}
\hline
Theorem~\ref{thm:triangular_Catalan_det_type_C} & $n$ & $m$ & $c_j$
\\\hline
Theorem~\ref{thm:general_det_formula_Cn} & $\ell$ & $n$ & $\overline{\lambda}_j$
\\\hline
\end{tabular}
\end{center}
Recall also that in Theorem~\ref{thm:general_det_formula_Cn}, we take the conjugate of the partition $(\lambda_1, \dotsc, \lambda_{\ell})$.
\end{proof}

%%%%%%%%%%%%%%%%%%%%%%%%%%%%%%%%%%%%%%%%
\section{Representation theoretic interpretations of combinatorial identities}
\label{sec:repr_identities}

% ========
\subsection{Type $C$ characters, $q$-binomials, and $q,t$-Catalan numbers}

We show that the principal specialization, when scaled by a power of $q$ to be polynomial, results in the Mahonian $q$-Catalan numbers.

By taking the (normalized) principal specialization, we can define $\wt_q(D) = \prod_N q^{w(N;D)}$, where $q^{\overline{\imath}} = q^{-i}$. Note that we can extend $w$ to a statistic on the entire $m \times n$ grid by considering $\overline{k} = -k$. In particular, the lower-left vertical step will be $-0$.

\begin{example}
\label{ex:specialization_statistic}
Consider the Dyck word $P = EENENEENEEN$. Then $\wt(P) = -1 - 2 + 4 - 5 = -4$.
The weighting $w$ on a full $7 \times 4$ grid, along with the particular Dyck word $P$ as a lattice path, is given by Figure~\ref{fig:path_statistic}.
\end{example}

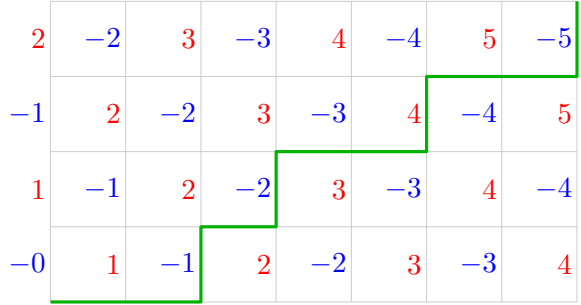
\begin{figure}
\begin{tikzpicture}[scale=1]
\draw[gray!40,very thin] (0,0) grid (7,4);
\draw (0,0.5) node[left=-2pt,blue] {$-0$};
\foreach \i in {0,...,2} \draw (2-\i,\i+0.5) node[left=-2pt,blue] {$-1$};
\foreach \i in {0,...,3} \draw (4-\i,\i+0.5) node[left=-2pt,blue] {$-2$};
\foreach \i in {0,...,3} \draw (6-\i,\i+0.5) node[left=-2pt,blue] {$-3$};
\foreach \i in {0,...,2} \draw (7-\i,\i+1.5) node[left=-2pt,blue] {$-4$};
\draw (7,3.5) node[left=-2pt,blue] {$-5$};
\foreach \i in {0,1} \draw (1-\i,\i+0.5) node[left=-2pt,red] {$1$};
\foreach \i in {0,...,3} \draw (3-\i,\i+0.5) node[left=-2pt,red] {$2$};
\foreach \i in {0,...,3} \draw (5-\i,\i+0.5) node[left=-2pt,red] {$3$};
\foreach \i in {0,...,3} \draw (7-\i,\i+0.5) node[left=-2pt,red] {$4$};
\foreach \i in {0,1} \draw (7-\i,\i+2.5) node[left=-2pt,red] {$5$};
\draw[-, very thick, black!30!green, line join=round] (0,0) -- (2,0) -- (2,1) -- (3,1) -- (3,2) -- (5,2) -- (5,3) -- (7,3) -- (7,4);
\end{tikzpicture}
\caption{An example of the weighting $w$ in a $7 \times 4$ grid and the lattice path from Example~\ref{ex:specialization_statistic}.}
\label{fig:path_statistic}
\end{figure}

\begin{theorem}
\label{thm:q_catalan_paths}
For $\g$ of type $C_{n-1}$, we have
\[
\nps(\fw_{n-1}) = q^{\binom{n}{2}} \ps(\fw_{n-1}) = \Cat_n(q).
\]
\end{theorem}

In order to show this, we separate the statistic $w(D)$ into a positive part and negative part. Thus, we now consider a $q,t$-analog by
\[
\sum_{D \in \Dyck_n} q^{w_+(D)} t^{w_-(D)},
\]
where $w_{\pm}(D)$ is the sum of the $\pm$ contributions to $w(D)$.

\begin{theorem}
\label{thm:qt_polynomial}
We have
\[
\Cat_n(q,t) = \sum_{D \in \Dyck_n} q^{w_+(D)} t^{w_-(D)}.
\]
\end{theorem}

\begin{proof}
Note that for any given path $D$, there can be at most one $N$ step that contributes a $k$ to $w_+(D)$ and similarly for a $\overline{k}$ to $w_-(D)$.
Thus, we can define sets
\begin{align*}
X_+(D) & \seteq \{ w_+(N, D) \mid \text{$N$ step of $D$ such that $N_X + N_Y$ is odd} \},
\\ X_-(D) & \seteq \{ \overline{w_-(N, D)} \mid \text{$N$ step of $D$ such that $N_X + N_Y$ is even} \},
\end{align*}
where the elements in $X_-(D)$ are unbarred letters (\textit{i.e.}, we consider $\overline{\overline{a}} = a$).
Note that we can recover $D$ knowing $X_{\pm}(D)$ since $X_{\pm}(D)$ uniquely determines a King tableaux.
Additionally note that $\lvert X_+(D) \rvert + \lvert X_-(D) \rvert = n-1$.

Recall that a Dyck path is determined by the positions of its valleys.
Thus, we define a map $\Upsilon \colon \Dyck_n \to \Dyck_n$ by having the valleys of the resulting Dyck path be given by $(x_i, y_i)$, where
\begin{align*}
\{x_1 < x_2 < \cdots < x_k\} & = X_+(D),
\\ \{y_1 < y_2 < \cdots < y_k \} & = \{1, \dotsc, n-1\} \setminus X_-(D).
\end{align*}
If $\Upsilon$ is well-defined, then it is clearly invertible.
In order to show $\Upsilon$ is well-defined, we require that $x_i \geq y_i$ for all $1 \leq i \leq k$.
By the definition of $w$ and the diagonal condition of a Dyck path, we can say a value $y_i$ was \defn{contributed} by a row $r$ where the $r$-th $N$ step is occurs to the right of an edge weighted by $\overline{y}_i$.
Note that the initial path of $D$ must be $E(EN)^{y_1-1}E^2$, which implies that $x_1 \geq y_1$.
It is a straightforward induction to see that any contributed $y_i$ must occur on the same row or below the $N$ step corresponding to $x_i$.
Thus, we have $x_i \geq y_i$ and $\Upsilon$ is well-defined.

Recall that $\maj_N(D)$ (resp.\ $\maj_E(D)$) is a sum over the vertical (resp.\ horizontal) positions of the peaks of $D$.
Therefore, we have $\maj_N\bigl(\eta(D) \bigr) = w_+(D)$ from Remark~\ref{rem:counting_NE}.
Additionally, we have $\binom{n}{2} - \maj_E\bigl(\eta(D) \bigr) = w_-(D)$ from Remark~\ref{rem:counting_NE} and that the sum of all possible values $\{1, \dotsc, n-1\}$ is $\binom{n}{2}$.
Hence, the claim follows.
\end{proof}

Theorem~\ref{thm:q_catalan_paths} is as an immediate consequence of Theorem~\ref{thm:qt_polynomial} and Proposition~\ref{prop:specialized_stump}.

\begin{example}
Suppose $n = 6$ and we have a path $D$ such that $X_+(D) = \{2, 5\}$ and $X_-(D) = \{1, 4, 5\}$. Then we construct $\Upsilon(D)$ as the Dyck path with the horizontal positions of the valleys being $2, 5$ with the corresponding vertical positions $2, 3$ coming from $[5] \setminus \{1, 4, 5\}$. Pictorially, we have
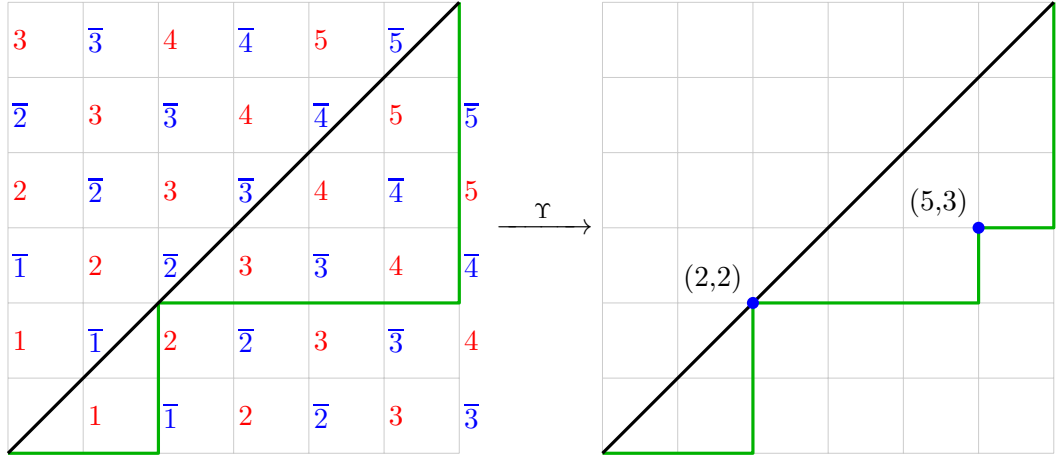
\[
\begin{tikzpicture}[scale=1,baseline=83pt]
\draw[gray!40,very thin] (0,0) grid (6,6);
\foreach \i in {0,...,2} \draw (2-\i,\i+0.5) node[right=-2pt,blue] {$\one$};
\foreach \i in {0,...,4} \draw (4-\i,\i+0.5) node[right=-2pt,blue] {$\otwo$};
\foreach \i in {0,...,5} \draw (6-\i,\i+0.5) node[right=-2pt,blue] {$\othree$};
\foreach \i in {0,...,3} \draw (6-\i,\i+2.5) node[right=-2pt,blue] {$\ofour$};
\foreach \i in {0,...,1} \draw (6-\i,\i+4.5) node[right=-2pt,blue] {$\overline{5}$};
\foreach \i in {0,1} \draw (1-\i,\i+0.5) node[right=-2pt,red] {$1$};
\foreach \i in {0,...,3} \draw (3-\i,\i+0.5) node[right=-2pt,red] {$2$};
\foreach \i in {0,...,5} \draw (5-\i,\i+0.5) node[right=-2pt,red] {$3$};
\foreach \i in {0,...,4} \draw (6-\i,\i+1.5) node[right=-2pt,red] {$4$};
\foreach \i in {0,...,2} \draw (6-\i,\i+3.5) node[right=-2pt,red] {$5$};
\draw[-, very thick, black!30!green,line join=round] (0,0) -- (2,0) -- (2,2) -- (6,2) -- (6,6);
\draw[-, very thick] (0,0) -- (6,6);
\end{tikzpicture}
\xrightarrow[\hspace{30pt}]{\Upsilon}
\begin{tikzpicture}[scale=1,baseline=83pt]
\draw[gray!40,very thin] (0,0) grid (6,6);
\draw[-, very thick, black!30!green, line join=round] (0,0) -- (2,0) -- (2,2) -- (5,2) -- (5,3) -- (6,3) -- (6,6);
\draw[-, very thick] (0,0) -- (6,6);
\fill[blue] (2,2) circle (0.08) node[anchor=south east, black] {(2,2)};
\fill[blue] (5,3) circle (0.08) node[anchor=south east, black] {(5,3)};
\end{tikzpicture}
\]
\end{example}

\begin{definition}
\label{def:qt_Catalan_triangle}
We define the \defn{$(n,k)$-th $(q,t)$-Catalan triangle number} as follows:
\[
\Cat_{(n,k)}(q,t) \seteq \sum_{D \in \Dyck_{n,k}} q^{w'_+(D)} t^{w'_-(D)}.
\]
\end{definition}

If we consider the Weyl group element that interchanges $i \leftrightarrow \overline{\imath}$ and use this to define a new weighting $\overline{w}'_{\pm}(D) = w'_{\mp}(D)$, then Remark~\ref{remark:weight_change} states that
\[
\Cat_{(n,k)}(q,t) = \sum_{D \in \Dyck_{n,k}} q^{\overline{w}'_+(D)} t^{\overline{w}'_-(D)} = \sum_{D \in \Dyck_{n,k}} q^{w'_-(D)} t^{w'_+(D)}.
\]
In other words, we have $\Cat_{(n,k)}(q,t) = \Cat_{(n,k)}(t,q)$; in particular, $\Cat_n(q,t) = \Cat_n(t,q)$.

\begin{remark}
We obtain the horizontal step approach by replacing each monomial $q^{\alpha} t^{\beta}$ by $q^{B-\alpha} t^{B-\beta}$, where $B = \binom{n}{2}$.
Furthermore, using the horizontal steps at $t = q^{-1}$, we obtain $\ps(\fw_{n-1})$ of type $C_{n-1}$.
\end{remark}

Our specialization of the character can be extended to general statistic on lattice paths in an $n \times m$ grid in a natural way, with the upper right horizontal step having weight $q^{-0}$. We say $-0$ because of the natural alternation of the sign as we go down or to the left. Let $R_{n,m}$ denote all paths given by $N$ and $E$ steps in an $n \times m$ rectangle, which we will consider as words in $\{N, E\}$ such that there are $n$ $N$'s and $m$ $E$'s appearing in the word. By slight abuse of notation, we also denote our extended statistic by $w' \colon R_{n,m} \to \Z$. Therefore, define
\[
w'_{n,m}(e) \seteq \begin{cases}
-(n - x + m - y - 1) / 2 & \text{if $x + y$ odd}, \\
 (n - x + m - y) / 2 & \text{if $x + y$ even},
\end{cases}
\]
where $e$ is an $E$ step beginning at $(x, y)$, and for $P \in R_{n,m}$, define
\begin{equation}
\label{eq:path_statistic}
w'(P) \seteq \sum_{e} w'_{n,m}(e)
\end{equation}
where the sum is over all $e \in P$ such that $e$ is an $E$ step. We call $w'_{n,m}(e)$ and $w'(P)$ the \defn{weight} of $e$ and $P$ respectively. We define the generating function
\begin{equation}
\label{eq:binomal_statistic}
\B_{n,m}(q) \seteq \sum_{P \in R_{n,m}} q^{w'(P)}.
\end{equation}

Before we can prove that this is a $q$-binomial coefficient up to a power of $q$, we first determine the valuation of $\B_{n,m}(q)$, the minimal power of $q$ occurring in $\B_{n,m}(q)$.

\begin{proposition}
The valuation of $\B_{n,m}(q)$ is
\[
v_{n,m} \seteq -\sum_{k=D_{n,m}}^{D_{n,m}+m-1} k,
\]
where $D_{n,m} \seteq \lfloor (n-m+1) / 2 \rfloor$.
\end{proposition}

\begin{proof}
We proceed by induction on $m$. Note that the case of $n = 0$ is trivial as $v_{n,m} = 0$ and the only path is $P = N^n$, which has $w'(P) = 0$. Suppose the claim holds for $m$. Fix some $n$. If $n-m$ is even,\footnote{Recall that the parity of $n-m$ is the same as for $n+m$.} then choose a path $P \in R_{n,m}$ such that $w'(P) = v_{n,m}$. Note that $D_{n,m} = D_{n,m+1}$. Therefore, we have
\begin{align*}
w'(EP) & = \frac{-(n-0+(m+1)-0-1)}{2} + w'(P) = -\frac{n+m}{2} + v_{n,m}
\\ & = -\frac{n+m}{2} - \sum_{k=D_{n,m}}^{D_{n,m}+m-1} k = -\sum_{k=D_{n,m}}^{D_{n,m}+m} k = -\sum_{k=D_{n,m+1}}^{D_{n,m+1}+(m+1)-1} k = v_{n,m+1}.
\end{align*}
To show there does not exist a path $P' \in R_{n,m+1}$ such that $w'(P') < w'(EP)$, we first write $P' = N^a E \overline{P}$. If $a = 0$, then $P' = P$, and so we assume $a > 0$. We can also assume that $\overline{P} \in R_{n-k,m}$ is a path such that $w'(\overline{P}) = v_{n-a,m}$. We note that $D_{n-a,m} \leq D_{n,m}$, and so $w'(P) = v_{n-a,m} \geq v_{n,m}$. Let $e$ be the first $E$ step in $P'$. If $a = 1$, then $w'_{n,m+1}(e) \geq 0$, which implies $w'(P) = v_{n,m}$. If $a > 1$, then $\lvert w'_{n,m+1}(e) \rvert \leq D_{n,m} + m$, and so $w'(P') = w'_{n,m+1}(e) + v_{n-a,m} \geq v_{n,m}$.

If $n-m$ is odd, then choose a path $P \in R_{n-1,m}$ such that $w'(P) = v_{n-1,m}$. Thus, we have
\begin{align*}
w'(NEP) & = \frac{(n - 0 + (m+1) - 1)}{2} + w'(P) = \frac{n+m}{2} + v_{n-1,m}
\\ & = \frac{n+m}{2} - \sum_{k=D_{n-1,m}}^{D_{n-1,m}+m-1} k = \frac{n+m}{2} - \sum_{k=D_{n,m+1}+1}^{D_{n,m+1}+m} k
\\ & = -\sum_{k=D_{n,m+1}}^{D_{n,m+1}+(m+1)-1} k = v_{n,m+1}.
\end{align*}
The proof that there does not exist a path $P' \in R_{n,m+1}$ such that $w'(P') < w'(EP)$ is similar to when $n-m$ is even except if $P' = E\overline{P} \in R_{n,m+1}$. In that case, we have $w'_{n,m+1}(e) \geq 0$. Since $D_{n,m} = D_{n,m-1} - 1$, the claim follows.
\end{proof}

%The {\sc SageMath} code to generate the shifts:
%\begin{verbatim}
%sage: matrix([[sum(k for k in range((i-j+1)//2, (i-j+1)//2+j))
%....:          for j in range(8)] for i in range(8)])
%\end{verbatim}

\begin{theorem}
\label{thm:q_binomial_paths}
We have
\[
q^{-v_{m,n}} \B_{n,m}(q) = \qbinom{n+m}{n}{q}.
\]
\end{theorem}

\begin{proof}
We first note that for the lower left horizontal edge $e$, we have
\[
w'(e) = \begin{cases}
(n+m)/2 & \text{if $n+m$ is even,} \\
-(n+m-1)/2 & \text{if $n+m$ is odd.}
\end{cases}
\]
We show that our summation fits into the $q$-Pascal's triangle relation~\ref{eq:q_pascal}, where which identity we will use depends on the parity of $n$.
Consider a path $P \in R_{n,m}$, and let $\overline{P}$ denote that path with the first step $e$ removed. We split this into two cases based upon the first step.
\begin{enumerate}
\item The first step is horizontal, so $\overline{P} \in R_{n,m-1}$. In this case, the net difference in weight of the paths is $w'(e) + (v_{n,m} - v_{n,m-1})$.
\begin{enumerate}
\item If $n + m$ is even, then $D_{n,m-1} = D_{n,m} + 1$ and
\[
-v_{n,m} = -v_{n,m-1} + D_{n,m}.
\]
So the net change in weight is
\[
\frac{n+m}{2} + D_{n,m} = \frac{n+m}{2} + \frac{n - m}{2} = n.
\]
where for the first equality, we could remove the floor because $n + m$ is even.
\item If $n + m$ is odd, then $D_{n,m-1} = D_{n,m}$ and
\[
-v_{n,m} = -v_{n,m-1} + D_{n,m} + m - 1.
\]
Therefore, the net change in weight is
\begin{align*}
-\frac{n + m - 1}{2} + D_{n,m} + m - 1 & =  -\frac{n + m - 1}{2} + \frac{n - m + 1}{2} + m - 1
\\ & = -m + 1 + m - 1 = 0,
\end{align*}
where the first equality follows from the fact $n + m$ is odd.
\end{enumerate}

\item The first step is vertical, so $\overline{P} \in R_{n-1,m}$. In this case the net difference in weight of the paths is $v_{n,m} - v_{n-1,m}$.
\begin{enumerate}
\item If $n + m$ is even, then $D_{n-1,m} = D_{n,m}$ and $v_{n,m} = v_{n-1,m}$. So the net change in weight is $0$.
\item If $n + m$ is odd, then $D_{n-1,m} = D_{n,m} - 1$ and
\[
-v_{n,m} = -v_{n-1,m} + (D_{n,m} + m - 1) - D_{n-1,m} = -v_{n-1,m} + m.
\]
So the net change in weight is $m$.
\end{enumerate}
\end{enumerate}
Therefore, if $n + m$ is even, we have a combinatorial interpretation of Equation~\eqref{eq:pascal1}, and otherwise we use Equation~\eqref{eq:pascal2}. The boundary cases are straightforward, and the claim follows by induction.
\end{proof}

We also give an alternative proof of Theorem~\ref{thm:q_binomial_paths} for the case when $n + m$ is odd. In this case, it can be considered as the principal specialization of a type $B_n$ character.

\begin{proof}[Second proof of Theorem~\ref{thm:q_binomial_paths} with $n + m$ is odd]
Assume $m + n$ is odd.
Let
\[
w(P; \xx) \seteq \prod_N \begin{cases}
x_{(N_X + N_Y + 1) / 2} & \text{if $N_X + N_Y$ odd}, \\
x^{-1}_{(N_X + N_Y) / 2} & \text{if $N_X + N_Y$ even},
\end{cases}
\]
where we take the product over all $N$ steps of $P$.
We have
\[
\nps(\tfw_i) = \qbinom{2n+1}{i}{q}
\]
by Proposition~\ref{prop:q_binomial_repr}.
So it is sufficient to find a bijection $\phi$ between the KN tableaux for $B(\tfw_m)$ in type $B_r$, where $r = (n+m-1) / 2$, and lattice paths in an $m \times n$ grid such that $w\bigl( \phi(T); \xx \bigr) = \wt(T)$.

Let $T \in B(\tfw_m)$, and define a new tableaux $T^F$ by removing pairs of $0$'s in $T$ and then performing the filling map $F_m$ given by~\cite[\S3.4]{S05}. Note that this gives a bijection with the set $F(\tfw_m)$ of semistandard tableaux in the alphabet $1 \prec 2 \prec \cdots \prec n \prec 0 \prec \on \prec \cdots \prec \otwo \prec \one$ with shape $1^m$.
Next, if we sort the alphabet by $0 \prec' 1 \prec' \one \prec' 2 \prec' \otwo \prec' \cdots \prec' n \prec' \on$ and the tableau $T^F \in F(\tfw_m)$ accordingly to $T^s$, we can construct the lattice path by having the $i$-th $N$ step correspond to the $i$-th row in $T^s$. This is clearly bijective, and define the result to be $\phi(T)$. Moreover, all steps are clearly weight preserving, and hence the claim follows.
\end{proof}

Note that we can construct a $q,t$-analog of a binomial coefficient by ending $\B_{n,m}(q)$ to $\B_{n,m}(q,t)$ analogously to the $q,t$-Catalan numbers in Theorem~\ref{thm:qt_polynomial}. Likewise, from our construction we have $\B_{n,m}(q,q^{-1}) = \B_{n,m}(q)$ .

% ========
\subsection{Touchard identity and its generalizations}

Recall the classical Touchard identity on Catalan numbers:
\[
\Cat_{n+1}=\sum_{0 \le k \le n/2} \binom{n}{2k} 2^{n-2k} \Cat_k
= \sum_{0 \le k \le n/2} \binom{n}{n-2k} 2^{n-2k} \Cat_k.
\]

Using~\cite[Thm.~7.3]{KLO17} and our observation, we can obtain the following
identity, which can be considered as triangular analogue of Touchard's identity:

\begin{theorem}
\label{thm:Touchard_identity}
In type $C_n$, we have
\[
\dim V(\fw_s) = \Cat_{(2n-s+1,s)} = \sum_{0 \le i \le s/2} \binom{n}{s-2i}2^{s-2i} \Cat_{(n-s+i,i)}.
\]
\end{theorem}

\begin{proof}
First, the fact that $\dim V(\fw_s) = \Cat_{(2n-s+1,s)}$ follows from
the description of King tableaux, and the fact that $\dim \ V(\fw_s)_{\fw_{s-2i}}=\Cat_{(n-s+i,i)}$
follows from~\cite[Thm.~7.3]{KLO17} or~\cite[Thm.~17.5]{FH91}.
Since every dominant weight space of $V(\fw_s)$ is one of
$V(\fw_s)_{\fw_{s-2i}}$, our assertion follows from the
$B$-type Weyl group $W_n$; that is,
the number of elements in the orbit of $\fw_{s-2i}$ is the same as
$\binom{n}{s-2i}2^{s-2i}$.
\end{proof}

Now we can express the Mahonian $q$-Catalan numbers $C_n(q)$ as follows:
\[
\Cat_n(q) = q^{\binom{n}{2}} \sum_{0 \le k \le n/2} \Cat_k \sum_{1 \le i_1 < \cdots < i_{2k} \le n}
\prod_{j \in [1,n] \setminus \{ i_1,\ldots,i_{2k}\} } (q^{-j}+q^{j}) = q^{\binom{n}{2}} \dim_q V(\fw_n).
\]
Moreover, we define new $(q,t)$-Catalan numbers by replacing $q^{-j}$ with $t^j$:
\begin{equation}\label{eq: q,t catalan number}
\Cat'_n(q,t) \seteq \sum_{0 \le k \le n/2} \Cat_k \sum_{1 \le i_1 <\cdots < i_{2k} \le n}
\prod_{j \in [1,n] \setminus \{ i_1,\ldots,i_{2k}\} } (q^{j}+t^{j}).
\end{equation}
We can also generalize the above definitions for $(q,t)$-Catalan triangle numbers:
\begin{equation}\label{eq: q,t catalan triangle}
\Cat'_{(2n-s+1,s)}(q,t) \seteq
\sum_{0 \le k \le s/2} \Cat_{(n-s+k,k)} \sum_{1 \le i_1 < \cdots < i_{s-2k} \le n}
\prod_{j \in \{ i_1,\ldots,i_{s-2k}\} } (q^{j}+t^{j}).
\end{equation}

\begin{remark}
\label{rem:qt_catalan_are_different}
We note that~\eqref{eq: q,t catalan number} (resp.~\eqref{eq: q,t catalan triangle}) is different than the $q,t$-Catalan (resp.\ triangle) number given by Definition~\ref{thm:qt_polynomial} (resp.~Definition~\ref{def:qt_Catalan_triangle}). For example, consider
\begin{align*}
\Cat'_4(q,t) & = q^6 + q^5 t + q^4 t^2 + 2 q^3 t^3 + q^2 t^4 + q t^5 + t^6 + q^3 + t^3 + q^2 + t^2 + q + t,
\\ \Cat_4(q,t) & = q^6 + q^5 t + q^4 t^2 + 2 q^3 t^3 + q^2 t^4 + q t^5 + t^6 + q^4 t + q^3 t^2 + q^2 t^3 + q t^4 + q^3 t + q t^3.
\end{align*}
\end{remark}

\begin{theorem}
\label{thm:column_ps_type_C}
In type $C_n$, we have
\[
\ps(\fw_s) = \Cat'_{(2n-s+1,s)}(q,q^{-1}).
\]
\end{theorem}

The above observation can be extended to the $B_n$-type cases. First, the following
identity can be proved by using representation theory since the dimension of the weight space $V(\tfw_s)_{\tfw_k}$ in type $B_n$ 
is $\binom{n-k}{ \lfloor (s-k)/2 \rfloor }$ for $0 \le k \le s$ (see~\cite[Theorem 7.5]{KLO17} or~\cite[Theorem 19.2]{FH91}) and $\dim V(\tfw_s) = \binom{2n+1}{s}$  (it may be well-known to combinatorialists): For $1 \le s \le n$, we have
\[
\sum_{k=0}^{s} 2^k \binom{n}{k}\binom{n-k}{ \lfloor (s-k)/2 \rfloor } = \binom{2n+1}{s}
= \dim V(\tfw_s).
\]

Similarly, we can obtain the following identity since every weight multiplicity of weight space for $V(\tfw_s)$ in type $D_n$
is $\binom{n-k-\delta_{n,s}}{ (s-k)/2 }$ for $0 \le k \le s$ and $k \equiv s   \mod{2} $  (see~\cite[Theorem 7.5]{KLO17} or~\cite[Theorem 19.4]{FH91}) and
\begin{equation}
\label{eq:dim_column_type_D}
\dim V(\tfw_s) = \binom{2n -\delta_{s,n}}{s}
\end{equation}
(it may be well-known to
combinatorialists): For $1 \le s \le n$, we have
\[
\sum_{\substack{0 \le k \le s \\ k \equiv s \ {\rm mod} \ 2 }}
2^{k-\delta_{k,n}} \binom{n}{n-k}\binom{n-k-\delta_{n,s}}{ (s-k)/2 } = \binom{2n -\delta_{s,n}}{s}
= \dim V(\tfw_s).
\]

Let us consider the modules $V(\fw_n+\tfw_{n-s})$ $(0\le s \le n-1)$ over $B_n$. Then every weight multiplicity of $V(\fw_n+\tfw_{n-s})$
is a Motzkin triangle number:

\begin{theorem}[{\cite{KLO17}}]
\label{thm: wt Motzkin}
For any weight multiplicity of $V(\fw_n+\tfw_{n-s})$ $(0\le s \le n-1)$ of type $B_n$ is a Motzkin triangle number. More precisely,
\[
\dim V(\fw_n+\tfw_{n-s})_{\fw_n+\tfw_{n-m}} = \Mot_{(m,s)} \quad \text{ for } 0 \le s\le m \le n.
\]
\end{theorem}
Then we can get an identity coming from the representation theory
\[
\dim V(3\fw_n) = 2^n \Cat_{n+1}  = \sum_{i=0}^n \Mot_i \binom{n}{i}2^{n},
\]
which arises from
\begin{subequations}
\begin{align}
\Cat_{n+1} = \sum_{i=0}^n \Mot_i \binom{n}{i}, \label{eq:catalan_from_motzkin} \\
\Cat_n = \sum_{i=0}^s \Rior_i\binom{n}{i}, \label{eq:catalan_from_riordan}
\end{align}
\end{subequations}
proved in~\cite[Sec.~5]{Bernhart97}. A bijective proof of the identity~\eqref{eq:catalan_from_motzkin} has been given in~\cite{Donaghey77} and~\cite[Sec.~3]{DY08}.

By taking a natural $q$-analog of Equation~\eqref{eq:catalan_from_motzkin}, we obtain
\[
\Cat_n(q) = \sum_{i=0}^n q^i \qbinom{n}{i}{q} \Mot'_i(q).
\]
Thus, we can define a new $q$-Motzkin number recursively by
\[
\Mot'_n(q) = q^{-n} \left( \Cat_n(q) - \sum_{i=0}^{n-1} q^i \qbinom{n}{i}{q} \Mot'_i(q) \right),
\qquad\qquad \Mot'_0(q) = 1
\]
We note that these are distinct from the $q$-Motzkin numbers in~\cite{BDLFP98,Cigler99}.

\begin{example}
We have
\begin{gather*}
\Mot'_0(q) = 1,
\qquad
\Mot'_1(q) = q,
\qquad
\Mot'_2(q) = q^2 (q^2 + 1),
\\
\Mot'_3(q) = q^5(q^4 + q^2 + q^1 + 1),
\qquad
\Mot'_4(q) = q^8(q^8 + q^6 + q^5 + 2 q^4 + q^3 + 2 q^2 + 1).
\end{gather*}
\end{example}

\begin{conjecture}
\label{conj:q_natural_motzkin}
We have $\Mot'_n(q) \in \Z_{\geq 0}[q]$.
\end{conjecture}

\begin{proposition}
For $0 \le s \le n$, we have
\begin{subequations} \label{eq:tri_Catalan_from_another triangles}
\begin{align}
\Cat_{(2n+1-s,s)} & = \sum_{i=0}^s \Mot_{(i+n-s,n-s)}\binom{n}{s-i}, \label{eq:tri_Catalan_from_Motzkin} \\
\Cat_{(2n-s,s)} & = \sum_{i=0}^s \Rior_{(i+n-s,n-s)}\binom{n}{s-i}. \label{eq:tri_Catalan_from_Riordan}
\end{align}
\end{subequations}
\end{proposition}

\begin{proof}
We shall use an induction (with appealing to representation theory) for~\eqref{eq:tri_Catalan_from_Motzkin}
and~\eqref{eq:tri_Catalan_from_Riordan} together.
Our assertions for $\Cat_{(2k+1-s,s)}$ and $\Cat_{(2k-s,s)}$ are assumed to be true when $k<n$ and $k-s \ge 0$. The $k=s$ cases of~\eqref{eq:tri_Catalan_from_Motzkin}
and~\eqref{eq:tri_Catalan_from_Riordan} are~\eqref{eq:catalan_from_motzkin} and~\eqref{eq:catalan_from_riordan}, respectively.
Then we have
\begin{align*}
\Cat_{(2n-s,s)} & = \Cat_{(2n-s-1,s)} + \Cat_{(2n-s,s-1)} \allowdisplaybreaks  \\
&= \sum_{i=0}^s \Mot_{(i+n-1-s,n-1-s)}\binom{n-1}{s-i}+ \sum_{i=0}^{s-1} \Mot_{(i+n-s,n-s)}\binom{n-1}{s-1-i} \allowdisplaybreaks  \\
&= \binom{n-1}{s}\Mot_{(n-1-s,n-1-s)}+\sum_{i=1}^{s} \left(\Mot_{(i+n-1-s,n-1-s)}+\Mot_{(i+n-1-s,n-s)}\right)\binom{n-1}{s-i}\allowdisplaybreaks  \\
&= \binom{n-1}{s}\Rior_{(n-s,n-s)}+\sum_{i=1}^{s} \left(\Rior_{(i+n-s,n-s)}+\Rior_{(i+n-1-s,n-s)}\right)\binom{n-1}{s-i} \allowdisplaybreaks  \\
&= \sum_{i=0}^s \Rior_{(i+n-s,n-s)}\binom{n-1}{s-i} + \sum_{i=1}^s \Rior_{(i+n-1-s,n-s)}\binom{n-1}{s-i} \allowdisplaybreaks  \\
&= \sum_{i=0}^s \Rior_{(i+n-s,n-s)}\binom{n-1}{s-i} + \sum_{i=0}^{s-1} \Rior_{(i+n-s,n-s)}\binom{n-1}{s-1-i} \allowdisplaybreaks  \\
& = \Rior_{(n,n-s)}+\sum_{i=0}^{s-1} \Rior_{(i+n-s,n-s)}\left(\binom{n-1}{s-i} +\binom{n-1}{s-1-i} \right) \allowdisplaybreaks  \\
& = \sum_{i=0}^s \Rior_{(i+n-s,n-s)}\binom{n}{s-i},
\end{align*}
where we used Lemma~\ref{lem:Mot_Rio_relation} in the fourth equality.
The proof of $\Cat_{(2n+1-s,s)}$ is similar.
\end{proof}

Now we can obtain an interesting formula for $\dim V(\fw_n + \tfw_s)$ over $B_n$ and interpret Equation~\eqref{eq:tri_Catalan_from_Motzkin} by using Theorem~\ref{thm: wt Motzkin}
and considering Weyl group orbits:

\begin{corollary} We have
\[
\dim V(\fw_n + \tfw_s) = 2^n \Cat_{(2n+1-s,s)} = \sum_{i=0}^s \Mot_{(i+n-s,n-s)}\binom{n}{s-i}2^n
\qquad\qquad (0 \le s \le n).
\]
In particular
\[
\dim V(3\fw_n) = 2^n \Cat_{n+1} = \sum_{i=0}^s \Mot_{i}\binom{n}{i}2^n
\]
$($see also ~\cite[A003645]{OEIS}$)$.
\end{corollary}

\begin{proposition}
\label{prop:B_near_spin_q_char}
We have
\[
\nps(\fw_n + \tfw_s) = \frac{\Cat_{(2n+1-s,s)}(q)}{q^{n+1-s} + 1} \prod_{k=1}^{n+1} (q^k + 1).
\]
\end{proposition}

\begin{proof}
We have
\begin{equation}
\label{eq:B_near_RHS}
\begin{split}
\frac{\Cat_{(2n+1-s,s)}(q)}{q^{n+1-s} + 1} \prod_{k=1}^{n+1} (q^k + 1) & = \left( \prod_{i=1}^{2n+2-s} \dfrac{1-q^{s+i}}{1-q^i} \right) \dfrac{1-q^{n-s+1}}{1-q^{2n+2}} \cdot \dfrac{1}{1-q}\prod_{k=1}^{n+1} (q^k + 1)
%\\ & = \dfrac{[2n+1]! [n-s+1]}{[s]! [2n-s+2]!} \langle n+1 \rangle
\\ & = \dfrac{(q)_{2n+1} (1-q^{n-s+1})}{(q)_s (q)_{2n-s+2}} (-q; q)_{n+1}
\end{split}
\end{equation}
and
\[
\nps(\fw_n + \tfw_s) = \pi_1 \pi_2 \pi_3 \pi_4,
\]
where
\begin{align*}
\pi_1 & = \prod_{i=1}^s \dfrac{1-q^{3+2n-2i+1}}{1-q^{2n-2i+1}} \prod_{i=s+1}^n \dfrac{1-q^{1+2n-2i+1}}{1-q^{2n-2i+1}},
&
\pi_2 & = \prod_{1 \leq i < j \leq s} \dfrac{1-q^{3+2n-i-j+1}}{1-q^{2n-i-j+1}},
\\
\pi_3 & = \prod_{s < i < j \leq n} \dfrac{1-q^{1+2n-i-j+1}}{1-q^{2n-i-j+1}},
&
\pi_4 & = \prod_{1 \leq i \leq s < j \leq n} \dfrac{1-q^{j-i+1}}{1-q^{j-i}} \cdot \dfrac{1-q^{2+2n-i-j+1}}{1-q^{2n-i-j+1}}.
\end{align*}
Define $\qPdf{n} \seteq (1 - q^n) (1 - q^{n-2}) \cdots$, and note that $\qPdf{n}$ to $(q)_n$ is what the double factorial $n!!$ is to the usual factorial $n!$. 
Next, we have
\begin{align*}
%\pi_1 & = \dfrac{[2n+2] [2n] \cdots [2n-2s+4] [2n-2s] [2n-2s-2] \cdots [2]}{[2n-1] [2n-3] \cdots [1]} = \dfrac{[2n+2]!!}{[2n-1]!! [2n-2s+2]},
\pi_1 & = \dfrac{(1-q^{2n+2}) (1-q^{2n}) \cdots (1-q^{2n-2s+4}) (1-q^{2n-2s}) (1-q^{2n-2s-2}) \cdots (1-q^2)}{(1-q^{2n-1}) (1-q^{2n-3}) \cdots (1-q)}
\\ & = \dfrac{\qPdf{2n+2}}{\qPdf{2n-1} (1-q^{2n-2s+2})},
%
%\allowdisplaybreaks  \\ \pi_2 & = \prod_{2 \leq j \leq s} \dfrac{[2n-j+3] \cdots [2n-2j+5]}{[2n-j] \cdots [2n-2j+2]} \prod_{1 \leq j \leq s-1} \dfrac{[2n-j+2] \cdots [2n-2j+3]}{[2n-j-1] \cdots [2n-2j]}
%\allowdisplaybreaks  \\ & = \prod_{1 \leq j \leq s-1} \dfrac{[2n-j+1] [2n-j+1] [2n-j]}{[2n-2j+2] [2n-2j+1] [2n-2j]}
%\allowdisplaybreaks  \\ & = \dfrac{[2n+1]!}{[2n-s+2]!} \cdot \dfrac{[2n]!}{[2n-s+1]!} \cdot \dfrac{[2n-2s+2]!}{[2n]!} \cdot \dfrac{[2n-1]!}{[2n-s]!} \cdot \dfrac{[2n-2s]!!}{[2n-2]!!},
%\allowdisplaybreaks  \\ & = \dfrac{[2n+1]!}{[2n-s+2]!} \cdot \dfrac{[2n-2s+2]!}{[2n-s+1]!} \cdot \dfrac{[2n-1]!}{[2n-s]!} \cdot \dfrac{[2n-2s]!!}{[2n-2]!!},
\allowdisplaybreaks  \\ \pi_2 & = \prod_{2 \leq j \leq s} \dfrac{(1-q^{2n-j+3}) \cdots (1-q^{2n-2j+5})}{(1-q^{2n-j}) \cdots (1-q^{2n-2j+2})} \prod_{1 \leq j \leq s-1} \dfrac{(1-q^{2n-j+2}) \cdots (1-q^{2n-2j+3})}{(1-q^{2n-j-1}) \cdots (1-q^{2n-2j})}
\allowdisplaybreaks  \\ & = \prod_{1 \leq j \leq s-1} \dfrac{(1-q^{2n-j+1}) (1-q^{2n-j+1}) (1-q^{2n-j})}{(1-q^{2n-2j+2}) (1-q^{2n-2j+1}) (1-q^{2n-2j})}
\allowdisplaybreaks  \\ & = \dfrac{(q)_{2n+1}}{(q)_{2n-s+2}} \cdot \dfrac{(q)_{2n}}{(q)_{2n-s+1}} \cdot \dfrac{(q)_{2n-2s+2}}{(q)_{2n}} \cdot \dfrac{(q)_{2n-1}}{(q)_{2n-s}} \cdot \dfrac{\qPdf{2n-2s}}{\qPdf{2n-2}},
\allowdisplaybreaks  \\ & = \dfrac{(q)_{2n+1} (q)_{2n-2s+2} (q)_{2n-1} \qPdf{2n-2s}}{(q)_{2n-s+2} (q)_{2n-s+1} (q)_{2n-s} \qPdf{2n-2}},
%
%\allowdisplaybreaks  \\ \pi_3 & = \prod_{s+2 \leq j \leq n} \dfrac{[2n-s-j+1] \cdots [2n-2j+3]}{[2n-s-j] \cdots [2n-2j+2]} = \prod_{s+2 \leq j \leq n} \dfrac{[2n-s-j+1]}{[2n-2j+2]}
% \\ & = \dfrac{[2n-2s-1]!}{[n-s]! [2n-2s-2]!!},
\allowdisplaybreaks  \\ \pi_3 & = \prod_{s+2 \leq j \leq n} \dfrac{(1-q^{2n-s-j+1}) \cdots (1-q^{2n-2j+3})}{(1-q^{2n-s-j}) \cdots (1-q^{2n-2j+2})} = \prod_{s+2 \leq j \leq n} \dfrac{1-q^{2n-s-j+1}}{1-q^{2n-2j+2}}
 \\ & = \dfrac{(q)_{2n-2s-1}}{(q)_{n-s} \qPdf{2n-2s-2}},
%
%\allowdisplaybreaks  \\ \pi_4 & = \prod_{s+1 \leq j \leq n} \dfrac{[j] \cdots [j-s+1]}{[j-1] \cdots [j-s]} \cdot \dfrac{[2n-j+2] \cdots [2n-j-s+3]}{[2n-j] \cdots [2n-j-s+1]}
%\allowdisplaybreaks  \\ & = \dfrac{[n]!}{[s]! [n-s]!} \prod_{s+1 \leq j \leq n} \dfrac{[2n-j+2] [2n-j+1]}{[2n-j-s+2] [2n-j-s+1]}
%\allowdisplaybreaks  \\ & = \dfrac{[n]!}{[s]! [n-s]!} \cdot \dfrac{[2n-s+1]! [n-s+1]! [2n-s]! [n-s]!}{[n+1]! [2n-2s+1]! [n]! [2n-2s]!}
%\allowdisplaybreaks  \\ & = \dfrac{[2n-s+1]! [n-s+1]! [2n-s]!}{[s]! [n+1]! [2n-2s+1]! [2n-2s]!}.
\allowdisplaybreaks  \\ \pi_4 & = \prod_{s+1 \leq j \leq n} \dfrac{(1-q^j) \cdots (1-q^{j-s+1})}{(1-q^{j-1}) \cdots (1-q^{j-s})} \cdot \dfrac{(1-q^{2n-j+2}) \cdots (1-q^{2n-j-s+3})}{(1-q^{2n-j}) \cdots (1-q^{2n-j-s+1})}
\allowdisplaybreaks  \\ & = \dfrac{(q)_n}{(q)_s (q)_{n-s}} \prod_{s+1 \leq j \leq n} \dfrac{(1-q^{2n-j+2}) (1-q^{2n-j+1})}{(1-q^{2n-j-s+2}) (1-q^{2n-j-s+1})}
\allowdisplaybreaks  \\ & = \dfrac{(q)_n}{(q)_s (q)_{n-s}} \cdot \dfrac{(q)_{2n-s+1} (q)_{n-s+1} (q)_{2n-s} (q)_{n-s}}{(q)_{n+1} (q)_{2n-2s+1} (q)_{n} (q)_{2n-2s}}
\allowdisplaybreaks  \\ & = \dfrac{(q)_{2n-s+1} (q)_{n-s+1} (q)_{2n-s}}{(q)_{s} (q)_{n+1} (q)_{2n-2s+1} (q)_{2n-2s}}.
\end{align*}
Thus, we can simplify the product $\pi_1 \pi_2 \pi_3 \pi_4$ to obtain
\begin{equation}
\label{eq:B_near_LHS_simplified}
%\nps(\fw_n + \tfw_s) = \dfrac{[2n+1]! [2n-2]!! [n-s+1]}{[2n-s+2]! [s]! [n+1]!}
% = \dfrac{\langle n+1 \rangle! [2n+1]! [n-s+1]}{[2n-s+2]! [s]!}.
\nps(\fw_n + \tfw_s) = \dfrac{(q)_{2n+1} \qPdf{2n-2} (1 - q^{n-s+1})}{(q)_{2n-s+2} (q)_{s} (q)_{n+1}}
= \dfrac{(-q;q)_{n+1} (q)_{2n+1} (1 - q^{n-s+1})}{(q)_{2n-s+2} (q)_{s}}.
\end{equation}
Hence, the claim follows from noting that~\eqref{eq:B_near_LHS_simplified} and~\eqref{eq:B_near_RHS} are equal.
\end{proof}

\begin{corollary}
\label{cor:catalan_q2}
For type $B_n$, we have
\[
C_{n+1}(q) \prod_{k=1}^n (1 + q^k) = \nps(3\fw_n).
\]
\end{corollary}

\begin{proof}
Using
\[
%\Cat_n(q) = \dfrac{\langle n \rangle}{\langle 1 \rangle} \Cat_{(n,n-1)}(q),
\Cat_n(q) = \dfrac{1+ q^n}{1 + q} \Cat_{(n,n-1)}(q),
\]
the claim follows from Proposition~\ref{prop:B_near_spin_q_char} with $s = n$.
\end{proof}

%\begin{proof}
%From~\cite[Eq.~(3.21)]{BKW16}, we have
%\[
%\nps(3\fw_n) = \prod_{i=1}^n \frac{1 - q^{3+2n-2i+1}}{1 - q^{2n-2i+1}} = \prod_{1 \leq i < j \leq n} \frac{1 - q^{3+2n-i-j+1}}{1 - q^{2n-i-j+1}}.
%\]
%Next, we have
%\[
%C_n(q) = \frac{1 - q}{1 - q^{n+1}} \times \frac{(q)_{2n}}{(q)_n^2} = \frac{(1 - q^{n+2}) \dotsm (1 - q^{2n})}{(1 - q^2) \dotsm (1 - q^n)}
%= \prod_{i=1}^{n-1} \frac{1 - q^{2n+1-i}}{1 - q^{n+1-i}}.
%\]
%where $(q)_n = \prod_{k=1}^n (1 - q^k)$ is the $q$-Pochhammer symbol.
%
%Set $[k]=(1-q^k)$, $\langle k \rangle =(1+q^k)$, $[k]!=\prod_{i=1}^k [i]$, $\langle k \rangle !=\prod_{i=1}^k \langle i \rangle$, $[2k]!!=\prod_{i=1}^k [2k][2k-2]\cdots [2]$ and $[2k-1]!!=\prod_{i=1}^k [2k-1][2k-3]\cdots [1]$
%Then
%\begin{align*}
%\nps(3\fw_n) &= \dfrac{[2n+2]!!}{[2] \times [2n-2]!!} \prod_{j=1}^{n-1} \dfrac{[2n-j+2]!}{[2n-2j+2]!}\dfrac{[2n-2j-1]!}{[2n-j-1]!} \\
%&= \dfrac{[2n+2]}{[2n-1]![2n-1]!!} \dfrac{\Phi(2n-1)}{\Phi(2n)} \times \dfrac{F(2n+1)}{F(2n-2)}\times \dfrac{F(n-1)}{F(n+2-1)} \times [2n]!!
%\end{align*}
%where $F(k)=\prod_{i=1}^k [i]!$, $\Phi(2k)=[2k]![2k-2]!\cdots[2]!$ and $\Phi(2k-1)=[2k-1]![2k-3]!\cdots[1]!$.
%Since $\dfrac{\Phi(2n-1)}{\Phi(2n)}  \times [2n]!! =1 $, we have
%\begin{align*}
%\nps(\fw_n) &= \dfrac{[2n+2]![2n]!}{[2n-1]!![n]![n+1]![n+2]!}.
%\end{align*}
%On the other hand,
%\[ C_{n+1}(q) \prod_{k=1}^n (1 + q^k) = \dfrac{[2n+2]!}{[n+2]![n+1]!} \times \langle n \rangle!.  \]
%Then one can easily see that
%\[ \dfrac{\nps(\fw_n) }{ C_{n+1}(q) \prod_{k=1}^n (1 + q^k)} =1. \]
%\end{proof}

We give a natural $q$-analog of Equation~\eqref{eq:tri_Catalan_from_Motzkin} in parallel to $\Mot'_n(q)$ from Equation~\eqref{eq:catalan_from_motzkin}.
We define recursively
\[
\Mot'_{(n,n-s)}(q) = \Cat_{(2n+1-s,s)}(q) - \sum_{i=0}^{s-1} q^{n(s-i)} \qbinom{n}{s-i}{q} \Mot'_{(i+n-s,n-s)}(q),
\qquad\quad
\Mot'_{(0,0)}(q) = 1,
\]
which implies that
\[
\Cat_{(2n+1-s,s)}(q) = \sum_{i=0}^s q^{n(s-i)} \qbinom{n}{s-i}{q} \Mot'_{(i+n-s,n-s)}(q).
\]
Note that $\Mot'_{(n,0)}(q) \neq q^v \Mot'_n(q)$ for some $v \in \ZZ$. For example, we have
\begin{align*}
\Mot'_{(4,0)}(q) & = q^8 \left( q^8 + q^6 + q^5 + 2 q^4 + q^3 + 2 q^2 + 1 \right),
\\ \Mot'_4(q) & = q^6 + q^5 + 2 q^4 + 2 q^3 + q^2 + q + 1,
\end{align*}
and note that $q^{-8} \Mot'_{(4,0)}$ and $\Mot'_4(q)$ are irreducible polynomials.
Similar to Conjecture~\ref{conj:q_natural_motzkin}, we conjecture that these $q$-Motzkin triangle numbers are positive.

\begin{conjecture}
\label{conj:q_natural_motzkin_tri}
We have $\Mot'_{(n,n-s)}(q) \in \Z_{\geq 0}[q]$.
\end{conjecture}

\begin{example}
We have
\[
[ \Mot'_{(n,r)} ]_{n,r=0}^3 =
\begin{bmatrix}
 1 \\
 1 & 1 \\
 q + 1 & q + 1 & 1 \\
 q^3 + q^2 + q + 1 & q^3 + 2q^2 + q + 1 & q^2 + q + 1 & 1
% q^6 + q^5 + 2q^4 + 2q^3 + q^2 + q + 1
%&  q^6 + 2q^5 + 2q^4 + 3q^3 + 2q^2 + q + 1
%&  q^5 + 2q^4 + 2q^3 + 2q^2 + q + 1
%&  q^3 + q^2 + q + 1
\end{bmatrix}.
\]
\end{example}

\begin{theorem}[{\cite{KLO17}}]
\label{thm: wt Riordan}
For any weight multiplicity of $V(\fw_{n+1}+\tfw_{n+1-s})$ in type $D_{n+1}$ is a Riordan triangle number. More precisely,
\[
\dim V(\fw_{n+1}+\tfw_{n+1-s})_{\mu} = \Rior_{(m+1,s)} \quad \text{ for } 0 \le s\le m \le n,
\]
where
\[
\mu = \begin{cases}
\fw_{n+1}+\tfw_{n-m} & \text{if } m \not\equiv s \mod{2}, \\
\fw_{n}+\tfw_{n-m} & \text{if } m \equiv s \mod{2}.
\end{cases}
\]
\end{theorem}

Now we can also obtain an interesting formula for $\dim V(\fw_n + \tfw_s)$ over $D_n$ and interpret Equation~\eqref{eq:tri_Catalan_from_Riordan} by using Theorem~\ref{thm: wt Riordan}
and considering Weyl group orbits:

\begin{corollary}
In type $D_n$, we have
\[
\dim V(\fw_n + \tfw_s) = 2^{n-1} \Cat_{(2n-s,s)} = \sum_{i=0}^s \Rior_{(i+n-s,n-s)}\binom{n}{s-i}2^{n-1}
\qquad\qquad (0 \le s \le n).
\]
In particular
\[
\dim V(3\fw_n) = 2^{n-1} \Cat_{n} = \sum_{i=0}^s \Rior_{i}\binom{n}{i}2^{n-1}.
\]
\end{corollary}

Similarly using Equation~\eqref{eq:tri_Catalan_from_Riordan}, we define recursively
\[
\Rior'_{(n,n-s)}(q) = \Cat_{(2n-s,s)}(q) - \sum_{i=0}^{s-1} q^{(n-1)(s-i)} \qbinom{n}{s-i}{q} \Rior'_{(i+n-s,n-s)}(q),
\qquad\quad
\Rior'_{(0,0)}(q) = 1,
\]
which implies that
\[
\Cat_{(2n-s,s)}(q) = \sum_{i=0}^s q^{(n-1)(s-i)} \qbinom{n}{s-i}{q} \Rior'_{(i+n-s,n-s)}(q).
\]
Similar to Conjecture~\ref{conj:q_natural_motzkin_tri}, we conjecture that these $q$-Riordan triangle numbers are positive.

\begin{conjecture}
We have $\Rior'_{(n,n-s)}(q) \in \Z_{\geq 0}[q]$.
\end{conjecture}

\begin{example}
We have
\[
[ \Rior'_{(n,r)} ]_{n,r=0}^4 =
\begin{bmatrix}
 1 \\
 0 & 1 \\
 1 & 1 & 1 \\
 1 & q^2 + q + 1 & q + 1 & 1 \\
 q^4 + q^2 + 1 & q^4 + q^3 + 2q^2 + q + 1 & q^4 + q^3 + 2q^2 + q + 1  & q^2 + q + 1 & 1
\end{bmatrix}.
\]
\end{example}

% ========
\subsection{Principal specializations by branching rules}

Next we prove some principal specializations by using the branching rules for the representations. This technique has been used before, \textit{e.g.},~\cite{Rains06}, and so the subsequent proofs may also be well-known to experts.

\begin{proposition}
\label{prop:likely_stanley}
In type $A_n$, we have
\[
\ps(\varpi_n + \varpi_k) = q^{\binom{n}{2} + \binom{k}{2}} [n-k+1]_q \qbinom{n+2}{k}{q},
\]
where we consider $\varpi_0 = 0$.
\end{proposition}

\begin{proof}
The case for $A_1$ is a straightforward computation. We proceed by induction.
We have
\begin{equation}
\label{eq:branching_target}
\begin{split}
\nps(\varpi_n + \varpi_k) & = q^{\eta_{n,k}} [n-k+1]_q \qbinom{n+2}{k}{q}
\\ & = q^{\eta_{n,k}} [n-k+1]_q \left( \qbinom{n+1}{k-1}{q} + q^{n+2-k} \qbinom{n+1}{k-1}{q} \right),
\end{split}
\end{equation}
where $\eta_{n,k} = \binom{n}{2} + \binom{k}{2}$.
Next, by the branching rule $A_n \to A_{n-1}$, we have
\[
B(\varpi_n + \varpi_k) \to B(\varpi_n + \varpi_k) \oplus B(\varpi_n + \varpi_{k-1}) \oplus B(\varpi_{n-1} + \varpi_k) \oplus B(\varpi_{n-1} + \varpi_{k-1})
\]
Taking the principal specializations of the above, we have
\begin{equation}
\label{eq:ch_spec_branching}
q^{\binom{n}{2}} \ps(\varpi_k) + q^{\binom{n}{2}} q^n \ps(\varpi_{k-1}) + q^n \ps(\varpi_{n-1} + \varpi_k) + q^n q^n \ps(\varpi_{n-1} + \varpi_{k-1}),
\end{equation}
where the $q^n$ factors are from fixing a box with an $n+1$ and $q^{\binom{n}{2}}$ from the column $1 2 \cdots n$.

In type $A_{n-1}$, we have
\begin{equation}
\label{eq:spec_char_column_type_A}
\ps(\varpi_k) = q^{\binom{k}{2}} \qbinom{n}{k}{q}
\end{equation}
by either applying branching rules $A_k \searrow A_{k-1}$ or~\cite[Thm.~7.21.2]{ECII}.
Substituting~\eqref{eq:spec_char_column_type_A} and by our induction hypothesis into~\eqref{eq:ch_spec_branching}, we obtain
\begin{align*}
\ps(\varpi_k) & = q^{\binom{n}{2}} q^{\binom{k}{2}} \qbinom{n}{k}{q} + q^n q^{\binom{n}{2}} q^{\binom{k-1}{2}} \qbinom{n}{k-1}{q} + q^n  q^{\eta_{n-1,k}} [n-k]_q \qbinom{n+1}{k}{q}
\\ & \hspace{20pt} + q^n q^n  q^{\eta_{n-1,k-1}} [n-k+1]_q \qbinom{n+1}{k-1}{q}.
\end{align*}
Thus, by using $\binom{n-1}{2} + (n-1) = \binom{n}{2}$, we obtain
\begin{equation}
\label{eq:branch_simplified1}
\begin{aligned}
\ps(\varpi_k) & =q^{\eta_{n,k}} \qbinom{n}{k}{q} + q^{n-k} q^{\eta_{n,k}} \qbinom{n}{k-1}{q} + q q^{\eta_{n,k}}[n-k]_q \qbinom{n+1}{k}{q}
\\ & \hspace{20pt} + q^{n+2-k} q^{\eta_{n,k}} [n-k+1]_q \qbinom{n+1}{k-1}{q}.
\end{aligned}
\end{equation}
Next, we apply the $q$-Pascal triangle identity~\eqref{eq:pascal2} to~\eqref{eq:branch_simplified1} and factoring out $q^{\eta_{n,k}}$ to obtain:
\[
q^{\eta_{n,k}} \left( \qbinom{n+1}{k}{q} + q [n-k]_q \qbinom{n+1}{k}{q} + q^{n+2-k}  [n-k+1]_q \qbinom{n+1}{k-1}{q} \right),
\]
Next, note that $[n-k+1]_q = q [n-k]_q + 1$, and so we have
\[
q^{\eta_{n,k}} \left( [n-k+1]_q \qbinom{n+1}{k}{q} + q^{n+2-k}  [n-k+1]_q \qbinom{n+1}{k-1}{q} \right),
\]
which equals~\eqref{eq:branching_target} as desired.
\end{proof}

Proposition~\ref{prop:likely_stanley} can alternatively be proven using~\cite[Thm.~7.21.2]{ECII}. Note that our proof is representation theoretic as we are using the branching rule $A_n \searrow A_{n-1}$.

When $q=1$ of Proposition~\ref{prop:likely_stanley}, this corresponds to negative of the reverse of the triangular array of~\cite[A055137]{OEIS} with removing the $0$ and $1$ portions.

\begin{proposition}
\label{prop:q_binomial_repr}
In type $B_n$, we have
\[
\nps(\tfw_i) = \qbinom{2n+1}{i}{q}.
\]
\end{proposition}

\begin{proof}
Note that $\nps(\tfw_i) = q^{\eta_{n,i}} \ps(\tfw_i)$, where $\eta_{n,i} = \binom{n+1}{2}-\binom{n+1-i}{2}$.
The claim holds for $B_2$ (and $B_1 = A_1$) by a straightforward computation.
Applying the $q$-Pascal's triangle relation~\eqref{eq:pascal1} twice, we obtain
\[
\qbinom{2n+1}{i}{q} = q^i q^i \qbinom{2n-1}{i}{q} + q^i \qbinom{2n-1}{i-1}{q} + q^{i-1} \qbinom{2n-1}{i-1}{q} + \qbinom{2n-1}{i-2}{q}.
\]
We show that this agrees with the principal specialization under the branching rule $B_n \to B_{n-1}$. For some crystal $B$ and $T \in B$, let
\[
\wt_{\ps}(T) = \wt(T) \bigr|_{x_1^{\pm1}=q^{\pm1}, \dotsc, x_n^{\pm1}=q^{\pm n}}
\]
be the principal specialization of the weight of $T$, and note that $\ps(B) = \sum_{T \in B} \wt_{\ps}(T)$. We denote the fundamental weights of type $B_{n-1}$ by $\{\zeta_i \mid i \in I_{B_{n-1}} \}$.

Assume $i < n$.
We first show that
\begin{equation}
\label{eq:binom_principal_spec_version}
\begin{split}
\ps(\fw_i) & = q^i \ps(\zeta_i) + q^{-(n-1)} q^{i-1} \ps(\zeta_{i-1}) + q^{-n} q^{i-1} \ps(\zeta_{i-1})
\\ & \hspace{20pt} + q^{-(n-1)} q^{-n} q^{i-2} \ps(\zeta_{i-2}).
\end{split}
\end{equation}
For this proof, we equate $\ok \equiv -k$ and we write the single column KN tableaux as sets. We define a map $\phi_m \colon B(\zeta_m) \to B(\fw_k)$ by $\{t_1, \dotsc, t_m\} \mapsto \{t_1+1, \dotsc, t_m+1\}$. Note that under the principal specialization, we have
$
\wt_{\ps}\bigl( \phi_m(T) \bigr) = q^m \wt_{\ps}(T).
$
Note that $\on$ and $\overline{n-1}$ cannot appear in the image under $\phi_m$, but every other set of size $m$ appears.
Thus, we can extend this to a weight preserving bijection by adding $\on$ and/or $\overline{n-1}$ to $\phi_m$ to obtain Equation~\eqref{eq:binom_principal_spec_version}.

Next, we rewrite Equation~\eqref{eq:binom_principal_spec_version} as
\begin{align*}
q^{\eta_{n,i}} \ps(\fw_i) & = q^i q^i q^{\eta_{n-1,i}} \ps(\zeta_i) + q^i q^{\eta_{n-1,i-1}} \ps(\zeta_{i-1})
\\ & \hspace{20pt} + q^{i-1} q^{\eta_{n-1,i-1}} \ps(\zeta_{i-1}) + q^{\eta_{n-2,i-2}} \ps(\zeta_{i-2}),
\end{align*}
and the claim follows by induction.

Note that for $i = n$, the branching rule $\pm$-diagram instead has a column with a $0$ instead of a blank column.
\end{proof}

\begin{proposition}
\label{prop:nps_type_B_spin}
For $\g$ of type $B_n$, we have
\[
\nps(\fw_n) = q^{\binom{n+1}{2}} \ps(\fw_n) = \prod_{k=1}^n (1 + q^k).
\]
\end{proposition}

\begin{proof}
Note that the coefficient of $q^m$ occurring in $\prod_{k=1}^n (1 + q^k)$ is precisely the number of strict partitions of $m$. For any element $(s_1, \dotsc, s_n) \in B(\fw_n)$, we construct a corresponding strict partition by $\{n+1-i \mid s_i = +\}$, and it is clear this is a bijection. Since an $i \in S \in B(\fw_n)$ contributes weight $\epsilon_i/2$ and $i \notin S$ contributes $-\epsilon_i/2$, the claim follows.
\end{proof}

\begin{remark}
Proposition~\ref{prop:q_binomial_repr} and Proposition~\ref{prop:nps_type_B_spin} are well-known consequences of, \textit{e.g.}, Proposition~\ref{prop:Bn_ps}.
Furthermore, in~\cite[Sec.~3]{Hughes77}, Hughes showed that
\[
\dim_q V(\fw_n) = \prod_{k=1}^n (1 + q^k),
\]
where $\dim_q V$ is the $q$-dimension (\textit{e.g.}, see~\cite[Sec.~10]{kac90}).
\end{remark}

We note that the normalized principal specialization of the characters of type $D_n$ are \emph{not} $q$-binomial coefficients. For example
\begin{align*}
\nps(\tfw_2) & = q^{14} + q^{13} + 2 q^{12} + q^{11} + 2 q^{10} + 2 q^9 + 3 q^8 \\
   & \hspace{20pt} + 4 q^7 + 3 q^6 + 2 q^5 + 2 q^4 + q^3 + 2 q^2 + q + 1, \\
\qbinom{2\cdot4}{2}{q} & = q^{12} + q^{11} + 2 q^{10} + 2 q^9 + 3 q^8 + 3 q^7 \\
  & \hspace{20pt} + 4 q^6 + 3 q^5 + 3 q^4 + 2 q^3 + 2 q^2 + q + 1,
\end{align*}
However, by applying the branching rule $B_n \to D_n$ corresponding to $SO(2n) \to SO(2n-1)$, we obtain the following relation on normalized principal specializations.

\begin{theorem}
Let $\{\zeta_i \mid i \in I_{D_n} \}$ denote the fundamental weights of type $D_n$.
For $s < n$,
\[
\nps(\tfw_s) = \nps(\widetilde{\zeta}_s) + q^{n-s-1} \nps(\widetilde{\zeta}_{s-1}),
\]
where we consider $\tfw_0 = 0$. Moreover, we have
\[
\nps(\tfw_n)
 = q \nps(\widetilde{\zeta}_{n-1}) + q^{1+(-1)^n} \nps(\widetilde{\zeta}_n^{+}) + q^{1-(-1)^n} \nps(\widetilde{\zeta}_n^{-}),
 \]
%\[
%q^{\eta_n} \ps_B(\tfw_n)
% = q \cdot q^{\eta_{n-1}} \ps_D(\tfw_{n-1}) + q^{1+(-1)^n} \cdot q^{\eta^+} \ps_D(\tfw_n^{+}) + q^{1-(-1)^n} \cdot q^{\eta^-} \ps_D(\tfw_n^{-}),
%\]
%where
%\[
%\eta_s = \binom{n+1}{2} - \binom{n+1-s}{2},
%\qquad\qquad
%\eta^{\pm} = \eta_{n-1} \pm (-1)^n,
%\]
where $\widetilde{\zeta}_n^+ = 2\zeta_n$ and $\widetilde{\zeta}_n^- = 2\zeta_{n-1}$.
\end{theorem}

\begin{proof}
We consider the KN tableaux representation. We replace pairs
\[
\young(0,0) \mapsto \young(\on,n)
\]
occurring in the KN tableaux of type $B_n$. This is a bijection since such pairs cannot occur in a type $B_n$ tableaux (but they can in a type $D_n$ tableaux). Thus, we can take the KN tableaux with an even number of $0$ entries and consider them as type $D_n$ tableaux. For those with an odd number of $0$ entries, we replace all but the top most $0$ and then remove this entry. The result is a height $s-1$ type $D_N$ column.
Note that for $\eta_s = \binom{n+1}{2} - \binom{n+1-s}{2}$, we have $\nps(\tfw_s) = q^{\eta_s} \ps(\tfw_s)$.
Moreover, note that $q^{n-s-1} q^{\eta_{s-1}} = q^{\eta_s}$. Thus, the claim follows for $s < n$.

In the case $s = n$, then we also have to consider the parity by the number of barred entries. By further subdividing the height $n$ columns (which must have an even number of $0$ entries) into those with an even or odd number of barred entries. Finally, note that $\nps(\widetilde{\zeta}_b^{\pm}) = q^{\eta^{\pm}} \ps(\widetilde{\zeta}_n^{\pm})$ for $\eta^{\pm} = \eta_{n-1} \pm (-1)^n$, and the claim follows.
\end{proof}

In contrast, there does not appear to be a simple formula for branching $D_{n+1} \to B_n$ (corresponding to $SO(2n+1) \to SO(2n)$). For example, from type $D_4 \to B_3$, the branching rule is $B(\widetilde{\zeta}_s) \mapsto B(\tfw_s) + B(\tfw_{s-1})$ for $s < n$. Moreover, we have the following:
\begin{align*}
q^7 \ps_D(\tfw_2) & = q^{14} + q^{13} + 2 q^{12} + q^{11} + 2 q^{10} + 2 q^9 + 3 q^8
\\ & \hspace{20pt} + 4 q^7 + 3 q^6 + 2 q^5 + 2 q^4 + q^3 + 2 q^2 + q + 1
 \allowdisplaybreaks  \\ & = [2]_{q^2}^2 (q^{10} + q^9 - q^7 + q^6 + 3 q^5 + q^4 - q^3 + q + 1),
\allowdisplaybreaks  \\ q^3 \ps_B(\tfw_1) & = q^6 + q^5 + q^4 + q^3 + q^2 + q + 1 = [7]_q,
\allowdisplaybreaks  \\ q^5 \ps_B(\tfw_2) & = q^{10} + q^9 + 2 q^8 + 2 q^7 + 3 q^6 + 3 q^5 + 3 q^4 + 2 q^3 + 2 q^2 + q + 1
 \allowdisplaybreaks  \\ & = (q^2 - q + 1) [3]_q [7]_q = [3]_{q^2} [7]_q,
\end{align*}
There does not appear to be a relation between these polynomials that specializes to the branching rule at $q = 1$.

% ========
\subsection{Other identities}

In this section, we give some other identities that do not appear to come from a $\mathfrak{so}_n$ character identity.

The following proposition is a specialization of~\cite[Thm.~2.1]{Okada09}, which is a $\mathfrak{o}_{2n+1}$ character identity. The proof we give is based on~\cite{Cigler09}.

\begin{proposition}
\label{prop:determinant_double_spin_B} \cite[Lemma 3.3]{BKW16}
Consider the Hankel matrix
\[
BH_{r,n} \seteq \left[ \binom{2(n+i+j)+1}{n+i+j} \right]_{i,j=0}^{r-1}
\]
For type $B_n$, we have
\[
\dim V(r \tfw_n) = \det BH_{r,n}.
\]
\end{proposition}

\begin{proof}
By using~\cite[Eq.~(43)]{Cigler09} and~\cite[Eq.~(42)]{Cigler09}, we have
\begin{align*}
%\det \left[ \binom{2(n+i+j)+1}{n+i+j} \right]_{i,j=0}^{r-1}
\det BH_{r,n}
& = \frac{1}{2^r} \det \left[ \binom{2i+2j+2(n+1)}{i+j+(n+1)} \right]_{i,j=0}^{r-1}
\allowdisplaybreaks  \\ & = \frac{1}{2^r} 2^{r-1+(n+1)} \prod_{j=0}^n \prod_{i=1}^j \frac{2r + j + i - 1}{j+i} = 2^n \prod_{j=0}^n \prod_{i=1}^j \frac{2r + j + i - 1}{j+i}
\allowdisplaybreaks  \\ & = 2^n \prod_{1 \leq i \leq j \leq n} \frac{2r + j + i - 1}{j+i}.
\end{align*}
From Proposition~\ref{prop:Bn_ps} at $q=1$ (see also~\cite{Okada09}), we have
\[
\dim V(r\tfw_n) = \prod_{1\leq i \leq j \leq n} \frac{2r+i+j-1}{i+j-1}.
\]
By direct computation, we have
\[
\prod_{1\leq i \leq j \leq n} \frac{2r+i+j-1}{i+j-1} \cdot \frac{i+j}{i+j}
%= \prod_{1 \leq i \leq j \leq n} \frac{i+j}{i+j-1} \times \prod_{1 \leq i \leq j \leq n} \frac{2r + j + i - 1}{j+i}
%= \prod_{i=1}^n \frac{i+n}{2i-1} \cdot \frac{2i}{2i} \times \prod_{1 \leq i \leq j \leq n} \frac{2r + j + i - 1}{j+i}
= 2^n \prod_{1 \leq i \leq j \leq n} \frac{2r + j + i - 1}{j+i},
\]
and the claim holds.
\end{proof}

We remark that for a fixed $n$, the sequence $\bigl( \dim V(r \fw_n) \bigr)_{r=1}^{\infty}$ corresponds to the $n$-th diagonal of~\cite[A102539]{OEIS}.

%We note that the na\"ive method of extending Proposition~\ref{prop:determinant_double_spin_B} to other height columns does not work. For $k = 2$ in type $B_3$, we have
%\[
%-126 = \det \left[ \binom{2(3+i+j)+1}{k+i+j} \right]_{i,j=0}^{1} \neq \dim V(2 \tfw_k) = 168.
%\]

There is not a natural $q$-analog of Proposition~\ref{prop:determinant_double_spin_B} because of the following. Consider $r = 2$ and $n = 2$ for the following determinant
\[
\left\lvert \begin{matrix}
\qbinom{2n+1}{n}{q} & \qbinom{2(n+1)+1}{n+1}{q} \\
\qbinom{2(n+1)+1}{n+1}{q} & \qbinom{2(n+2)+1}{n+2}{q}
\end{matrix} \right\rvert
= d_+ - d_-,
\]
where
\begin{align*}
d_+ & = \qbinom{5}{2}{q} \qbinom{9}{4}{q}
\\ & = q^{26} + 2q^{25} + 5q^{24} + 9q^{23} + 16q^{22} + 24q^{21} + 36q^{20} + 48q^{19} + 63q^{18} + 76q^{17}
\\ & \hspace{20pt} + 90q^{16} + 99q^{15} + 107q^{14} + 108q^{13} + 107q^{12} + 99q^{11} + 90q^{10} + 76q^{9}
\\ & \hspace{20pt} + 63q^{8} + 48q^{7} + 36q^{6} + 24q^{5} + 16q^{4} + 9q^{3} + 5q^{2} + 2q + 1 \allowdisplaybreaks \\
d_-  & = \qbinom{7}{3}{q} \qbinom{7}{3}{q}
\\ & = q^{24} + 2q^{23} + 5q^{22} + 10q^{21} + 18q^{20} + 28q^{19} + 43q^{18} + 58q^{17} + 76q^{16} + 92q^{15}
\\ & \hspace{20pt} + 106q^{14} + 114q^{13} + 119q^{12} + 114q^{11} + 106q^{10} + 92q^{9} + 76q^{8} + 58q^{7}
\\ & \hspace{20pt} + 43q^{6} + 28q^{5} + 18q^{4} + 10q^{3} + 5q^{2} + 2q + 1
\end{align*}
Note that $d_+ - q^k d_-$ contains a negative term for any $k \in \Z$. Hence, such a determinant form with some power of $q$ in each entry cannot result in the principal specialization of $V(r \tfw_n)$.

\begin{remark}
As an alternative approach to proving Proposition~\ref{prop:determinant_double_spin_B}, we can consider using~\cite[Thm.~11]{Krat99}.
Thus, we have that
\[
\det \left[ \binom{2(n+i+j)+1}{n+i+j} \right]_{i,j=0}^{r-1} = \mu_0^r b_1^{r-1} b_2^{r-2} \dotsm b_{r-2}^2 b_{r-1},
\]
where we consider the generating function and J-fraction
\begin{align*}
\sum_{k=0}^{\infty} \binom{2k+2n+1}{k+n} x^k & = \frac{1}{\sqrt{1-4x}} \left( \frac{1 - \sqrt{1-4x}}{2x} \right)^{2n+1}
\\ & =  \frac{\mu_0}{\displaystyle 1 + a_0 x - \frac{b_1 x^2}{\displaystyle 1 + a_1 x -\frac{b_2 x^2}{\displaystyle 1 + a_2 x - \cdots}}}
\end{align*}
For $n = 1$, we have $\mu_0 = 3$ and expect $b^{(1)}_i = (2i+5)(2i+1) / (2i+1)^2$.
However, the validity of Proposition~\ref{prop:determinant_double_spin_B} implies that $\mu_0 = 10$ and
\[
\left( b^{(2)}_i \right)_{i=1}^{13} = \left( \frac{7}{20}, \frac{24}{35}, \frac{275}{336}, \frac{728}{825}, \frac{525}{572}, \frac{2992}{3185}, \frac{5187}{5440}, \frac{2800}{2907}, \frac{12903}{13300}, \frac{19000}{19481}, \frac{9009}{9200}, \frac{37352}{38025}, \frac{50375}{51156} \right)
\]
for $n = 2$. There does not appear to be a simple expression for $\left( b^{(n)}_i \right)_{i=1}^{\infty}$. However, we believe $b_i^{(n)} < 1$ and $\lim_{i \to \infty} b_i^{(n)} = 1$, likely at an exponential rate, for all $n$.
\end{remark}

\begin{corollary}
For type $D_{n+1}$, we have
\[
\det \left[\binom{2(n+i+j)+1}{n+i+j}\right]_{i,j=0}^{r-1} = \dim V(2r \fw_n) = \dim V(2r \fw_{n+1}).
\]
\end{corollary}

\begin{proof}
Let $\{\zeta_i \mid i \in I_{B_n} \}$ denote the fundamental weights of type $B_n$.
We define a map $\phi \colon B(\zeta_n) \to B(\fw_n)$ by adding an additional sign such that the parity is correct. The map $\phi$ is a bijection (and a $\{1, \dotsc, n\}$-crystal isomorphism), and this extends to tensor products; in particular, this induces a bijection $B(k\zeta_n) \to B(k\fw_n)$. This proves the first equality.

Note that there exists a similar bijection to $B(\fw_{n+1})$ by taking the opposite parity. This shows the second equality.
\end{proof}

From~\eqref{eq:dim_column_type_D}, it is straightforward to construct a bijection from the tableaux of~\cite[Thm.~6.1]{Proctor94} to subsets of size $s$ of the corresponding alphabet since the $m$-protection condition of a $(2n)$-orthogonal tableau from~\cite{Proctor94} is vacuously true. Hence, we have a description using lattice paths from $(0,0)$ to $(2n-\delta_{s,n}-s, s)$.
An equivalent construction can be given using the tableaux of~\cite[Def.~3.5]{KW93}.

However, we note that for type $D_4$, we have
\begin{align*}
\nps(\tfw_{n-1}) & = q^{18} + q^{17} + q^{16} + 2q^{15} + 2q^{14} + 4q^{13} + 5q^{12} + 5q^{11} + 5q^{10} + 4q^{9}
\\ & \hspace{20pt} + 5q^{8} + 5q^{7} + 5q^{6} + 4q^{5} + 2q^{4} + 2q^{3} + q^{2} + q + 1
\\ \qbinom{8}{3}{q} & = q^{15} + q^{14} + 2q^{13} + 3q^{12} + 4q^{11} + 5q^{10} + 6q^{9} + 6q^{8}
\\ & \hspace{20pt} + 6q^{7} + 6q^{6} + 5q^{5} + 4q^{4} + 3q^{3} + 2q^{2} + q + 1,
\end{align*}
so there is not a natural $q$-analog of~\eqref{eq:dim_column_type_D}. Moreover, the corresponding ratio does not appear to be simple.
Additionally, we have
\[
840 = \dim V\bigl(2(\omega_3 + \omega_4)\bigr) \neq \det \begin{bmatrix}
\displaystyle \binom{2\cdot4}{4-1} & \displaystyle \binom{2\cdot5}{5-1} \vspace{5pt} \\
\displaystyle \binom{2\cdot5}{5-1} & \displaystyle \binom{2\cdot6}{6-1} \end{bmatrix}
= 252,
\]
so there is not a natural extension of~\eqref{eq:dim_column_type_D} with $s = n-1$ to $r\tfw_{n-1}$ using determinants of $r \times r$ matrices.
We note that from~\cite[Thm.~2.1]{Okada09}, we have in type $D_n$
\[
\det \left[ \binom{2(n+i+j)}{n+i+j} \right]_{i,j=0}^{r-1} = 2^r \dim V(2r\fw_n).
\]

The following result has appeared in~\cite[Thm.~2.1]{Okada09} and is likely well-known to experts.

\begin{proposition}
\label{prop:all_spin_rectangles}
Let $\g$ be of type $B_n$. We have
\[
\dim V(r\fw_n) = \prod_{1\leq i \leq j \leq n} \dfrac{r+i+j-1}{i+j-1}.
\]
\end{proposition}

\begin{proof}
Let
\begin{equation}
\label{eq:simplified_factorials}
F(n) = \prod_{i=1}^n i!,
\qquad\qquad
\Phi(n) = n! \cdot (n-2)! \cdot (n-4)! \cdots.
\end{equation}
From Proposition~\ref{prop:Bn_ps} as $q \to 1$, we have
\begin{align*}
\dim V(r\fw_n) & = \prod_{i=1}^n \dfrac{r + 2n-2i+1}{2n-2i+1} \prod_{1 \leq i < j \leq n} \dfrac{r+2n-i-j+1}{2n-i-j+1}
\allowdisplaybreaks \\ & =  \prod_{i=1}^n \dfrac{r + 2n-2i+1}{2n-2i+1} \prod_{2 \leq j \leq n} \dfrac{(r+2n-j) \cdots (r+2n-2j+2)}{(2n-j) \cdots (2n-2j+2)}
\allowdisplaybreaks \\ & = \dfrac{(r + 2n-1) \cdots (r + 3) (r + 1)}{(2n-1) \cdots 3 \cdot 1} \prod_{j=1}^{n-1} \dfrac{(r+2n-j-1) \cdots (r+2n-2j)}{(2n-j-1) \cdots (2n-2j)}
\allowdisplaybreaks \\ & = \dfrac{(r + 2n-1)!!}{(r-1)!! (2n-1)!!} \prod_{j=1}^{n-1} \dfrac{(r+2n-j-1)! (2n-2j-1)!}{(r+2n-2j-1)! (2n-j-1)!}
%\allowdisplaybreaks \\ & = \dfrac{(r + 2n-1)!!}{(r-1)!! (2n-1)!!} \prod_{j=1}^{n-1} \dfrac{(r+n+j-1)! (2j-1)!}{(r+2j-1)! (n+j-1)!}
\allowdisplaybreaks \\ & = \dfrac{(r + 2n-1)!!}{(r-1)!! (2n-1)!!} \cdot \dfrac{F(2n+r-2) \Phi(r-1) F(n-1) \Phi(2n-3)}{F(r+n-1) \Phi(r
+2n-3) F(2n-2)}
\allowdisplaybreaks \\ & = \dfrac{(r + 2n-1)!!}{(r-1)!! (2n-1)!!} \cdot \dfrac{\Phi(2n+r-2) \Phi(r-1) F(n-1)}{F(r+n-1) \Phi(2n-2)}
\allowdisplaybreaks \\ & = \dfrac{\Phi(2n+r-1) \Phi(r-2) F(n-1)}{F(r+n-1) \Phi(2n-1)}.
\end{align*}
Moreover, we have
\begin{align*}
 \prod_{1\leq i \leq j \leq n} \dfrac{r+i+j-1}{i+j-1} & = \prod_{j=1}^n \dfrac{(r+2j-1) \cdots (r+j)}{(2j-1) \cdots j} = \prod_{j=1}^n\dfrac{(r+2j-1)! (j-1)!}{(r+j-1)! (2j-1)!}
 \\ & = \dfrac{\Phi(r+2n-1)F(r-1) F(n-1)}{\Phi(r-1)F(r+n-1)\Phi(2n-1)}
 \\ & = \dfrac{\Phi(r+2n-1) \Phi(r-2) F(n-1)}{F(r+n-1)\Phi(2n-1)},
\end{align*}
and so the claim follows.
\end{proof}

%%%%%%%%%%%%%%%%%%%%%%%%%%%%%%%%%%%%%%%%
\section{Semistandard rigid tableaux}
\label{sec:rigid_tableaux}

In this section, we recall the notion of semistandard rigid tableaux from~\cite{KLO17} and then define a crystal structure directly on semistandard rigid tableaux.
We begin with the definition in type $B_n$ and then give some results. We then describe an extension to type $C_n$.

% ========
\subsection{Definition and type $B_n$}

In this section, we assume that $\g$ is of type $B_n$.
A \defn{strict partition} is a partition such that all parts are distinct (equivalently, it is a proper set of positive integers). We can identify elements of the spin representation $B(\fw_n)$ with strict partitions $\nu$ such that $\nu_1 \leq n$ (equivalently, all subsets of $\{1, \dotsc, n\}$) by
\begin{equation}
\label{eq:spin_isomorphism}
(s_1, \dotsc, s_n) \mapsto \{n+1-i \mid s_i = -\}.
\end{equation}
Note that we are counting $-$ from the bottom rather than from the top. Thus we define a crystal structure on strict partitions by this identification. Explicitly, for a strict partition $\nu$, we have
\begin{align*}
e_{n-i}(\nu) & = \begin{cases}
(\nu \setminus \{i+1\}) \cup \{i\} & \text{if $i+1 \in \nu$ and $i \notin \nu$}, \\
0 & \text{otherwise},
\end{cases}
\\ f_{n-i}(\nu) & = \begin{cases}
(\nu \setminus \{i\}) \cup \{i+1\} & \text{if $i \in \nu$ and $i+1 \notin \nu$}, \\
0 & \text{otherwise},
\end{cases}
\end{align*}
where we append/remove $0$ to $\nu$ as necessary. See Figure~\ref{fig:strict3} for an example.

We identify the reverse of an $m$-fold tensor product of $B(\fw_n)$ (equivalently strict partitions) with a sequence of strict partitions denoted by $\mathcal{SP}_{n,m}$. We use this to induce a crystal structure on $\mathcal{SP}_{n,m}$. Thus, we have the following by the identification~\eqref{eq:spin_isomorphism}.

\begin{proposition}
\label{prop:SSRT_hw}
Suppose a sequence of strict partitions $T$ is a highest weight element of weight $\fw$. Then the closure of $T$ under the crystal operators is isomorphic to $B(\lambda)$. Moreover, we have $\mathcal{SP}_{n,m} \iso B(\fw_n)^{\otimes m}$.
\end{proposition}

\begin{example}
Let $\varpi = \sum_{i \in I} c_i \fw_i$. The sequence of strict partitions
\begin{align}\label{eq: def Ow}
O_{\varpi} = \bigl( \underbrace{\emptyset, \dotsc, \emptyset}_K, \underbrace{(1), \dotsc, (1)}_{c_{n-1}}, \dotsc, \underbrace{(n-1, \dotsc, 1), \dotsc, (n-1, \dotsc, 1)}_{n_1} \bigr),
\end{align}
for $K = c_1 + \cdots + c_n$ corresponds to a highest weight crystal of highest weight $\varpi$. Hence the closure of $O_{\fw}$, which we denote by $R(\fw)$, under the induced Kashiwara operators is isomorphic to $B(\varpi)$.
\end{example}

\begin{figure}
\[
\begin{tikzpicture}[xscale=1.3,baseline=-4]
\node (e) at (0,0) {$\emptyset$};
\node (1) at (1.5,0) {$(1)$};
\node (2) at (3,0) {$(2)$};
\node (3) at (4.5,1) {$(3)$};
\node (21) at (4.5,-1) {$(2,1)$};
\node (31) at (6,0) {$(3,1)$};
\node (32) at (7.75,0) {$(3,2)$};
\node (321) at (9.5,0) {$(3,2,1)$};
\draw[->,blue] (e) to node[above]{\tiny$3$} (1);
\draw[->,magenta] (1) to node[above]{\tiny$2$} (2);
\draw[->,red] (2) to node[above,sloped]{\tiny$1$} (3);
\draw[->,blue] (3) to node[above,sloped]{\tiny$3$} (31);
\draw[->,blue] (2) to node[below,sloped]{\tiny$3$} (21);
\draw[->,red] (21) to node[below,sloped]{\tiny$1$} (31);
\draw[->,magenta] (31) to node[above]{\tiny$2$} (32);
\draw[->,blue] (32) to node[above]{\tiny$3$} (321);
\end{tikzpicture}
\]
\caption{The crystal of strict partitions $B(\fw_3)$ in type $B_3$.}
\label{fig:strict3}
\end{figure}
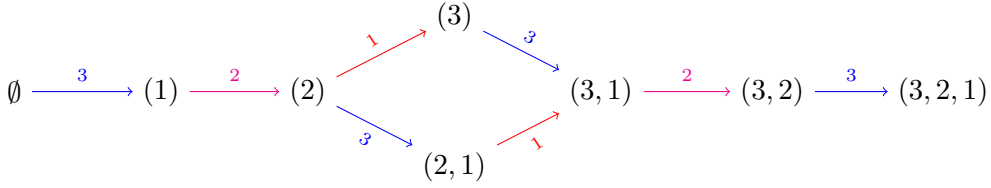

For a tuple $\lambda = (\lambda_1,\lambda_2,\ldots,\lambda_k)$, $1 \le u \le k$ and $1 \le s < k$, we define a tuple
\[
\lambda_{> s} \seteq (\lambda_{s+1},\lambda_{s+2},\ldots,\lambda_k).
\]
For a positive integer $m$, we denote by $\uplambda(m)$ the strict partition given by
\[
\uplambda(m) = (m,m-1,\ldots,2,1),
\]
and call $\uplambda (m)$ the \defn{$m$-th staircase partition}. We also set $\uplambda(m)=(0)$ for any non-positive integer $m$. More generally, for $ a \ge b \ge 1$, we denote by $\uplambda(a;b)$ the strict partition given by
\[
\uplambda(a;b)=(a,a-1,\ldots,b).
\]

Next, we define the notion of a \defn{semistandard rigid tableau of shape $\mu / \eta$} (in short SSRT) to be a skew-tableaux of shape $\mu / \eta$ that is weakly decreasing along columns and strictly decreasing along rows.
SSRTs with $m$ rows and max entry $n$ are in bijection with $\mathcal{SP}_{n,m}$. Indeed, consider a sequence of strict partition $T = (\tau^{(1)},\tau^{(2)},\ldots,\tau^{(l-1)},\tau^{(l)})$. We construct two partitions $\mu = (\mu_1,\ldots,\mu_\ell) \supset \eta = (\eta_1,\ldots,\eta_{\ell-1})$ in the following way: for each $1 \le i \le l-1$, there exists a unique $t_i \in \Z_{\ge 0}$ such that
\[
\tau^{(i)} \supset \tau^{(i+1)}_{>t_i} \text{ and } \tau^{(i)} \not \supset \tau^{(i+1)}_{>(t_i-1)} \text{ for } 1 \le i \le k-1.
\]
Then we set
\[
\eta_i = \displaystyle\sum_{s=i}^{\ell-1} t_s \ \ (1 \le i \le \ell-1) \quad \text{ and } \quad \mu_j = \eta_j +\ell(\tau^{(j)})\ \ (1 \le j \le \ell).
\]
Moreover, we can consider the entries of the $i$-th row of the corresponding SSRT as $\tau^{(i)}$.
Thus, we identify sequences of strict partitions and SSRTs.
For a given partition $\eta$ with $\ell(\eta)<m$, we denote by $R_{n,m}(\eta)$ the set of all SSRTs whose lengths (resp. max entries) are less than or equal to $m$ (resp. $n$) and
inner shapes are the same as $\eta$.

\begin{example}
For $n \ge 5$ and $m\ge 5$, the sequence of strict partitions
\[
T=\bigl( (5,4,3),(5,4),(5,4,3,2),(5,4,3,2,1) \bigr)
\]
corresponds to a skew-tableaux of shape $(6,5,5,5)/(3,3,1)$ as follows:
\[
\young(\cdot\cdot\cdot543,\cdot\cdot\cdot54,\cdot5432,54321).
\]
Note that $T \in R_{n,m}\bigl( (3,3,1) \bigr)$.
\end{example}

\begin{lemma}[{\cite[Lemma 6.1]{KLO17}}]
\label{lemma: general hw B}
For a sequence of strict partitions
\[
T=\left( \tau^{(1)},\tau^{(2)},\ldots,\tau^{(\ell-1)},\tau^{(\ell)} \right),
\]
such that $T$ is a highest weight element, then $\tau^{(1)}=\emptyset$ and $\tau^{(2)}=\uplambda(s)$ for some $s \in \Z_{\ge 0}$.
\end{lemma}

\begin{proposition}
\label{prop: same inner shape}
The set $R_{n,m}(\eta)$ is closed under the Kashiwara operators. In particular, $R_{n,m}(\eta)$ is a subcrystal consisting of connected components of $\mathcal{SP}_{n,m}$.
\end{proposition}

\begin{proof}
It is enough to show when $m=2$ and $\eta = (n-k)$ for some $0 \le k \le n$. Fix $T = (\tau^{(1)},\tau^{(2)}) \in R_{n,2}\bigl((n-k)\bigr)$. Now we shall prove that
$e_i(T) \in R_{n,2}\bigl((n-k)\bigr)$ if $e_i(T) \ne 0$. Assume to the contrary that
$e_i(T) \not \in R_{n,2}\bigl((n-k)\bigr)$. Then, for $i<n$, one of the following happens:
\begin{itemize}
\item[{\rm (a)}] $\tau^{(1)}_s=\tau^{(2)}_{s+n-k}=n+1-i$ and $e_i(T)^{(1)}_s = n-i$, or
\item[{\rm (b)}] $\tau^{(2)} \ne e_i(T)^{(2)}$ and $e_i(T)^{(1)} \supseteq e_i(T)^{(2)}_{>n-k-1}$.
\end{itemize}
Consider {\rm (a)} first. In that case $\sig_i(\tau^{(1)})=-$. However, since $T$ is a SSRT, we have $n+1-i > \tau^{(1)}_{s+1} \ge \tau^{(2)}_{s+1+n-k}$ and hence $\sig_i(\tau^{(2)})=-$.
Thus we have a contradiction by the tensor product rule.

Now let us consider {\rm (b)}. In that case, there exists $s \in \Z_{\ge 0}$ such that
\[
\tau^{(1)}_{s+1}=n-i, \ \tau^{(2)}_{s+n-k}=n+1-i \ \text{ and } \ \tau^{(2)}_{s+n-k+1} < n-i.
\]
Since $\sig_i(\tau^{(1)})= \cdot$, we have $\tau^{(1)}_{s}=n+1-i$ and hence $\tau^{(2)}_{s+n-k-1} > \tau^{(1)}_{s}$ which yields a contradiction for the assumption that $T \in R_{n,2}\bigl((n-k)\bigr)$.

Let $i=n$.
Since $\tilde{e}_n(T) \ne 0$, we have $\sig_n(T)=(-,-)$ or  $\sig_n(T)=(-,+)$. Note that if $\tau^{(2)} = \uplambda(n-k)$, then $\tau^{(1)} = \emptyset$, and hence $\sig_n(T) = (+,-) = (\cdot,\cdot)$.
Then our assertion for $i=n$ can be easily checked.
\end{proof}

%----
\subsubsection{Fundamental representation cases: $B(\tfw_k)$}

\begin{theorem} \label{thm: main fw}
The subcrystal $R_{n,2}\bigl( (n-k) \bigr)$ is isomorphic to $B(\tfw_k)$ of type $B_n$. Moreover, $R_{n,2}\bigl( (n-k) \bigr) = R(\tfw_k)$.
\end{theorem}

\begin{proof}[First proof]
By Proposition~\ref{prop:SSRT_hw}, Proposition~\ref{prop: same inner shape}, and Lemma~\ref{lemma: general hw B}, any $T \in R_{n,2}\bigl( (n-k) \bigr)$ is connected to $(\emptyset,\uplambda(s))$ for some $s$. Note that the inner shape of $(\emptyset,\uplambda(s))$ is $s$, and so we must have $s = n-k$.
\end{proof}

\begin{proof}[Second proof]
By Proposition~\ref{prop:SSRT_hw} and~\eqref{eq:2tens_B}, $R_{n,2}\bigl( (n-k) \bigr)$ must be connected since it is stable under the Kashiwara operators by Proposition~\ref{prop: same inner shape}.
\end{proof}

\subsubsection{Almost spin representation cases}

\begin{definition}
We call an integral dominant weight $\fw$ \defn{almost spin} if
\[
\fw = (\ell-1) \fw_n +\tfw_k
\]
for some $1 \le k \le n$. In particular, we call $(\ell+1)\fw_n$ \defn{pure spin}.
\end{definition}

For an almost spin $\fw = (\ell-1) \fw_n +\tfw_k$, the element $O_\fw$ from~\eqref{eq: def Ow} is
\[
O_\fw= \left( O_{\fw_{n}},\ldots,O_{\fw_{n}},O_{\fw_{k}} \right) = (\underbrace{\emptyset,\ldots,\emptyset}_{\text{$\ell$-times}},\uplambda(n-k)).
\]
Now we fix an almost spin weight $\fw = (\ell-1) \fw_n +\tfw_k$. We denote by $R^\mathrm{as}(k) \seteq R_{n,\ell+1}\bigl((n-k)^{\ell}\bigr)$.

\begin{theorem}
\label{thm: main as}
Let $\fw = (\ell-1) \fw_n +\tfw_k$.
The subcrystal $R^\mathrm{as}(\ell, k)$ is isomorphic to $B(\fw)$ of type $B_n$. Moreover, $R^\mathrm{as}(\ell, k) = R(\fw)$.
\end{theorem}

\begin{proof}
By Proposition~\ref{prop: same inner shape}, it is enough to show that $R^\mathrm{as}(\ell,k)$ is connected by the Kashiwara operators $e_i$ and $f_i$.
By Lemma~\ref{lemma: general hw B}, any $T \in R^\mathrm{as}(\ell,k)$ is connected to a highest weight element
\[
(\emptyset,\uplambda(s),\tau^{(3)},\ldots,\tau^{(\ell)},\tau^{(\ell+1)}).
\]
Note that for the inner shape, either the first row is $n-k+s$ or the second row is $n-k-s$. Therefore, by the definition of $R^\mathrm{as}(\ell,k)$, we must have $s=0$. By iterating this argument, we also have $\tau^{(k)} = \emptyset$ for $3 \le k \le \ell$ and $\tau^{(\ell+1)} = \uplambda(n-k)$. Thus, the assertion follows.
\end{proof}

\begin{figure}[p]
\[
\input{rigid_tableaux_ex}
\]
\caption{The crystal of semistandard rigid tableaux $R(\tfw_2)$ (left) and $R(\tfw_3)$ (right) in type $B_3$.}
\label{fig:rigid_tableaux2}
\end{figure}

\begin{remark}
\label{rem:bijective_spin_rectangles}
We can also give a bijective proof of Proposition~\ref{prop:all_spin_rectangles}. Indeed, we recall from~\cite{Gordon83} that
\[
\prod_{1 \leq i \leq j \leq n} \dfrac{r+i+j-1}{i+j-1}
\]
equals the number of semistandard Young tableaux with at most $r$ columns and entries in $\{1, \dotsc, n\}$.
Note that any SSRT in $R(r\fw_n)$ is precisely the dual conjugate of such a semistandard tableau. Hence, Theorem~\ref{thm: main as} implies Proposition~\ref{prop:all_spin_rectangles}. More explicitly, we reflect each entry $t$ about the $y=-x$ line and replace it with $n+1-t$. As an example, note the tableaux for $R(\tfw_3)$ in Figure~\ref{fig:rigid_tableaux2}.
\end{remark}

% ========
\subsection{Motzkin triangle numbers}

In this section, we shall give various realization of Motzkin triangular numbers $\Mot_{(m,s)}$ in term of SSRTs.

For $0\le s \le n-1$ and any highest weight element $T^{\natural} =(\tau^{(1)},\tau^{(2)},\tau^{(3)})$ of highest weight $\fw_n+\tfw_{n-s}$, let us denote by $R(T^{\natural})$ the closure under $e_i$ and $f_i$ containing $T^{\natural}$.
Recall that a tableau $T$ is \defn{standard} if $T$ has $n$ boxes and each $i \in \{1, \dotsc, n\}$ appears exactly once in $T$.
Define the subset
\[
S(T^{\natural})_m \seteq \{  T \in R(T^{\natural}) \mid \text{$\Sh(T) \vdash m$ and $T$ is standard} \}.
\]
Then Theorem~\ref{thm: wt Motzkin} can be restated as follows:
\[
\lvert S(T^{\natural})_m \rvert = \Mot_{(m,s)}.
\]

In this section, for each $(m,s)$, we shall show that there are distinct $(s+1)$-many sets of standard rigid tableaux
whose cardinalities are the same as $\Mot_{(m,s)}$. In this section we fix $\fw=\fw_n+\tfw_{n-s}$.

We shall start with the following lemma which is easy to check:

\begin{lemma}
\label{lemma: tens3_Bn}
In type $B_n$, we have
\[
B(\fw_n)^{\otimes 3} \iso \bigoplus_{s=0}^{n} B(\fw_n+\tfw_{n-s})^{\oplus s+1}.
\]
Furthermore, for $\fw = \fw_n+\tfw_{n-s}$, the highest crystal elements are given by
\[
O^t_{\fw} = (\emptyset, \uplambda(t), \uplambda(s;t+1)) \quad \text{ for } \quad 0 \le t \le s.
\]
\end{lemma}
Note that for $\fw = \fw_n+\tfw_{n-s}$, the inner shape of $O^t_{\fw}$ is $(s,s-t)$. Next, we denote $R^t(\fw_n+\tfw_{n-s}) \seteq R_{n,3}\bigl((s,s-t)\bigr)$.

\begin{theorem} \label{thm: main mz}
For $0\le s \le n$, the subcrystal $R^t(\fw_n+\tfw_{n-s})$ is isomorphic to $B(\fw_n+\tfw_{n-s})$ of type $B_n$.
\end{theorem}

\begin{proof}
Since $R^t(\fw_n+\tfw_{n-s})$ is stable under the Kashiwara operators by Proposition~\ref{prop: same inner shape}, Lemma~\ref{lemma: tens3_Bn} and the pigeonhole principle implies that the same argument of the second proof of Theorem~\ref{thm: main fw} holds.
\end{proof}

Note that $R^0(\fw_n+\tfw_{n-s}) = R(O_{\fw_n+\tfw_{n-s}})$ in the previous section. Then the following corollary follows from
Theorem~\ref{thm: wt Motzkin}.

\begin{corollary} For $0\le s \le n-1$,
let $S^t(\fw_n+\tfw_{n-s})_m$ be the subset of $R^t(\fw_n+\tfw_{n-s})$ satisfying the following properties:
$T \in S^t(\fw_n+\tfw_{n-s})_m$ if $\Sh(T) \vdash m$ and $T$ is standard. Then we have
\[
\lvert S^t(\fw_n+\tfw_{n-s})_m \rvert = \Mot_{(m,s)}.
\]
\end{corollary}

\begin{example}
\mbox{}
\begin{enumerate}
\item For $m=5$, $s=3$ and $t=3$, the elements in $S^3(\fw_n+\fw_{n-3})_5$ are
\begin{align*}
&\young(\cdot\cdot\cdot,543,21), \ \young(\cdot\cdot\cdot,542,31), \ \young(\cdot\cdot\cdot,541,32), \ \young(\cdot\cdot\cdot,532,41), \ \young(\cdot\cdot\cdot,531,42),  \young(\cdot\cdot\cdot1,542,3), \ \young(\cdot\cdot\cdot1,543,2), \  \young(\cdot\cdot\cdot2,543,1), \ \\
&\young(\cdot\cdot\cdot1,532,4), \ \young(\cdot\cdot\cdot2,5431), \  \young(\cdot\cdot\cdot3,5421), \  \young(\cdot\cdot\cdot4,5321), \  \young(\cdot\cdot\cdot5,4321), \   \young(\cdot\cdot\cdot21,543),
\end{align*}
which shows that
\[
\lvert S^3(\fw_n+\fw_{n-3})_5 \rvert = \Mot_{(5,3)}=14.
\]
\item For $m=5$, $s=3$ and $t=2$, the elements in $S^2(\fw_n+\fw_{n-3})_5$ are
\begin{align*}
&\young(\cdot\cdot\cdot21,\cdot43,5), \  \young(\cdot\cdot\cdot,\cdot43,521), \ \young(\cdot\cdot\cdot,\cdot53,421), \ \young(\cdot\cdot\cdot,\cdot42,531), \ \young(\cdot\cdot\cdot,\cdot52,431), \ \young(\cdot\cdot\cdot,\cdot54,321), \ \young(\cdot\cdot\cdot2,\cdot431,5), \\
 & \young(\cdot\cdot\cdot3,\cdot421,5), \  \young(\cdot\cdot\cdot4,\cdot321,5), \ \young(\cdot\cdot\cdot5,\cdot321,4), \young(\cdot\cdot\cdot2,\cdot43,51), \ \young(\cdot\cdot\cdot1,\cdot42,53), \ \young(\cdot\cdot\cdot1,\cdot43,52),  \ \young(\cdot\cdot\cdot1,\cdot52,43),
\end{align*}
which also yields
\[
\lvert S^2(\fw_n+\fw_{n-3})_5 \rvert = \Mot_{(5,3)} = 14.
\]
\item For $m=5$, $s=3$ and $t=1$, the elements in $S^1(\fw_n+\fw_{n-3})_5$ are
\begin{align*}
&\young(\cdot\cdot\cdot21,\cdot\cdot3,54), \ \young(\cdot\cdot\cdot2,\cdot\cdot31,54), \ \young(\cdot\cdot\cdot3,\cdot\cdot21,54), \ \young(\cdot\cdot\cdot4,\cdot\cdot21,53), \
\young(\cdot\cdot\cdot1,\cdot\cdot3,542), \ \young(\cdot\cdot\cdot1,\cdot\cdot4,532), \ \young(\cdot\cdot\cdot2,\cdot\cdot3,541), \\
&  \young(\cdot\cdot\cdot1,\cdot\cdot5,432), \ \young(\cdot\cdot\cdot5,\cdot\cdot21,43), \ \young(\cdot\cdot\cdot3,\cdot\cdot4,521), \ \young(\cdot\cdot\cdot2,\cdot\cdot4,531), \
\young(\cdot\cdot\cdot2,\cdot\cdot5,431), \ \young(\cdot\cdot\cdot3,\cdot\cdot5,421), \ \young(\cdot\cdot\cdot4,\cdot\cdot5,321), \
\end{align*}
which results in
\[
\lvert S^2(\fw_n+\fw_{n-3})_5 \rvert = \Mot_{(5,3)} = 14.
\]
\end{enumerate}
\end{example}

\begin{remark}
Theorem~\ref{thm: wt Motzkin} and Lemma~\ref{lemma: tens3_Bn} imply the identity
\begin{align}\label{eq: Motzkin}
 3^m =\sum_{s=0}^m (s+1) \Mot_{(m,s)}
\end{align}
without appealing to Schur--Weyl duality.
Indeed, note that the number of strict partitions $(\tau^{(1)},\tau^{(2)},\tau^{(3)})$ satisfying $\tau^{(1)} \cup \tau^{(2)} \cup \tau^{(3)}=\uplambda(m)$ is $3^m$.
\end{remark}

% ========
\subsection{Extension to type $C_n$}

Now we assume that $\g$ is of type $C_n$ for this section.
Recall that there exists a virtualization map on highest weight crystals from type $C_n$ to $B_n$ given by
\[
e_i \mapsto \virtual{e}_i^{\gamma_i},
\qquad\qquad
f_i \mapsto \virtual{f}_i^{\gamma_i},
\qquad\qquad
\fw_i \mapsto \gamma_i \virtual{\fw}_i,
\]
where $\gamma_i = 1+\delta_{in}$ (see, \textit{e.g.},~\cite{K96,OSS03III,OSS03II,SchillingS15}) and for each object $X$ in type $C_n$, we write the corresponding object in type $B_n$ as $\virtual{X}$. Thus, we can realize a highest weight crystal $B(\lambda)$ as a subset of semistandard rigid tableaux $R(O_{\lambda})$.

\begin{proposition}
\label{prop:virtual_rigid}
Let $\g$ be of type $C_n$ and $v \colon B(r\fw_n) \to R(O_{r\tfw_n})$ be the virtualization map. Then we have
\[
v\bigl( B(r\fw_n) \bigr) = \left\{ T \in R\left(O_{r\tfw_n} \right) \mid \text{$\Sh(T)$ is an even partition} \right\}.
\]
\end{proposition}

\begin{figure}
\[
\input{rigid_tableaux_virtual_ex}
\]
\caption{The crystal of semistandard rigid tableaux $R(\fw_2)$ (left) and $R(\fw_3)$ (right) in type $C_3$.}
\label{fig:rigid_virtual}
\clearpage
\end{figure}

For an example of Proposition~\ref{prop:virtual_rigid}, see Figure~\ref{fig:rigid_virtual}.
Thus, from~\cite[Eq.~(2)]{dSCV86} and taking the dual conjugate tableau similar to Remark~\ref{rem:bijective_spin_rectangles}, we have a bijective proof of the following alternative formula to Corollary~\ref{cor: Cn determinatal r omega_n}.

\begin{proposition}
In type $C_n$, we have
\[
\dim V(r\fw_n) = \prod_{1\leq i \leq j \leq n} \dfrac{2r+i+j}{i+j}.
\]
\end{proposition}

%%%%%%%%%%%%%%%%%%%%%%%%%%%%%%%%%%%%%%%%
\section{Semistandard spin rigid tableaux}
\label{sec:spin_rigid_tableaux}

In this section, we assume $\g$ is of type $D_n$ and use the notation of semistandard spin rigid tableaux from~\cite{KLO17}. We will describe the corresponding crystal structure on semistandard spin rigid tableaux and give some applications.

%=========
\subsection{Definition}

A \defn{colored integer} is an integer with a color, gray or white. For example,
\begin{itemize}
\item $\gn{3}$ denotes a gray integer whose quantity $|\gn{3}| = 3$ and
\item $\nn{3}$ denotes a white integer whose quantity $|\nn{3}| = 3$.
\end{itemize}

\begin{remark}
When we do not want to tell the color of a colored integer, we write it without circle; that is, $\cn{n}$ can be
$\gn{n}$ or $\nn{n}$.  Also we assign an integer $\clr(\cn{n})=1$ if $\cn{n}$ is gray and $\clr(\cn{n})=0$ if $\cn{n}$ is white.
\end{remark}

\begin{definition}
For colored integers $\cn{j}$ and $\cn{k}$, we construct two partial orders:
\begin{enumerate}
\item $\cn{j} \succeq \cn{k}$ if and only if their colors coincide $(\clr(\cn{j})=\clr(\cn{k}))$ and $|\cn{j}| \ge |\cn{k}|$;
\item $\cn{j} > \cn{k}$ if and only if $|\cn{j}| > |\cn{k}|$.
\end{enumerate}
\end{definition}

\begin{definition} For a sequence of colored integers $\uptau=(\uptau_1,\uptau_2,\ldots,\uptau_s)$, we say
$\uptau$ a \defn{alternating strict partition (ASP)} if
\begin{enumerate}
\item $\clr(\uptau_i) \ne \clr(\uptau_{i+1})$ for all $i$,
\item $\uptau_i > \uptau_{i+1}$.
\end{enumerate}
\end{definition}

\begin{definition} For two $\ASP$s $\blam$ and $\bmu$, we denote by $\blam \supseteq \bmu$, if
$\blam_i \succeq \bmu_i$ for each $1 \le i \le \min\{ \ell(\blam),\ell(\bmu) \}$.
\end{definition}

Take a ASP $\btau= (\btau_1 > \btau_2 > \ldots > \btau_\ell)$ such that $n \ge \btau_1$ and $i \in I=\{ 1,\ldots,n,n+1\}$
We define ASPs $e_i(\btau)$ and $f_i(\btau)$ as follows:
\begin{align*}
&e_i(\btau) \seteq \begin{cases}
(\btau_1,\btau_2,\ldots,\cn{n}-\cn{i},\ldots,\btau_\ell) & \text{ if $i<n$, $\cn{n}-\cn{i}+\cn{1}$ is a part and $\cn{n}-\cn{i}$ is not}, \\
(\btau_1,\btau_2,\ldots,\btau_{\ell-1})& \text{ if $i=n+1$ and  $\btau_\ell=\nn{1}$}, \\
(\btau_1,\btau_2,\ldots,\btau_{\ell-1})& \text{ if $i=n$ and  $\btau_\ell=\gn{1}$}, \\
0 & \text{otherwise},
\end{cases} \allowdisplaybreaks \\
& f_i(\btau) \seteq \begin{cases}
(\btau_1,\btau_2,\ldots,\cn{n}+\cn{1}-\cn{i},\ldots,\btau_\ell) & \text{ if $i<n$, $\cn{n}-\cn{1}$ is a part and $\cn{n}-\cn{i}+\cn{1}$ is not}, \\
(\btau_1,\btau_2,\ldots,\btau_\ell,\nn{1})& \text{ if $i=n+1$, $\clr(\btau_\ell)$ is gray and $\cn{1}$ is not a part}, \\
(\btau_1,\btau_2,\ldots,\btau_\ell,\gn{1})& \text{ if $i=n$, $\clr(\btau_\ell)$ is white and $\cn{1}$ is not a part}, \\
0 & \text{otherwise}.
\end{cases}
\end{align*}

One can easily check the following theorem from the spin representations of type $D_{n+1}$.

\begin{theorem}
The set of ASPs $\{ \btau \mid \btau_1 \le n \text{ and } \clr(\btau_1)=0 \}$ $($resp. $\{ \btau \mid \btau_1 \le n \text{ and } \clr(\btau_1)=1 \})$ with Kashiwara operators is isomorphic to
$B(\fw_{n+1})$ $($resp. $B(\fw_{n}))$ over $D_{n+1}$.
\end{theorem}

For later use, we denoted by $O_{\fw_{n+1}} \seteq \nn{$\emptyset$}$ and $O_{\fw_n} \seteq \gn{$\emptyset$}$, which can be understood as the highest weight elements of $B(\fw_{n+1})$ and $B(\fw_{n})$, respectively.
By the tensor product rule of crystals, we can define the crystal structure on the sequence of ASPs.

From Equation~\eqref{eq:2tens_Dpp} and~\eqref{eq:2tens_Dpm}, any highest weight crystal $B(\lambda)$ appears as a subcrystal of $B(\fw_{n})^{\otimes a} \otimes B(\fw_{n+1})^{\otimes b}$ for some $a,b \in \Z_{\ge 0}$. Thus any
$B(\lambda)$ can be realized as certain set of sequence of ASPs.

For a positive colored integer $\cn{m}$, we denote by $\uplambda(\cn{m})$ the ASP given by
\[
\uplambda (\cn{m})=(\cn{m},\cn{m}-\cn{1},\ldots,\cn{2},\cn{1}).
\]
More generally, for $ a \ge b \ge 1$, we
 we denote by $\uplambda(\cn{a};\cn{b})$ the ASP given by
\[
\uplambda(\cn{a};\cn{b})=(\cn{a},\cn{a}-\cn{1},\ldots,\cn{b})
\]
such that if $j \equiv k \mod{2}$, then $\clr(\cn{j}) = \clr(\cn{k})$.
If $b>a$, then we define $\uplambda(\cn{a};\cn{b}) = \uplambda(\cn{0})$ with $\clr(\cn{a}) = \clr(\cn{0})$.
Note that if $a\equiv b \mod{2}$, then $\clr(\cn{a}) = \clr(\cn{b})$.

 Next, we define the notion of a \defn{semistandard spin rigid tableau of shape $\mu / \eta$} (in short SSSRT) to be a skew-tableaux of shape $\mu / \eta$ that is weakly decreasing along columns and strictly decreasing along rows. As we did for SSRTs, we can assign two partitions $\mu \supset \eta$ for each SSSRT, in the following way:
for each $1 \le i \le \ell-1$, there exists a unique $t_i \in \Z_{\ge 0}$ such that
\[
\text{$\btau^{(i)} \supseteq \btau^{(i+1)}_{>t_i}$ and $\btau^{(i)} \not \supseteq \btau^{(i+1)}_{> t_i-2}$ for $1 \le i \le k-1$,}
\]
Then we set
\[
\eta_i = \displaystyle\sum_{s=i}^{\ell-1} t_s \ \ (1 \le i \le \ell-1) \quad \text{ and } \quad \mu_j = \eta_j +\ell(\btau^{(j)})\ \ (1 \le j \le \ell).
\]
Moreover, we can consider the entries of the $i$-th row of the corresponding SSSRT as $\btau^{(i)}$.
Thus, we identify sequences of ASPs and SSSRTs.

\begin{example}
The sequence of ASPs given by $\left(\left(\nn{$4$},\gn{$3$},\nn{$2$}\right),\left(\nn{$5$},\gn{$3$}\right) \right)$ is a SSSRT of shape $(5,2) \setminus (2)$:
\[
%\ytableausetup{smalltableaux}
\begin{ytableau}
\cdot& *(gray!40)  \cdot &  4 & *(gray!40) 3&  2\\
5& *(gray!40) 3
\end{ytableau}
\]
\end{example}

Hence we can understand any sequence of strict ASPs as a SSSRT $T$ of certain shape $\mu / \eta$.

\begin{lemma}[{\cite[Lemma 6.1]{KLO17}}]
\label{lemma: general hw D}
For a sequence of ASPs
\[
T=\left( \ntau{1},\ntau{2},\ldots,\ntau{\ell-1},\ntau{\ell} \right),
\]
such that $T$ is a highest weight element, then $\ntau{1}=\emptyset$ and $\ntau{2}=\uplambda(\cn{s})$ for some $s \in \Z_{\ge 0}$.
\end{lemma}

% ========
\subsection{Riordan triangle numbers}
In this section, we assume $\g$ is of type $D_{n+1}$.

\begin{definition}
We say that a composition $\lambda$ of $m$ is \defn{almost even} when it satisfies one of the following conditions:
\begin{itemize}
\item If $m$ is odd, then it contains one odd part and the other parts are even.
\item If $m$ is even, then it contains two odd parts and the other parts are even.
\end{itemize}
We write $\lambda \Vdash_0 m$ to denote an almost even composition $\lambda$ of $m$.
\end{definition}

\begin{remark}\label{rmk:weight consideration}
Note that the weight of $\nn{$k$}$ and $\gn{$k$}$ are given as follows:
\[
\wt\bigl(\nn{$k$}\bigr) = \epsilon_{k}-\epsilon_{n+1} \qquad \text{ and } \qquad \wt\bigl(\gn{$k$}\bigr) = \epsilon_{k}+\epsilon_{n+1}.
\]
\end{remark}

\subsubsection{The pure spin case $B(\fw_{n+1})^{\otimes 3}$}

The following is a straightforward computation similar to Lemma~\ref{lemma: tens3_Bn}.

\begin{lemma}
In type $D_{n+1}$, we have
\begin{align} \label{eq: tens3}
B(\fw_{n+1})^{\otimes 3} \iso \bigoplus_{s=0}^{\lfloor (n+1)/2\rfloor} B(\fw_{n+1}+\tfw_{n+1-2s})^{\oplus s+1}
\bigoplus_{s=1}^{\lfloor n/2\rfloor} B(\fw_n+\fw_{n-2s})^{\oplus s}.
\end{align}
\end{lemma}

\begin{example} \label{ex: D4 3tensor}
Let us consider $B(\fw_4)$ over $D_4$. Then we have
\[
B(\fw_{4})^{\otimes 3} \iso \left( B(\fw_{4})^{\oplus 3} \oplus B(\fw_{2}+\fw_{4})^{\oplus 2} \oplus B(3\fw_{4}) \right) \bigoplus
\left( B(\fw_{1}+\fw_{3}) \right),
\]
where corresponding highest weight elements can be described as follows:
\begin{align*}
\ytableausetup{smalltableaux}
B(\fw_{4})^{\oplus 3}  & \leftrightarrow \begin{ytableau}
\cdot& *(gray!40) \cdot &  \cdot & *(gray!40) \cdot \\
\cdot& *(gray!40) \cdot &  \cdot & *(gray!40) \cdot \\
3 & *(gray!40) 2 &  1 \\
\end{ytableau}\ , \
\begin{ytableau}
\cdot& *(gray!40) \cdot &  \cdot& *(gray!40) \cdot \\
\cdot& *(gray!40) \cdot &  1\\
3 & *(gray!40) 2 \\
\end{ytableau}\ , \
\begin{ytableau}
\cdot& *(gray!40) \cdot &  \cdot& *(gray!40) \cdot \\
3 & *(gray!40) 2 &  1 \\
\end{ytableau}\ ,
&
B(\fw_{2}+\fw_{4})^{\oplus 2}  &  \leftrightarrow \begin{ytableau}
\cdot& *(gray!40) \cdot \\
1
\end{ytableau}\ , \
\begin{ytableau}
\cdot& *(gray!40) \cdot \\
\cdot& *(gray!40) \cdot \\
1
\end{ytableau}\ ,
\\
 B(\fw_{1}+\fw_{3}) & \leftrightarrow \begin{ytableau}
\cdot& *(gray!40) \cdot &  \cdot& *(gray!40) \cdot \\
\cdot& *(gray!40) \cdot &  1\\
2 \\
\end{ytableau}\ , \
&
B(3\fw_{4})   & \leftrightarrow \emptyset.
\end{align*}
Hence we have two highest weight elements whose weights are different but inner shapes coincide:
\[
\begin{ytableau}
\cdot& *(gray!40) \cdot &  \cdot& *(gray!40) \cdot \\
\cdot& *(gray!40) \cdot &  1\\
3 & *(gray!40) 2 \\
\end{ytableau}\ ,
\qquad\qquad
\begin{ytableau}
\cdot& *(gray!40) \cdot &  \cdot& *(gray!40) \cdot \\
\cdot& *(gray!40) \cdot &  1\\
2 \\
\end{ytableau}\ .
\]
\end{example}

\begin{example}
For a connected component of $ \begin{ytableau}
\cdot& *(gray!40) \cdot \\
\cdot& *(gray!40) \cdot \\
1
\end{ytableau}$, we know $\dim V(\fw_{2}+\fw_{4})_{\fw_4}=6=\Rior_{(4,2)}$ is represented by following semistandard spin rigid tableaux:
\[
\begin{ytableau}
\cdot& *(gray!40) \cdot &  2& *(gray!40)1 \\
\cdot& *(gray!40) \cdot  \\
3
\end{ytableau}\ , \quad
\begin{ytableau}
\cdot& *(gray!40) \cdot &  3& *(gray!40)1 \\
\cdot& *(gray!40) \cdot  \\
2
\end{ytableau}\ , \quad
\begin{ytableau}
\cdot& *(gray!40) \cdot &  3& *(gray!40)2 \\
\cdot& *(gray!40) \cdot  \\
1
\end{ytableau}\ , \quad
\begin{ytableau}
\cdot& *(gray!40) \cdot &  1 \\
\cdot& *(gray!40) \cdot  \\
3 & *(gray!40) 2
\end{ytableau}\ , \quad
\begin{ytableau}
\cdot& *(gray!40) \cdot &  2 \\
\cdot& *(gray!40) \cdot  \\
3 & *(gray!40) 1
\end{ytableau}\ , \quad
\begin{ytableau}
\cdot& *(gray!40) \cdot &  3 \\
\cdot& *(gray!40) \cdot  \\
2 & *(gray!40) 1
\end{ytableau}\ .
\]
For a connected component of $ \begin{ytableau}
\cdot& *(gray!40) \cdot \\
1
\end{ytableau}$, we know $\dim V(\fw_{2}+\fw_{4})_{\fw_4}=6=\Rior_{(4,2)}$ is represented by following semistandard spin rigid tableaux:
\[
\begin{ytableau}
\cdot& *(gray!40) \cdot \\
3 & *(gray!40)2 \\
1
\end{ytableau}\ , \quad
\begin{ytableau}
\cdot& *(gray!40) \cdot \\
3 & *(gray!40)1 \\
2
\end{ytableau}\ , \quad
\begin{ytableau}
\cdot& *(gray!40) \cdot & 1 \\
3 & *(gray!40)2
\end{ytableau}\ , \quad
\begin{ytableau}
\cdot& *(gray!40) \cdot & 3 \\
2 & *(gray!40)1
\end{ytableau}\ , \quad
\begin{ytableau}
\cdot& *(gray!40) \cdot & 2 \\
3 & *(gray!40)1
\end{ytableau}\ , \quad
\begin{ytableau}
\cdot& *(gray!40) \cdot & 2 &  *(gray!40)1\\
3
\end{ytableau}\ .
\]
Note that inner shapes for $ \begin{ytableau}
\cdot& *(gray!40) \cdot \\
\cdot& *(gray!40) \cdot \\
1
\end{ytableau}$
and $ \begin{ytableau}
\cdot& *(gray!40) \cdot \\
1
\end{ytableau}$
are unique in Example~\ref{ex: D4 3tensor}, respectively.
 \end{example}

\begin{lemma}
\label{lem: hw 3}
\hfill
\begin{enumerate}
\item[{\rm (1)}] The $(s+1)$-many highest weight elements of weight $\fw_{n+1}+\tfw_{n+1-2s}$ in $B(\fw_{n+1})^{\tens 3}$ for $0 \le s \le \lfloor (n+1)/2\rfloor$ are given by
\[
O^0(s,t) \seteq \left(\nn{$\emptyset$},  \uplambda\left(\nn{$2t$-$1$}\right), \uplambda\left(\nn{$2s$-$1$},\gn{$2t$}\right) \right) \quad \text{ for } 0 \le t \le s,
\]
whose inner shape is $(2s,2s-2t)$ for $0 \le t \le s$, as a standard spin rigid tableau.
\item[{\rm (2)}] The $s$-many highest weight elements of weight $\fw_{n}+\fw_{n-2s}$ in $B(\fw_{n+1})^{\tens 3}$ for $1 \le s \le \lfloor n/2\rfloor$ are given by
\[
O^1(s,t) \seteq \left(\nn{$\emptyset$},  \uplambda\left(\nn{$2t$-$1$}\right), \uplambda\left(\nn{$2s$},\nn{$2t$}\right) \right) \quad \text{ for } 1 \le t \le s,
\]
whose inner shape is $(2s+2,2s-2t+2)$ for $1 \le t \le s$, as a standard spin rigid tableau.
\end{enumerate}
\end{lemma}

\begin{corollary} \label{cor: unique}
Let $\eta$ be $(2s)$ or $(2s,2s)$.
There exists a unique highest weight element with weight $\fw_{n+1}+\tfw_{n+1-2s}$ in $B(\fw_{n+1})^{\tens 3}$ for $0 \le s \le \lfloor (n+1)/2\rfloor$ whose inner shape is $\eta$.
\end{corollary}

\begin{definition} For each $0 \le s \le \lfloor (n+1)/2 \rfloor$ and $x=0,1$, we define a subset $R^x(\fw_{n+1}+\fw_{n+1-2s})$ of SSSRTs consisting of $T = (\ntau{1},\ntau{2},\ntau{3})$ satisfying the following conditions:
\begin{enumerate}
\item[{\rm (a)}] $\max\{ \ntau{k}_1 \mid 1 \le k \le 3 \} \le n$ and $\clr(\ntau{k}_1)=0$ for $1 \le k \le 3$,
\item[{\rm (b)}] the inner shape of $T$ is $(2s,\delta_{0x}2s)$.
\end{enumerate}
\end{definition}

Alternatively, the definition of $R^x(\fw_{n+1}+\fw_{n+1-2s})$ $(x=0,1)$ can be written as the following form: The subset of SSSRTs consisting of $T=(\ntau{1},\ntau{2},\ntau{3})$ such that
\begin{enumerate}
\item[{\rm (a)}] $\max\{ \ntau{k}_1 \mid 1 \le k \le 3 \} \le n$,
\item[{\rm (b)}] $\begin{cases}
\ntau{2} \supseteq \ntau{3},  \ \ntau{1} \supseteq \ntau{2}_{> 2s} \text{ and } \ntau{1} \not\supseteq \ntau{2}_{> 2s-2} & \text{ if } x =1, \\
\ntau{1} \supseteq \ntau{2},  \ \ntau{2} \supseteq \ntau{3}_{> 2s} \text{ and } \ntau{2} \not\supseteq \ntau{3}_{> 2s-2} & \text{ if } x =0.
\end{cases}
$
\end{enumerate}

\begin{theorem}
\label{thm: main Rd}
The subcrystal $R^x(\fw_{n+1}+\fw_{n+1-2s})$ is isomorphic to $B(\fw_{n+1}+\fw_{n+1-2s})$.
\end{theorem}

It is straightforward to see that Theorem~\ref{thm: main Rd} holds by the following two lemmas.

\begin{lemma} \label{lem: stable Rd}
$R^x(\fw_{n+1}+\fw_{n+1-2s})$ is stable under the Kashiwara operators $e_i$ and $f_i$. In particular, $R^x(\fw_{n+1}+\fw_{n+1-2s})$ is a subcrystal consisting of connected components.
\end{lemma}

\begin{proof}
By applying the similar argument in the proof of Proposition \ref{prop: same inner shape} twice for $(\ntau{1},\ntau{2})$ and $(\ntau{2},\ntau{3})$, one can check our assertion.
\end{proof}

\begin{lemma}
\label{lem: connected Rd}
The subcrystal $R^x(\fw_{n+1}+\fw_{n+1-2s})$ is connected.
\end{lemma}

\begin{proof}
Each element in $R^x(\fw_{n+1}+\fw_{n+1-2s})$ is connected to
\[
\begin{cases}
\left( \nn{$\emptyset$},  \nn{$\emptyset$}, \uplambda\left(\nn{$2s$-$1$}\right) \right)   & \text{ if } x=0, \\[10pt]
\left(\nn{$\emptyset$},   \uplambda\left(\nn{$2s$-$1$}\right), \nn{$\emptyset$} \right)   & \text{ if } x=1,
\end{cases}  \]
whose inner shapes are $(2s,2s)$ and $(2s)$, respectively, by a similar argument given in the proof of Theorem~\ref{thm: main as} using Lemma~\ref{lemma: general hw D} in place of Lemma~\ref{lemma: general hw B}.
Thus our assertion follows from Corollary~\ref{cor: unique}.
\end{proof}

Now, the following corollary follows from Theorem~\ref{thm: wt Riordan}.

\begin{corollary}
\label{cor:riordan_triangle_set}
For $m \ge 2s-1$,
let $S^x(\fw_{n+1}+\fw_{n+1-2s})_m$ be a subset of $R^x(\fw_{n+1}+\fw_{n+1-2s})$ consisting of $T=(\ntau{1},\ntau{2},\ntau{3})$ such that
\[
\ntau{1}\cup \ntau{2} \cup \ntau{3} = \uplambda(m)   \qquad \text{ and } \qquad \bigl( \ell(\ntau{1}), \ell(\ntau{2}), \ell(\ntau{3}) \bigr) \Vdash_0 m.
\]
Then we have
\[
\lvert S^x(\fw_{n+1}+\fw_{n+1-2s})_m \rvert = \Rior_{(m+1,2s)}.
\]
\end{corollary}

We remark that in Corollary~\ref{cor:riordan_triangle_set}, the almost even condition appears by the weight consideration as per Remark~\ref{rmk:weight consideration}.

\subsubsection{The mixed spin case $B(\fw_{n}) \otimes B(\fw_{n+1})^{\otimes 2}$}

A straightforward computation yields:

\begin{lemma}
\label{lemma: tens3 odd}
We have
\[
B(\fw_{n}) \otimes B(\fw_{n+1})^{\otimes 2} \iso \bigoplus_{s=0}^{\lfloor n/2\rfloor} B(\fw_{n+1}+\tfw_{n-2s})^{\oplus s+1}
\oplus \bigoplus_{s=1}^{\lfloor (n+1)/2\rfloor} B(\fw_n+\fw_{n-2s-1})^{\oplus s}.
\]
\end{lemma}

\begin{example}
Let us consider $B(\fw_3) \otimes B(\fw_4)^{\otimes 2}$ over $D_4$. Then we have
\[
B(\fw_3) \otimes B(\fw_4)^{\otimes 2} \iso B(\fw_{3})^{\oplus 2} \oplus B(\fw_{1}+\fw_{4})^{\oplus 2} \oplus B(\fw_2+\fw_{3}) \oplus
 B(\fw_{3}+2\fw_4).
\]
where corresponding highest weight vector can be described as follows by taking $\clr(\ntau{1}_1)=1$:
\begin{align*} %\ytableausetup{smalltableaux}
B(\fw_{3})^{\oplus 2}  & \leftrightarrow \begin{ytableau}
\cdot& *(gray!40) \cdot &  \cdot & *(gray!40) \cdot &  \cdot \\
\cdot& *(gray!40) \cdot &  \cdot & *(gray!40) \cdot \\
3 & *(gray!40) 2 &  1 \\
\end{ytableau}\ , \
\begin{ytableau}
\cdot& *(gray!40) \cdot &  \cdot & *(gray!40) \cdot &  \cdot \\
\cdot& *(gray!40) \cdot &  2 & *(gray!40) 1 \\
3  \\
\end{ytableau}\ ,
& B(\fw_{1}+\fw_{4})^{\oplus 2} & \leftrightarrow
\begin{ytableau}
\cdot& *(gray!40) \cdot &  \cdot \\
\cdot& *(gray!40) \cdot \\
2 & *(gray!40)  1 \\
\end{ytableau}\ , \
\begin{ytableau}
\cdot& *(gray!40) \cdot &  \cdot \\
2 & *(gray!40)  1 \\
\end{ytableau}
\\
B(\fw_{2}+\fw_{3}) & \leftrightarrow
\begin{ytableau}
\cdot& *(gray!40) \cdot &  \cdot \\
\cdot& *(gray!40) \cdot \\
1 \\
\end{ytableau}\ ,
&
B(\fw_{3}+2\fw_4)   & \leftrightarrow \begin{ytableau}
\cdot
\end{ytableau}\ .
\end{align*}
By taking $\clr(\ntau{3}_1)=1$, the highest weight vector can be described as follows:
\begin{align*} %\ytableausetup{smalltableaux}
B(\fw_{3})^{\oplus 2}  & \leftrightarrow \begin{ytableau}
 *(gray!40) \cdot &  \cdot & *(gray!40) \cdot & \cdot & *(gray!40) \cdot \\
 *(gray!40) \cdot &  \cdot & *(gray!40) \cdot & 1 \\
*(gray!40) 3 & 2  \\
\end{ytableau}\ , \
\begin{ytableau}
 *(gray!40) \cdot &  \cdot & *(gray!40) \cdot & \cdot & *(gray!40) \cdot \\
 *(gray!40) \cdot &  3 & *(gray!40) 2 & 1
\end{ytableau}\ ,
&
B(\fw_{1}+\fw_{4})^{\oplus 2} & \leftrightarrow
\begin{ytableau}
*(gray!40) \cdot &  \cdot &*(gray!40) \cdot \\
*(gray!40) \cdot &  \cdot &*(gray!40) \cdot \\
*(gray!40) 2 &   1 \\
\end{ytableau}\ , \
\begin{ytableau}
*(gray!40) \cdot &  \cdot &*(gray!40) \cdot \\
*(gray!40) \cdot &  1 \\
*(gray!40) 2  \\
\end{ytableau}\ ,
\\
 B(\fw_{2}+\fw_{3}) &  \leftrightarrow
\begin{ytableau}
*(gray!40) \cdot&  \cdot &  *(gray!40) \cdot \\
*(gray!40) \cdot& 1
\end{ytableau}
&
B(\fw_{3}+2\fw_4) & \leftrightarrow
\begin{ytableau}
 *(gray!40) \cdot \\
  *(gray!40) \cdot
\end{ytableau}\ .
\end{align*}
\end{example}

\begin{lemma}
\label{lem: hw 3 odd 100}
Suppose $\clr(\ntau{1}_1)=1$.
\begin{enumerate}
\item[{\rm (1)}] The $(s+1)$-many highest weight elements of weight $\fw_{n+1}+\tfw_{n-2s}$ for $0 \le s \le \lfloor n/2\rfloor$ are given by
\[
\mathscr{O}^0_{(1,0,0)}(s,t) \seteq \left(\gn{$\emptyset$},  \uplambda\left(\nn{$2t$}\right), \uplambda\left(\nn{$2s$},\gn{$2t$+$1$}\right) \right) \quad \text{ for } 0 \le t \le s,
\]
whose inner shape is $(2s+1,2s-2t)$ for $0 \le t \le s$, as a standard spin rigid tableau.
\item[{\rm (2)}] The $s$-many highest weight elements of weight $\fw_{n}+\fw_{n-2s+1}$ for $1 \le s \le \lfloor (n+1)/2\rfloor$ are given by
\[
\mathscr{O}^1_{(1,0,0)}(s,t) \seteq \left(\gn{$\emptyset$},  \uplambda\left(\nn{$2t$-$2$}\right), \uplambda\left(\nn{$2s$-$1$},\nn{$2t$-$1$}\right) \right) \quad \text{ for } 1 \le t \le s,
\]
whose inner shape is $(2s+1,2s-2t+2)$ for $1 \le t \le s$, as a standard spin rigid tableau.
\end{enumerate}
\end{lemma}

\begin{lemma}
\label{lem: hw 3 odd 001}
Suppose $\clr(\ntau{3}_1)=1$.
\begin{enumerate}
\item[{\rm (1)}] The $(s+1)$-many highest weight elements of weight $\fw_{n+1}+\tfw_{n-2s}$ for $0 \le s \le \lfloor n/2\rfloor$ are given by
\[
\mathscr{O}^0_{(0,0,1)}(s,t) \seteq \left(\nn{$\emptyset$},  \uplambda\left(\nn{$2t$-$1$}\right), \uplambda\left(\gn{$2s$},\gn{$2t$}\right) \right) \quad \text{ for } 0 \le t \le s,
\]
whose inner shape is $(2s+1,2s-2t+1)$ for $0 \le t \le s$, as a standard spin rigid tableau.
\item[{\rm (2)}] The $s$-many highest weight elements of weight $\fw_{n}+\fw_{n-2s+1}$ for $1 \le s \le \lfloor (n+1)/2\rfloor$ are given by
\[
\mathscr{O}^1_{(0,0,1)}(s,t) \seteq \left(\nn{$\emptyset$},  \uplambda\left(\nn{$2t$-$1$}\right), \uplambda\left(\gn{$2s$-$1$},\nn{$2t$}\right) \right) \quad \text{ for } 1 \le t \le s,
\]
whose inner shape is $(2s+1,2s-2t+1)$ for $1 \le t \le s$, as a standard spin rigid tableau.
\end{enumerate}
\end{lemma}

\begin{corollary}
\label{cor: unique odd}
\hfill
\begin{enumerate}
\item[{\rm (1)}] There exists a unique highest weight element of weight $\fw_{n+1}+\tfw_{n-2s}$ for $0 \le s \le \lfloor n/2\rfloor$ whose inner shape is $(2s+1)$ and $\clr(\ntau{1}_1)=1$.
\item[{\rm (2)}] There exists a unique highest weight element corresponding to $\fw_{n+1}+\tfw_{n-2s}$ for $0 \le s \le \lfloor n/2\rfloor$ whose inner shape is $(2s+1,2s+1)$ and $\clr(\ntau{3}_1)=1$.
\end{enumerate}
\end{corollary}

\begin{definition}
For each $0 \le s \le \lfloor n/2 \rfloor$ and $x=0,1$, we define a subset $\mathscr{R}^x(\fw_{n+1}+\fw_{n-2s})$ of SSSRT consisting of $T = (\ntau{1},\ntau{2},\ntau{3})$ satisfying the following conditions:
\begin{enumerate}
\item[{\rm (a)}] $\max\{ \ntau{k}_1 \mid 1 \le k \le 3 \} \le n$ and $\clr(\ntau{k}_1,\ntau{k}_2,\ntau{k}_3)= \begin{cases}
(1,0,0) &\text{ if } x=1, \\
(0,0,1) &\text{ if } x=0,
\end{cases}$
\item[{\rm (b)}] $T$ is a SSSRT with inner shape $(2s+1, \delta_{0x}(2s+1))$.
\end{enumerate}
\end{definition}

Alternatively, the definition $\mathscr{R}^x(\fw_{n+1}+\fw_{n-2s})$ $(x=0,1)$ can be written as the following form: The set of SSSRT consisting of $T=(\ntau{1},\ntau{2},\ntau{3})$ such that
\begin{enumerate}
\item[{\rm (a)}] $\max\{ \ntau{k}_1 \mid 1 \le k \le 3 \} \le n$ and $\clr(\ntau{k}_1,\ntau{k}_2,\ntau{k}_3)= \begin{cases}
(1,0,0) &\text{ if } x=1, \\
(0,0,1) &\text{ if } x=0,
\end{cases}$
\item[{\rm (b)}] $\begin{cases}
\ntau{2} \supseteq \ntau{3},  \ \ntau{1} \supseteq \ntau{2}_{\ge 2s+2} \text{ and } \ntau{1} \not\supseteq \ntau{2}_{\ge 2s} & \text{ if } x =1, \\
\ntau{1} \supseteq \ntau{2},  \ \ntau{2} \supseteq \ntau{3}_{\ge 2s+2} \text{ and } \ntau{2} \not\supseteq \ntau{3}_{\ge 2s} & \text{ if } x =0.
\end{cases}
$
\end{enumerate}

\begin{example}
\hfill
\begin{enumerate}
\item For a connected component of $\begin{ytableau}
\cdot& *(gray!40) \cdot &  \cdot \\
2 & *(gray!40)  1
\end{ytableau}$ with $\clr(\ntau{k}_1,\ntau{k}_2,\ntau{k}_3)=(1,0,0)$, we know that $\dim V(\fw_{3}+\fw_{6})_{\fw_1+\fw_6}=10=\Rior_{(5,3)}$ is represented by following semistandard spin rigid tableaux:
\begin{align*}
&\begin{ytableau}
\cdot& *(gray!40) \cdot &  \cdot &  *(gray!40)1\\
4 & *(gray!40)  3 \\
2
\end{ytableau}, \
\begin{ytableau}
\cdot& *(gray!40) \cdot &  \cdot &  *(gray!40)2\\
4 & *(gray!40)  3 \\
1
\end{ytableau}, \
\begin{ytableau}
\cdot& *(gray!40) \cdot &  \cdot &  *(gray!40)1\\
4 & *(gray!40)  2 \\
3
\end{ytableau}, \
\begin{ytableau}
\cdot& *(gray!40) \cdot &  \cdot &  *(gray!40) 2 & 1\\
4 & *(gray!40)  3
\end{ytableau}, \
\begin{ytableau}
\cdot& *(gray!40) \cdot &  \cdot \\
4 & *(gray!40)  3 \\
2 & *(gray!40)  1
\end{ytableau}, \\
& \begin{ytableau}
\cdot& *(gray!40) \cdot &  \cdot \\
4 & *(gray!40)  2 \\
3 & *(gray!40)  1
\end{ytableau}, \
\begin{ytableau}
\cdot& *(gray!40) \cdot &  \cdot & *(gray!40)1\\
4 & *(gray!40)  3 & 2
\end{ytableau}, \
\begin{ytableau}
\cdot& *(gray!40) \cdot &  \cdot & *(gray!40)3\\
4 & *(gray!40)  2 & 1
\end{ytableau}, \
\begin{ytableau}
\cdot& *(gray!40) \cdot &  \cdot & *(gray!40)4\\
3 & *(gray!40)  2 & 1
\end{ytableau}, \
\begin{ytableau}
\cdot& *(gray!40) \cdot &  \cdot & *(gray!40)2\\
4 & *(gray!40)  3 & 1
\end{ytableau}.
\end{align*}

\item For a connected component of $\begin{ytableau}
*(gray!40) \cdot &  \cdot & *(gray!40) \cdot \\
*(gray!40) \cdot &  \cdot & *(gray!40) \cdot \\
*(gray!40) 2 &   1
\end{ytableau}$ with $\clr(\ntau{k}_1,\ntau{k}_2,\ntau{k}_3)=(0,0,1)$, we know that $\dim V(\fw_{3}+\fw_{6})_{\fw_1+\fw_6}=10=\Rior_{(5,3)}$ is represented by following semistandard spin rigid tableaux:
\begin{align*}
& \begin{ytableau}
*(gray!40) \cdot &  \cdot & *(gray!40) \cdot & 4 & *(gray!40) 3\\
*(gray!40) \cdot &  \cdot & *(gray!40) \cdot \\
*(gray!40) 2 &   1
\end{ytableau}, \
\begin{ytableau}
*(gray!40) \cdot &  \cdot & *(gray!40) \cdot & 4 & *(gray!40) 1\\
*(gray!40) \cdot &  \cdot & *(gray!40) \cdot \\
*(gray!40) 3 &   2
\end{ytableau}, \
\begin{ytableau}
*(gray!40) \cdot &  \cdot & *(gray!40) \cdot & 2 & *(gray!40) 1\\
*(gray!40) \cdot &  \cdot & *(gray!40) \cdot \\
*(gray!40) 4 &   3
\end{ytableau}, \
\begin{ytableau}
*(gray!40) \cdot &  \cdot & *(gray!40) \cdot & 3 & *(gray!40) 1\\
*(gray!40) \cdot &  \cdot & *(gray!40) \cdot \\
*(gray!40) 4 &   2
\end{ytableau}, \
\begin{ytableau}
*(gray!40) \cdot &  \cdot & *(gray!40) \cdot & 4 & *(gray!40) 2\\
*(gray!40) \cdot &  \cdot & *(gray!40) \cdot \\
*(gray!40) 3 &   1
\end{ytableau}, \\
&\begin{ytableau}
*(gray!40) \cdot &  \cdot & *(gray!40) \cdot & 3 & *(gray!40) 2\\
*(gray!40) \cdot &  \cdot & *(gray!40) \cdot \\
*(gray!40) 4 &   1
\end{ytableau}, \
\begin{ytableau}
*(gray!40) \cdot &  \cdot & *(gray!40) \cdot & 4\\
*(gray!40) \cdot &  \cdot & *(gray!40) \cdot \\
*(gray!40) 3 &   2 & *(gray!40) 1
\end{ytableau}, \
\begin{ytableau}
*(gray!40) \cdot &  \cdot & *(gray!40) \cdot & 3\\
*(gray!40) \cdot &  \cdot & *(gray!40) \cdot \\
*(gray!40) 4 &   2 & *(gray!40) 1
\end{ytableau}, \
\begin{ytableau}
*(gray!40) \cdot &  \cdot & *(gray!40) \cdot & 2\\
*(gray!40) \cdot &  \cdot & *(gray!40) \cdot \\
*(gray!40) 4 &   3 & *(gray!40) 1
\end{ytableau}, \
\begin{ytableau}
*(gray!40) \cdot &  \cdot & *(gray!40) \cdot & 1\\
*(gray!40) \cdot &  \cdot & *(gray!40) \cdot \\
*(gray!40) 4 &   3 & *(gray!40) 2
\end{ytableau}.
\end{align*}
\end{enumerate}
\end{example}

\begin{theorem} \label{thm: main Rd odd}
The subcrystal $\mathscr{R}^x(\fw_{n+1}+\fw_{n-2s})$ is isomorphic to $B(\fw_{n+1}+\fw_{n-2s})$.
\end{theorem}

\begin{proof}
As Lemma~\ref{lem: stable Rd}, the stability condition under Kashiwara operators can be easily checked.
Similar to the proof of Lemma~\ref{lem: connected Rd}, each element in $R^x(\fw_{n+1}+\fw_{n-2s})$ is connected to
\[
\begin{cases}
\left( \nn{$\emptyset$},  \nn{$\emptyset$}, \uplambda\left(\gn{$2s$}\right) \right)   & \text{ if } x=0, \\
\left(\gn{$\emptyset$},   \uplambda\left(\nn{$2s$}\right), \nn{$\emptyset$} \right)   & \text{ if } x=1,
\end{cases} \]
whose inner shapes are $(2s+1,2s+1)$ and $(2s+1)$, respectively. Thus our assertion follows from Corollary~\ref{cor: unique odd}.
\end{proof}

Similar to Corollary~\ref{cor:riordan_triangle_set}, we have the following corollary.

\begin{corollary} \label{cor: main Rd odd}  For $m \ge 2s$,
let $\mathscr{S}^x(\fw)_m$ be the subset of $\mathscr{R}^x(\fw_{n+1}+\fw_{n-2s})$ consisting of $T=(\ntau{1},\ntau{2},\ntau{3})$ such that $\ntau{1} \cup \ntau{2} \cup \ntau{3}=\uplambda(m)$ and either
\begin{itemize}
\item $(\lambda_1,\lambda_2+(2s+1),\lambda_3+(2s+1)) \Vdash_0 m+4s+2$ if $x = 1$,
\item $(\lambda_1,\lambda_2,\lambda_3+(2s+1)) \Vdash_0 m+2s+1$ if $x = 0$,
\end{itemize}
where $\lambda_i = \ell(\ntau{i})$ for $1 \le i \le 3$.
Then we have
\[
\lvert \mathscr{S}^x(\fw_{n+1}+\fw_{n-2s})_m \rvert = \Rior_{(m+1,2s+1)}.
\]
\end{corollary}

We remark that the almost even condition appears in Corollary~\ref{cor: main Rd odd} by the weight consideration as per Remark~\ref{rmk:weight consideration}.

%%%%%%%%%%%%%%%%%%%%%%%%%%%%%%%%%%%%%%%%
\section{Open problems}
\label{sec:problems}

We conclude with a list of open problems that came up during the course of this work.

\begin{problem}
Determine if the lattice paths used in Theorem~\ref{thm:general_det_formula_Cn} can give a formula for the factorial characters in type $C_n$ similar to Hamel and King~\cite{HK17}.
\end{problem}

\begin{problem}
Determine if the proof of Theorem~\ref{thm:Touchard_identity} can be extended to prove the $q$-Touchard's identity of~\cite[Thm.~1]{Andrews10}, a triangle version, or an alternative form using $\Cat'_n(q)$.
\end{problem}

\begin{problem}
Determine a combinatorial interpretation of $\Mot'_n(q)$, $\Mot'_{(n,k)}(q)$, and $\Rior'_{(n,k)}(q)$.
\end{problem}

Our initial approach in trying to prove Proposition~\ref{prop:determinant_double_spin_B} was based on using the LGV lemma and a crystal model on non-intersecting lattice paths. Indeed, we consider the infinite $\Z^2$ grid as the infinite board with tiles given by Figure~\ref{fig:lattice_board}.

\begin{figure}
\[
\begin{tikzpicture}[yscale=0.5]
\clip (-0.65,-14.65) rectangle (6.65, 1.65);
\foreach \y in {1,...,10} {
  \draw (0, \y-9+0.15) node {$\ddots$};
}
\foreach \y in {1,...,10} {
  \draw (6, \y-15+0.15) node {$\ddots$};
}
% Draw the background grid
\draw[step=1,xshift=0.5cm,yshift=0.5cm,very thin,gray!40] (-2,2) grid (7,-16);
\foreach \x in {1,...,5} {
  \draw (\x,1-\x) node {$1$};
  \draw (\x,-8-\x) node {$1$};
  \draw (\x,0-\x+0.2) node {$\vdots$};
  \draw (\x,-7-\x+0.2) node {$\vdots$};
  \draw (\x,-1-\x) node {$n-2$};
  \draw (\x,-2-\x) node {$n-1$};
  \draw (\x,-3-\x) node {$n$};
  \draw (\x,-4-\x) node {$n$};
  \draw (\x,-5-\x) node {$n-1$};
  \draw (\x,-6-\x) node {$n-2$};
}
\node[circle,inner sep=1,fill=black] at (0.5,-9.5) {};
\node[gray!70,anchor=north east] at (0.5,-9.5) {$(0,0)$};
\end{tikzpicture}
\]
\caption{The board used to construct the non-intersecting lattice path model for crystals. The black dot denotes the origin $(0,0) \in \Z^2$.}
\label{fig:lattice_board}
\end{figure}
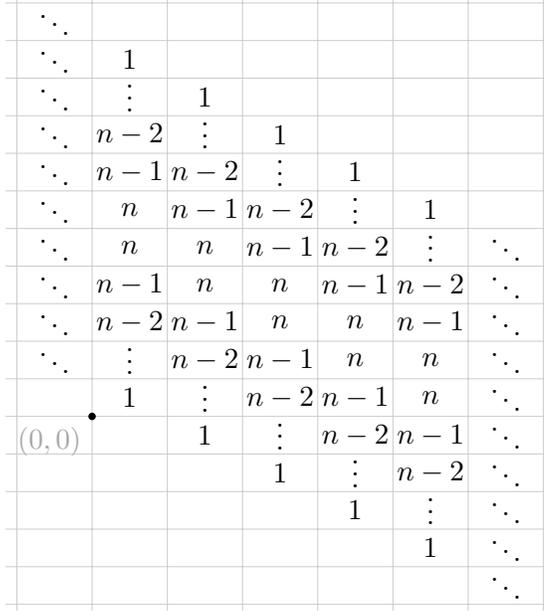

Next, we define a crystal structure on the non-intersecting lattice paths on this board as follows. An \defn{addable} (resp.\ \defn{removable}) $i$-corner is a path that has a corner on the south-east (resp.\ north-west) corner of an $i$ tile.
We define the \defn{good} addable (resp.\ removable) $i$-corner by using the signature rule given by reading top-to-bottom, left-to-right, where an addable $i$-corner is a $+$ and a removable $i$-corner is a $-$ and add a second $\pm$ if we can add (resp.\ remove) a second $n$-corner. The good addable (resp.\ removable) $i$-corner corresponds to the leftmost $+$ (resp.\ rightmost $-$) in the reduced signature.
The crystal structure is given by having an $f_i$ removing the good removable corner and $e_i$ adding the good addable corner.

Let $L_{n,i}$ denote the set of lattice paths from $(0,0)$ to $(i, 2n-i+1)$ with the aforementioned crystal structure.
We leave it to the interested reader to construct an explicit crystal isomorphism $\phi_i \colon L_{n,i} \to B(\tfw_i)$.
Furthermore, for a dominant integral weight $\lambda = \sum_{i \in I} c_i \tfw_i$, consider the sequence
\[
(i_k)_{k=0}^{\ell-1} = (\underbrace{n,\dotsc,n}_{c_n}, \dotsc, \underbrace{2,\dotsc,2}_{c_2}, \underbrace{1,\dotsc,1}_{c_1}),
\]
where $\ell = \sum_{i \in I} c_i$.
Let $L(\lambda)$ denote the set of non-intersecting lattice paths from $(k,-k)$ to $(i_k+k, 2n-i_k+1-k)$ with the crystal structure above. It is straightforward to see that this is compatible with the tensor product rule in $L_{n,i_0} \otimes \cdots \otimes L_{n,i_{\ell-1}}$. Hence, we obtain $L(\lambda) \iso B(\lambda)$ by applying $\phi_i$ to every factor $L_{n,i}$. Note that for $r \tfw_n$, we can uniquely extend each path $p_k$ to start at $(-k,-k)$ and end at $(k+n,k+n+1)$, where $k = 0, \dotsc, r-1$.

\begin{proposition}
\label{prop:NILP_model}
For type $B_n$, we have $L(\lambda) \iso B(\lambda)$.
\end{proposition}

\newcommand{\board}{
\clip (-0.65,-5.65) rectangle (6.65, 0.65);
\foreach \y in {1,2,3,4} {
  \draw (0, \y-4+0.2) node {$\ddots$};
}
\foreach \y in {5,6} {
  \draw (6, \y-10+0.2) node {$\ddots$};
}
% Draw the background grid
\draw[step=1,xshift=0.5cm,yshift=0.5cm,very thin,gray!40] (-3,2) grid (7,-7);
%\draw[gray!30] (-0.5,2.25) -- (-0.5,-6.5);
%\foreach \x in {0,...,5} {
%  \draw[gray!30] (\x+0.5,4.5-\x) -- (\x+0.5,-6.5-\x);
%}
\foreach \x in {1,...,5} {
  \foreach \i in {1,2,3} {
    \draw (\x,3-\i-\x) node {$\i$};
    \draw (\x,-4+\i-\x) node {$\i$};
  }
}
} % board

\begin{example}
\label{ex:nilp_crystal}
For $\mcL(2\tfw_3)$ in type $B_3$, the first few non-intersecting lattice paths are:
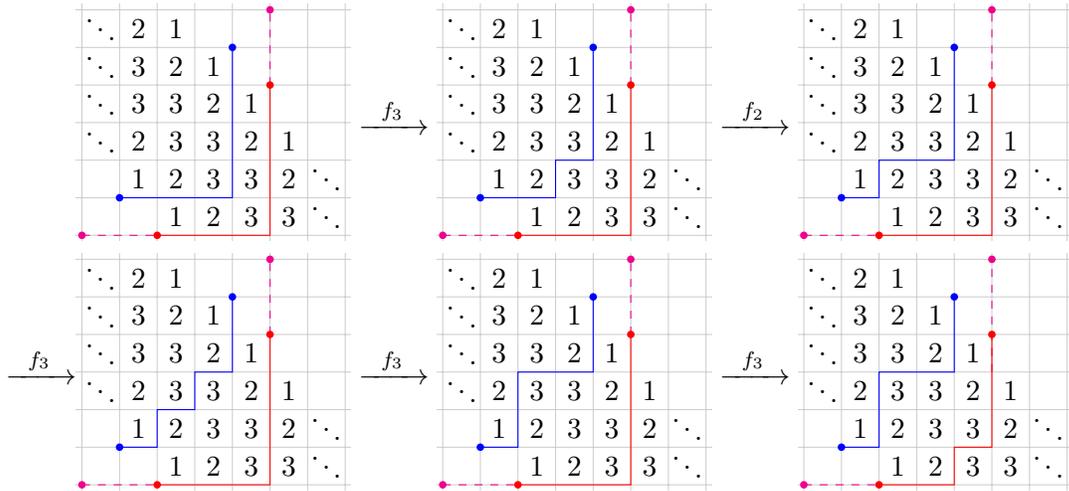
\begin{align*}
& \begin{tikzpicture}[scale=0.5,baseline=-40]
\board
\fill[blue] (0.5,-4.5) circle (0.1);
\fill[blue] (3.5,-0.5) circle (0.1);
\draw[blue] (0.5,-4.5) -- (3.5,-4.5) -- (3.5,-0.5);
\fill[magenta] (-0.5,-5.5) circle (0.1);
\draw[magenta,dashed] (-0.5,-5.5) -- (1.5,-5.5);
\fill[magenta] (4.5,0.5) circle (0.1);
\draw[magenta,dashed] (4.5,-1.5) -- (4.5,0.5);
\fill[red] (1.5,-5.5) circle (0.1);
\fill[red] (4.5,-1.5) circle (0.1);
\draw[red] (1.5,-5.5) -- (4.5,-5.5) -- (4.5,-1.5);
\end{tikzpicture}
\xrightarrow[\hspace{20pt}]{f_3}
\begin{tikzpicture}[scale=0.5,baseline=-40]
\board
\fill[blue] (0.5,-4.5) circle (0.1);
\fill[blue] (3.5,-0.5) circle (0.1);
\draw[blue] (0.5,-4.5) -- (2.5,-4.5) -- (2.5,-3.5) -- (3.5,-3.5) -- (3.5,-0.5);
\fill[magenta] (-0.5,-5.5) circle (0.1);
\draw[magenta,dashed] (-0.5,-5.5) -- (1.5,-5.5);
\fill[magenta] (4.5,0.5) circle (0.1);
\draw[magenta,dashed] (4.5,-1.5) -- (4.5,0.5);
\fill[red] (1.5,-5.5) circle (0.1);
\fill[red] (4.5,-1.5) circle (0.1);
\draw[red] (1.5,-5.5) -- (4.5,-5.5) -- (4.5,-1.5);
\end{tikzpicture}
\xrightarrow[\hspace{20pt}]{f_2}
\begin{tikzpicture}[scale=0.5,baseline=-40]
\board
\fill[blue] (0.5,-4.5) circle (0.1);
\fill[blue] (3.5,-0.5) circle (0.1);
\draw[blue] (0.5,-4.5) -- (1.5,-4.5) -- (1.5,-3.5) -- (3.5,-3.5) -- (3.5,-0.5);
\fill[magenta] (-0.5,-5.5) circle (0.1);
\draw[magenta,dashed] (-0.5,-5.5) -- (1.5,-5.5);
\fill[magenta] (4.5,0.5) circle (0.1);
\draw[magenta,dashed] (4.5,-1.5) -- (4.5,0.5);
\fill[red] (1.5,-5.5) circle (0.1);
\fill[red] (4.5,-1.5) circle (0.1);
\draw[red] (1.5,-5.5) -- (4.5,-5.5) -- (4.5,-1.5);
\end{tikzpicture}
\\
\xrightarrow[\hspace{20pt}]{f_3} &
\begin{tikzpicture}[scale=0.5,baseline=-40]
\board
\fill[blue] (0.5,-4.5) circle (0.1);
\fill[blue] (3.5,-0.5) circle (0.1);
\draw[blue] (0.5,-4.5) -- (1.5,-4.5) -- (1.5,-3.5) -- (2.5,-3.5) -- (2.5,-2.5) -- (3.5,-2.5) -- (3.5,-0.5);
\fill[magenta] (-0.5,-5.5) circle (0.1);
\draw[magenta,dashed] (-0.5,-5.5) -- (1.5,-5.5);
\fill[magenta] (4.5,0.5) circle (0.1);
\draw[magenta,dashed] (4.5,-1.5) -- (4.5,0.5);
\fill[red] (1.5,-5.5) circle (0.1);
\fill[red] (4.5,-1.5) circle (0.1);
\draw[red] (1.5,-5.5) -- (4.5,-5.5) -- (4.5,-1.5);
\end{tikzpicture}
\xrightarrow[\hspace{20pt}]{f_3}
\begin{tikzpicture}[scale=0.5,baseline=-40]
\board
\fill[blue] (0.5,-4.5) circle (0.1);
\fill[blue] (3.5,-0.5) circle (0.1);
\draw[blue] (0.5,-4.5) -- (1.5,-4.5) -- (1.5,-2.5) -- (3.5,-2.5) -- (3.5,-0.5);
\fill[magenta] (-0.5,-5.5) circle (0.1);
\draw[magenta,dashed] (-0.5,-5.5) -- (1.5,-5.5);
\fill[magenta] (4.5,0.5) circle (0.1);
\draw[magenta,dashed] (4.5,-1.5) -- (4.5,0.5);
\fill[red] (1.5,-5.5) circle (0.1);
\fill[red] (4.5,-1.5) circle (0.1);
\draw[red] (1.5,-5.5) -- (4.5,-5.5) -- (4.5,-1.5);
\end{tikzpicture}
\xrightarrow[\hspace{20pt}]{f_3}
\begin{tikzpicture}[scale=0.5,baseline=-40]
\board
\fill[blue] (0.5,-4.5) circle (0.1);
\fill[blue] (3.5,-0.5) circle (0.1);
\draw[blue] (0.5,-4.5) -- (1.5,-4.5) -- (1.5,-2.5) -- (3.5,-2.5) -- (3.5,-0.5);
\fill[magenta] (-0.5,-5.5) circle (0.1);
\draw[magenta,dashed] (-0.5,-5.5) -- (1.5,-5.5);
\fill[magenta] (4.5,0.5) circle (0.1);
\draw[magenta,dashed] (4.5,-2.5) -- (4.5,0.5);
\fill[red] (1.5,-5.5) circle (0.1);
\fill[red] (4.5,-1.5) circle (0.1);
\draw[red] (1.5,-5.5) -- (3.5,-5.5) -- (3.5,-4.5) -- (4.5,-4.5) -- (4.5,-1.5);
\end{tikzpicture}
\end{align*}
The dashed lines represent the extension to obtain the non-intersecting lattice paths for the LGV lemma applied to Proposition~\ref{prop:determinant_double_spin_B}. The corresponding KN tableaux in $B(2\tfw_3)$ are
\[
\young(11,22,33)
\xrightarrow[\hspace{20pt}]{f_3} \young(11,22,30)
\xrightarrow[\hspace{20pt}]{f_2} \young(11,23,30)
\xrightarrow[\hspace{20pt}]{f_3} \young(11,20,30)
\xrightarrow[\hspace{20pt}]{f_3} \young(11,20,3\othree)
\xrightarrow[\hspace{20pt}]{f_3} \young(11,20,0\othree)\ .
\]
\end{example}

However, via the LGV lemma, there are non-intersecting lattice paths that do not correspond to the identity permutation; \textit{e.g.},
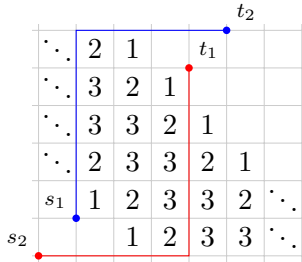
\[
\begin{tikzpicture}[scale=0.5,baseline=-40]
\begin{scope}
\board
\fill[blue] (0.5,-4.5) circle (0.1);
\fill[blue] (4.5,0.5) circle (0.1);
\draw[blue] (0.5,-4.5) -- (0.5,0.5) -- (4.5,0.5);
\fill[red] (-0.5,-5.5) circle (0.1);
\fill[red] (3.5,-0.5) circle (0.1);
\draw[red] (-0.5,-5.5) -- (3.5,-5.5) -- (3.5,-0.5);
\end{scope}
\node[anchor=south east] at (0.5,-4.5) {\scriptsize $s_1$};
\node[anchor=south east] at (-0.5,-5.5) {\scriptsize $s_2$};
\node[anchor=south west] at (3.5,-0.5) {\scriptsize $t_1$};
\node[anchor=south west] at (4.5,0.5) {\scriptsize $t_2$};
\end{tikzpicture}
\]
Hence, to obtain Proposition~\ref{prop:determinant_double_spin_B}, one would need to construct a sign-reversing involution on the corresponding non-intersecting lattice paths, where the fixed points are those non-intersecting lattice paths of Proposition~\ref{prop:NILP_model}.

\begin{problem}
Show Proposition~\ref{prop:determinant_double_spin_B} bijectively using the LGV lemma on the natural lattice paths for binomial coefficients and extend it to general $B(\lambda)$.
\end{problem}

We note the following Jacobi--Trudi-type formula was given in~\cite{Okada89} for type $B_n$:
\[
\ch(\lambda) = \det \left[ \ch_{(n+j-1)}(\tfw_{\lambda'_i - i + j}) \right]_{i,j=1}^{\ell},
\]
where $(\lambda'_1, \dotsc, \lambda'_{\ell})$ is the conjugate partition of $\lambda$ and $\ch_{(m)}(\mu)$ denotes the character of $V(\lambda)$ in type $B_m$.

%\begin{problem}
%\label{prob:LR_rigid_tableaux}
%Find a set of rigid tableaux that is equinumerous to $\mathsf{c}_{\mu,\nu}^\lambda$.
%\end{problem}

%Note that Problem~\ref{prob:LR_rigid_tableaux} can be solved by defining a rigid analog of Littlewood--Richardson tableaux (or more generally, a Littlewood--Richardson rule).

%\begin{problem}
%Define skew rigid tableaux $\SSRT(\lmm)$ such that skew rigid symmetric functions $\rr_{\lmm} = \sum_{T \in \SSRT(\lmm)} x^T$ such that
%\[
%\rr_{\lmm} =  \sum_{\nu} \mathsf{c}_{\mu,\nu}^{\lambda}(\xx) \rr_\nu.
%\]
%\end{problem}

\begin{problem}
For Theorem~\ref{thm:triangular_Catalan_det}, we can instead consider the graded characters/multiplicity in $(B^{n,1})^{\otimes 2m}$ in type $D_{n+1}^{(2)}$ as the corresponding classical type is $B_n$. Does the graded multiplicity have an interpretation as non-intersecting lattice paths or a determinant identity?
Similarly for $(\bigoplus_{k=0}^n B^{k,1})^{\otimes m}$ in type $A_{2n}^{(2)}$ with Theorem~\ref{thm:triangular_Catalan_det_type_C} (see Remark~\ref{remark:KR_tableaux}).
\end{problem}

\begin{example}
Let $M_{n,m}(q)$ denote the graded multiplicity of $B(0)$ inside of $(B^{n,1})^{\otimes 2m}$ in type $D_{n+1}^{(2)}$. Then we have
\begin{align*}
% n = 2
M_{2,1}(q) & = q^2
\allowdisplaybreaks \\ M_{2,2}(q) & = q^8+q^6+q^4
\allowdisplaybreaks \\ M_{2,3}(q) & = q^{18} + q^{16} + 2 q^{14} + q^{13} + 3 q^{12} + q^{11} + 2 q^{10} + q^9 + q^8 + q^6,
% n = 3
\allowdisplaybreaks \\ M_{3,1}(q) & = q^{3},
\allowdisplaybreaks \\ M_{3,2}(q) & = q^{12} + q^{10} + q^{8} + q^{6},
\allowdisplaybreaks \\ M_{3,3}(q) & = q^{27} + q^{25} + 2 q^{23} + 3 q^{21} + q^{20} + 4 q^{19} + 2 q^{18} + 4 q^{17}
  \\ & \hspace{10pt} + 2 q^{16} + 4 q^{15} + q^{14} + 2 q^{13} + q^{12} + q^{11} + q^{9},
%\\ M_{3,4}(q) & = q^{48} + q^{46} + 2 q^{44} + 4 q^{42} + q^{41} + 6 q^{40} + 2 q^{39} + 9 q^{38} + 4 q^{37} + 14 q^{36} + 7 q^{35}
%  \\ & \hspace{10pt} + 17 q^{34} + 10 q^{33} + 22 q^{32} + 13 q^{31} + 25 q^{30} + 16 q^{29} + 27 q^{28} + 17 q^{27} + 25 q^{26} + 16 q^{25}
%  \\ & \hspace{10pt} + 23 q^{24} + 13 q^{23} + 16 q^{22} + 9 q^{21} + 11 q^{20} + 5 q^{19} + 6 q^{18} + 2 q^{17} + 3 q^{16} + q^{15} + q^{14} + q^{12},
% n = 4
\allowdisplaybreaks \\ M_{4,1}(q) & = q^{4},
\allowdisplaybreaks \\ M_{4,2}(q) & = q^{16} + q^{14} + q^{12} + q^{10} + q^{8},
\allowdisplaybreaks \\ M_{4,3}(q) & = q^{36} + q^{34} + 2 q^{32} + 3 q^{30} + 4 q^{28} + q^{27} + 5 q^{26} + 2 q^{25} + 7 q^{24} + 3 q^{23}
\allowdisplaybreaks   \\ & \hspace{10pt} + 6 q^{22} + 3 q^{21} + 5 q^{20} + 2 q^{19} + 4 q^{18} + q^{17} + 2 q^{16} + q^{15} + q^{14} + q^{12},
%\\ M_{4,4}(q) & = q^{64} + q^{62} + 2 q^{60} + 4 q^{58} + 6 q^{56} + q^{55} + 9 q^{54} + 2 q^{53} + 14 q^{52} + 4 q^{51} + 19 q^{50} + 8 q^{49}
%  \\ & \hspace{10pt} + 27 q^{48} + 12 q^{47} + 35 q^{46} + 18 q^{45} + 44 q^{44} + 25 q^{43} + 52 q^{42} + 31 q^{41} + 61 q^{40} + 37 q^{39}
%  \\ & \hspace{10pt} + 64 q^{38} + 42 q^{37} + 66 q^{36} + 43 q^{35} + 63 q^{34} + 41 q^{33} + 56 q^{32} + 35 q^{31} + 43 q^{30} + 26 q^{29}
%  \\ & \hspace{10pt} + 32 q^{28} + 17 q^{27} + 19 q^{26} + 10 q^{25} + 12 q^{24} + 5 q^{23} + 6 q^{22} + 2 q^{21} + 3 q^{20} + q^{19} + q^{18} + q^{16}.
% n = 5
\allowdisplaybreaks \\ M_{5,1}(q) & = q^{5},
\allowdisplaybreaks \\ M_{5,2}(q) & = q^{20} + q^{18} + q^{16} + q^{14} + q^{12} + q^{10},
\allowdisplaybreaks \\ M_{5,3}(q) & = q^{45} + q^{43} + 2 q^{41} + 3 q^{39} + 4 q^{37} + 5 q^{35} + q^{34} + 7 q^{33} + 2 q^{32}
\allowdisplaybreaks \\ & \hspace{10pt} + 8 q^{31} + 4 q^{30} + 9 q^{29} + 4 q^{28} + 9 q^{27} + 4 q^{26} + 7 q^{25} + 3 q^{24}
\allowdisplaybreaks \\ & \hspace{10pt} + 5 q^{23} + 2 q^{22} + 4 q^{21} + q^{20} + 2 q^{19} + q^{18} + q^{17} + q^{15}
\end{align*}
Note that the coefficients of the lowest and highest degree terms of $M_{n,k}(q)$ match for fixed $k$ as $n$ varies (and the number of terms that agree depends on $k$).
\end{example}

While $\Cat'_n(q,t) \neq \Cat_n(q,t)$ as noted in Remark~\ref{rem:qt_catalan_are_different}, we conjecture the following relation, which we have verified for all $n \leq 10$.

\begin{conjecture}
We have
\[
\Cat_n(q,t) - \Cat'_n(q,t) = (qt - 1) f_n(q,t)
\]
for some $f_n(q,t) \in \Z_{\geq 0}[q,t]$.
\end{conjecture}

We note that introducing $t$ to count the negative powers of $q$ to obtain Theorem~\ref{thm:qt_polynomial} is artificial. Indeed, it was not motivated representation theory.

\begin{problem}
Give a natural representation theoretic interpretation of the $q, t$-Catalan triangle numbers $\Cat_{(n,k)}(q,t)$.
\end{problem}

One possible interpretation of the $q,t$ version of our paths is coming from the branching rule coming from the natural embedding $\mathfrak{sp}_{2n} \to \mathfrak{sl}_{2n}$. Indeed, recall that we can relate characters $\chi(\lambda)(x_1, \dotsc, x_n)$ in type $C_n$ with a specialized type $A_{2n-1}$ character
\[
\chi(\lambda)(x_1, x_1^{-1}, x_2, x_2^{-1}, \dotsc , x_n, x_n^{-1})
\]
(see, \textit{e.g.},~\cite[Lemma~1.5.1]{KT87}). Therefore, we take the type $A_n$ character specialized as $\chi(\lambda)(q, t, q^2, t^2, \dotsc, q^n, t^n)$ and take the corresponding type $C_n$ character $\chi(\lambda)$ from the branching rule, we have an representation theoretic interpretation. Note the bijection on alphabets
\begin{align*}
1 & < 2 < 3 < 4 < \cdots < 2n - 1 < 2n,
\\ 1 & < \one < 2 < \otwo < \cdots < n < \on,
\end{align*}
and the semistandard condition on King tableaux.
However, it is remains artificial in type $C_n$ since it comes from the branching rule. Yet, this construction leads to the following questions.

\begin{problem}
The $q,t$ version of the statistic $w'$ from Equation~\eqref{eq:path_statistic} on lattice paths yields a $q,t$-analog of the binomial coefficients.
Determine the properties these $q,t$-binomial coefficients satisfy.
Give a natural representation theoretic interpretation of these $q,t$-binomial coefficients.
Determine what relation these $q,t$-binomial coefficients have with the $q,t$-Catalan numbers $\Cat_n(q,t)$.
\end{problem}

Our proof of Theorem~\ref{thm:q_binomial_paths} is essentially algebraic: showing that our combinatorial objects with their statistics satisfy the recursion relation for $q$-binomial coefficients. However, $q$-binomials $\qbinom{n+m}{n}{q}$ have an interpretation as generating function for the size of a partition in an $n \times m$ rectangle. Thus, we have the following problem.

\begin{problem}
\label{prob:combinatorial_q_binom}
Find a combinatorial proof of Theorem~\ref{thm:q_binomial_paths} by finding a bijection between paths and partitions in an $n \times m$ rectangle that sends our statistic to the size of the partition.
\end{problem}

One potential approach to solving Problem~\ref{prob:combinatorial_q_binom} for the $m=n$ case would be to follow the proof of Theorem~\ref{thm:q_catalan_paths} using the results of~\cite[Thm.~1.3]{Stump13} with an analogous result to Theorem~\ref{thm:qt_polynomial}.

\begin{conjecture}
\label{conj:other_specialization}
For type $B_n$, we have
\[
C_{n+1}(q) K_n(q) = \nps'\bigl( B(3\fw_n) \bigr),
\]
where $\nps'(B)$ is the normalized specialization $\chi(B)(q, q^3, q^5, \cdots, q^{2n-1})$ and
\begin{align*}
K_n(q) & = \frac{(q + 1) \prod_{i=0}^{n-1} (q^{2n+1-i} + 1) \prod_{i=1}^{\lfloor n/2 \rfloor - 1} (q^{2n+1-2i} + 1)}{\prod_{i=1}^{\lfloor n/2 \rfloor} (q^{2i+2} + 1)}
\\ & = (q^{2n+1} + 1) (q^{n+2} + 1) (q+1) \prod_{i=1}^{\lfloor n/2 \rfloor} \frac{(q^{2n+1-i} + 1)(q^{2n-i} + 1)^2}{(q^{2i+2} + 1)}.
%\\ & = \frac{(q^{2n+1} + 1) (q^{n+2} + 1) (q+1) \prod_{i=1}^{\lfloor n/2 \rfloor} (q^{2n+1-i} + 1)(q^{2n-i} + 1)^2}{\prod_{i=1}^{\lfloor n/2 \rfloor} (q^{2i+2} + 1)}.
\end{align*}
\end{conjecture}

Note that Conjecture~\ref{conj:other_specialization} is the analog of Corollary~\ref{cor:catalan_q2} using instead the specialization used in~\cite{BKW16}. 
We expect that a proof of Conjecture~\ref{conj:other_specialization} can be given by a direct, but tedious and lengthly, computation.

%\begin{proof}
%From~\cite[Eq.~(3.23)]{BKW16} with the substitution $q \mapsto q^2$, we have
%\begin{align*}
%\nps'\bigl( B(3\fw_n) \bigr) & = \prod_{i=1}^n \frac{1 - q^{2(n-i+2)}}{1 - q^{2n-2i+1}} \prod_{1 \leq i < j \leq n} \frac{1 - q^{2(3+2n-i-j+1)}}{1-q^{2(2n-i-j+1)}}
%\\ & = \dfrac{[2n]!!}{[2n-1]!! [2]} \prod_{2 \leq j \leq n} \dfrac{[6+4n-2j] \cdots [6+4n-4j+4]}{[4n-2j] \cdots [4n-4j+4]}
%\\ & = \dfrac{[2n]!!}{[2n-1]!! [2]} \prod_{1 \leq j \leq n-1} \dfrac{[4+4n-2j] \cdots [6+4n-4j]}{[4n-2j-2] \cdots [4n-4j]}
%\\ & = \dfrac{[2n]!!}{[2n-1]!! [2]} \prod_{1 \leq j \leq n-1} \dfrac{[4+4n-2j] [2+4n-2j] [4n-2j]}{[4n-4j] [4n-4j+2] [4n-4j+4]}.
%\end{align*}
%Next, we have
%\[
%C_n(q) = \frac{1 - q^{n+1}}{1 - q} \times \frac{(q)_{2n}}{(q)_n^2} = \frac{(1 - q^{n+2}) \dotsm (1 - q^{2n})}{(1 - q^2) \dotsm (1 - q^n)}
%= \prod_{i=1}^{n-1} \frac{1 - q^{2n+1-i}}{1 - q^{n+1-i}} = \dfrac{[2n+1]!}{[n+1]!}
%\]
%where $(q)_n = \prod_{k=1}^n (1 - q^k)$ is the $q$-Pochhammer symbol. Next, we have
%\[
%K_n(q) = \langle 1 \rangle \dfrac{\langle 2n + 1 \rangle !}{\langle n+1 \rangle !} \cdot \dfrac{\langle 2n-1 \rangle !!}{\langle 2n - 2 \lfloor n/2 \rfloor \rangle !!} \cdot \dfrac{\langle 2 \rangle}{\langle 2 \lfloor n/2 \rfloor + 2\rangle !!}
%\]
%\Fixme{Finish this computation.}
%\end{proof}

% ========
\subsection*{Acknowledgments}

The authors would like to thank Jang Soo Kim, Christian Stump, and Ole Warnaar for useful discussions.
The authors additionally thank Ole Warnaar and Christian Stump for comments on earlier versions of this manuscript.
Furthermore, the authors thank Jang Soo Kim an initial proof of Proposition~\ref{prop:determinant_double_spin_B} using~\cite[Lemma 3.3]{BKW16} for $\dim V(r\tfw_n)$.
The authors would like to thank Christian Krattenthaler for pointing out that Theorem~\ref{thm:general_det_formula_Cn} was previously proven in~\cite{Okada89}.
The authors would like to thank Soichi Okada sending them~\cite{Okada89} and an English version of~\cite{Okada09} and for valuable discussions on his results.

TS would like to thank Ewha Womans University for its hospitality during his stay in June, 2017 and March, 2018. SjO would like to thank The University of Queensland for its hospitality during his stay in February, 2018.
This work benefited from the online databases {\sc OEIS}~\cite{OEIS} and {\sc FindStat}~\cite{FindStat} and computations using {\sc SageMath}~\cite{sage,combinat}.

\appendix

%%%%%%%%%%%%%%%%%%%%%%%%%%%%%%%%%%%%%%%%
\section{Other $q$-determinants}

We collect some other identities we noticed during our work on this paper.

\begin{proposition}
\label{prop:q_Catalan_Hankel}
We have
\[
\det \left[ \CatCR_{r+i+j}(q) \right]_{i,j=1}^n = \sum_{\nilp} \prod_{i=1}^n q^{a(P^{(i)})},
\]
where the sum is over all non-intersecting lattice paths $\nilp = (P^{(1)}, \dotsc, P^{(n)})$ from $(i_1, \dotsc, i_n)$ to $(j_1, \dotsc, j_n)$, respectively, and $a(d)$ is the area of the corresponding Dyck path.
\end{proposition}

\begin{proof}
This follows from the LGV lemma and taking the weight of the vertical edge $e$ to be $q^a$, where $a$ is the number of vertical steps directly to the left of $e$. In particular, this corresponds to the contribution of area of a Dyck path passing through $e$.
\end{proof}

We were unable to find a reference in the literature, but we believe that Proposition~\ref{prop:q_Catalan_Hankel} is likely known to experts.
However, Proposition~\ref{prop:q_Catalan_Hankel} gives a combinatorial proof of the following formulas due to Cigler~\cite{Cigler99}:
\[
\det \left[ \CatCR_{0+i+j}(q) \right]_{i,j=0}^n = q^{n(n+1)(4n-1)/6},
\hspace{40pt}
\det \left[ \CatCR_{1+i+j}(q) \right]_{i,j=0}^n = q^{n(n+1)(4n+5)/6}.
\]
Indeed, this is a consequence of that there is precisely one family of non-intersecting lattice paths,
\[
\sum_{k=1}^n \frac{2k(2k-1)}{2} = \frac{n(n+1)(4n-1)}{6},
\hspace{40pt}
\sum_{k=1}^n \frac{(2k+1)2k}{2} = \frac{n(n+1)(4n+5)}{6},
\]
and the Dyck path $E^k N^k$ has area $\sum_{i=1}^k i = \frac{k(k-1)}{2}$.

%We have the following $q$-Motzkin numbers given by~\cite{BSS93}:
%\begin{align*}
%\Mot_{(n,k)}(q) & = \frac{1}{[k+1]_q} \qbinom{2k}{k}{q} \qbinom{n}{2k},
%\\ \Mot_n(q) & = \sum_{k=0}^{\lfloor n/2 \rfloor} \Mot_{(n,k)}(q),
%\end{align*}
%where $\Mot_{(n,k)}$ are Motzkin paths with $k$ horizontal steps. Another $q$-Motzkin number was defined by~\cite{BDLFP98}
%\begin{align*}
%\Mot^{\natural}_{n+1}(q) & = \Mot^{\natural}_n(q) + \sum_{k=0}^{n-1} q^{(k+2)(n-k)} \Mot^{\natural}_k(q) \Mot^{\natural}_{n-k-1}(q), & \Mot^{\natural}_0(q) = 1, \\
%\widetilde{\Mot}_{n+1}(q) & = q^{n+2} \widetilde{\Mot}_n(q) + \sum_{k=0}^{n-1} q^{k+2} \widetilde{\Mot}_k(q) \widetilde{\Mot}_{n-k-1}(q), & \widetilde{\Mot}_0(q) = q.
%\end{align*}
%We note that $\widetilde{\Mot}_n(q) = q^{\binom{n+2}{2}} \Mot^{\natural}_n(q^{-1})$.
%Yet another variation was defined by Cigler~\cite{Cigler99}:
%\[
%\widetilde{\Mot}^{\dagger}_{n+1}(q) = \widetilde{\Mot}^{\dagger}_n(q) + \sum_{k=0}^{n-1} q^{k+1} \widetilde{\Mot}^{\dagger}_k(q) \widetilde{\Mot}^{\dagger}_{n-k-1}(q),
%\qquad\qquad \widetilde{\Mot}^{\dagger}_0(q) = 1,
%\]
%but there does not appear to be a relation between $\widetilde{\Mot}^{\dagger}_n(q)$, $\Mot_n(q)$, and $\Mot^{\natural}_n(q)$.

We consider the $q$-Motzkin numbers defined by Cigler~\cite{Cigler99}:
\[
\widetilde{\Mot}^{\dagger}_{n+1}(q) = \widetilde{\Mot}^{\dagger}_n(q) + \sum_{k=0}^{n-1} q^{k+1} \widetilde{\Mot}^{\dagger}_k(q) \widetilde{\Mot}^{\dagger}_{n-k-1}(q),
\qquad\qquad \widetilde{\Mot}^{\dagger}_0(q) = 1,
\]
We note that these are distinct from the $q$-Motzkin numbers defined in~\cite{BDLFP98,BSS93} and do not appear to be related by a simple closed formula. Cigler proved algebraically in~\cite[Eq.~(39)]{Cigler99} that
\[
\det [ \widetilde{\Mot}^{\dagger}_{i+j}(q) ]_{i,j=0}^{n-1} = q^{n(n-1)(2n-1)/6}.
\]
Furthermore, Cigler also showed in~\cite[Eq.~(40)]{Cigler99} that
\[
\det [ \widetilde{\Mot}^{\dagger}_{i+j+1}(q) ]_{i,j=0}^{n-1} = \kappa_n q^{2\binom{n}{3}} = \kappa_n q^{\binom{n}{2}} q^{n(n-1)(2n-1)/6},
\]
where $(\kappa_n) = (1,1,0,-1,-1,0,1,1,0,-1,-1,\ldots)$.

Based on numerical computations, we have the following conjectures.

\begin{conjecture}
\label{conj:factored_motzkin_2shifted}
Define
\[
f_n(q) \seteq \begin{cases}
\displaystyle \sum_{\substack{1 \leq k \leq n \\ k \not\equiv 1 \;\mathrm{mod}\; 3}} q^k & \text{if } n \equiv 0 \mod{3}, \\
\displaystyle (q+1) \left( \sum_{k=0}^{\lfloor n/3 \rfloor} q^{3k} \right) & \text{otherwise}.
\end{cases}
\]
Then we have
\[
\left( \det \left[ \widetilde{\Mot}^{\dagger}_{i+j+2}(q) \right]_{i,j=0}^{n-1} \right)_{n=1}^{\infty} = \bigl( q^{c_n} f_n(q) \bigr)_{n=1}^{\infty}.
\]
for some $c_n \in \Z_{\geq 0}$.
\end{conjecture}

% (q + 1)^2,
% (2*q + 1)

% (q^6 + 2*q^3 + 2*q^2 + 1),
% (q + 1)^2 (q^3+1)^2,
% (2*q^7 + q^6 + 2*q^4 + 2*q^3 + 2*q + 1),

% q^12 + 2*q^9 + 2*q^8 + 3*q^6 + 2*q^5 + 2*q^3 + 2*q^2 + 1,   ==  (q^6 + q^3 + 1) * (q^6 + q^3 + 2*q^2 + 1)
% (q + 1)^2 (q^6 + q^3 + 1)^2,
% (q^6 + q^3 + 1) * (2*q^7 + q^6 + q^3 + 2*q + 1),

% (q^6 + 1) * (q^12 + 2*q^9 + 2*q^8 + 2*q^6 + 2*q^5 + 2*q^3 + 2*q^2 + 1),
% (q + 1)^2 * (q^9 + q^6 + q^3 + 1)^2,
% (q^6 + 1) * (2*q^13 + q^12 + 2*q^10 + 2*q^9 + 2*q^7 + 2*q^6 + 2*q^4 + 2*q^3 + 2*q + 1),

% (q^12 + q^9 + q^6 + q^3 + 1) * (q^12 + q^9 + 2*q^8 + q^6 + q^3 + 2*q^2 + 1),
% (q + 1)^2 (q^12 + q^9 + q^6 + q^3 + 1)^2,
% (q^12 + q^9 + q^6 + q^3 + 1) * (2*q^13 + q^12 + q^9 + 2*q^7 + q^6 + q^3 + 2*q + 1)]

\begin{conjecture}
\label{conj:factored_motzkin_2shifted}
Define
\[
t_n(q) \seteq \left( \sum_{k=0}^{\lfloor n/3 \rfloor} q^{3k} \right).
\]
We have
\[
\left( \det \left[ \widetilde{\Mot}^{\dagger}_{i+j+3}(q) \right]_{i,j=0}^{n-1} \right)_{n=1}^{\infty} = \bigl( (-1)^{\lfloor n/3 \rfloor} q^{c_n} f_n(q) \bigr)_{n=1}^{\infty}.
\]
for some $c_n \in \Z_{\geq 0}$ and $g_n(q) \in \Z_{\geq 0}[q]$. Moreover, if $n \equiv 1 \mod{3}$ then $g_n(q) = (q+1)^2 t_n(q)^2$, and if $n \equiv 2 \mod{3}$, then $t_n(q)$ divides $g_n(q)$.
\end{conjecture}

\begin{conjecture}
\label{conj:factored_motzkin_positive}
For any $k,n \in \Z_{>0}$, we have
\[
\det \left[ \widetilde{\Mot}^{\dagger}_{i+j+2k}(q) \right]_{i,j=0}^{n-1} \in \Z_{\geq 0}[q].
\]
\end{conjecture}

We note that for Conjecture~\ref{conj:factored_motzkin_2shifted}, we have $f_{3m+1}(q) = f_{3m+2}(q)$. Moreover, the positivity of the Hankel determinants does not seem to extend to higher odd shifts since
\begin{align*}
\det \left[ \widetilde{\Mot}^{\dagger}_{i+j+5}(q) \right]_{i,j=0}^{2} & = -q^{38} + 2q^{34} + 6q^{33} + 4q^{32} + 4q^{31} + q^{30} - 6q^{29} - 10q^{28}
\\ & \hspace{20pt} - 8q^{27} - 12q^{26} - 8q^{25} - 3q^{24} - 4q^{23} - 4q^{22} - q^{20}.
\end{align*}

The LGV approach does not naturally extend to show Conjecture~\ref{conj:factored_motzkin_positive} as the only statistic that we can see that results in Cigler's $q$-Motzkin number is \defn{tunnel length}. Indeed, recall that a \defn{tunnel} for a Motzkin path is the positions $(i, j)$ of a pair of matching parentheses\footnote{Recall that we can represent a Motzkin path as a sequence of parentheses with a letter $x$ such that every open parenthesis `(' has a matching closing parenthesis `)'.} and the length of the tunnel is $j - i + 1$. Thus, define the tunnel length of $M$ as $T(M) = \sum_{(i,j)} j - i + 1$, where the sum is over all tunnels $(i,j)$ of $M$. The following proposition is straightforward from the recursion on Motzkin paths for the first return to $y=0$.

\begin{proposition}
We have
\[
\Mot^{\dagger}_n(q) = \sum_M q^{T(M)},
\]
where the sum is over all Motzkin paths of length $M$.
\end{proposition}

Note that this is not compatible with the LGV lemma as we cannot assign a weighting to each edge that agrees with the tunnel length. Contrast this with area for Dyck paths which yields Proposition~\ref{prop:q_Catalan_Hankel}. Hence, to prove Conjecture~\ref{conj:factored_motzkin_positive} combinatorially, one would need a statistic on Motzkin paths that can be considered as an edge weighting on the corresponding digraph to generated Motzkin paths.

%%%%%%%%%%%%%%%%%%%%%%%%%%%%%%%%%%%%%%%%
\section{Alternative proofs of Corollary~\ref{cor: Cn determinatal r omega_n}}
\label{app:alt_proofs}

We give two alternative proofs of Corollary~\ref{cor: Cn determinatal r omega_n}.

\begin{proof}[Alternative proof using the LGV for Catalan numbers]
The result follows from LGV lemma and a bijection with non-intersecting lattice paths and type $C_n$ King tableaux given by using $\Xi_{n+1}$ on each path $P^{(i)}$, with an appropriate shift of starting point, yielding the $i$-th column. It is straightforward to see that the non-intersecting condition corresponds to semistandard condition on King tableaux, similar to the case for the Jacobi--Trudi formula.
\end{proof}

\begin{proof}[Alternative proof using elementary manipulations]
From~\cite[Thm.~3]{Krat10}, we can compute
\[
\det [\Cat_{n+1+i+j} ]_{i,j=0}^{r-1}  = \prod_{0 \le i < j \le r-1}(j-i) \prod_{i=0}^{r-1} \dfrac{(i+r)!(2n+2+2i)!}{2i!(n+1+i)!(n+i+r+1)!}.
\]
Next, from Proposition~\ref{prop:Cn_ps}, we have
\[
\dim V(r \varpi_n) = \prod_{i=1}^r \dfrac{i}{n+i} \binom{2n+2r}{n+r-i}\binom{2n+2r}{r-i}^{-1} \prod_{1 \leq i < j \leq r} \frac{i+j}{2n+i+j}
\]
Now we shall prove that
\begin{align*}
& \prod_{0 \le i < j \le r-1}(j-i) \prod_{i=0}^{r-1} \dfrac{(i+r)!(2n+2+2i)!}{2i!(n+1+i)!(n+i+r+1)!}  \\
& \hspace{20ex} =  \prod_{i=1}^r \dfrac{i}{n+i} \binom{2n+2r}{n+r-i}\binom{2n+2r}{r-i}^{-1} \prod_{1 \leq i < j \leq r} \frac{i+j}{2n+i+j} \end{align*}
Using the notation defined in~\eqref{eq:simplified_factorials}, we have
\begin{align*}
\det [\Cat_{n+1+i+j} ]_{i,j=0}^{r-1} &= F(r-1) \dfrac{F(2r-1)}{F(r-1)}\dfrac{1}{\Phi(2r-2)}\dfrac{\Phi(2n+2r)}{\Phi(2n)}\dfrac{F(n)}{F(n+r)}\dfrac{F(n+r)}{F(n+2r)} \\
&= \dfrac{\Phi(2r-1)\Phi(2n+2r)F(n)}{\Phi(2n)F(n+2r)},
\end{align*}
while
\begin{align*}
\dim V(r \varpi_n) &= F(r)  \dfrac{F(n+r)}{F(n+2r)}\dfrac{F(2n+2r)}{F(2n+r)}\dfrac{F(n)}{F(n+r)} \times \dfrac{F(2n+r)}{F(2n+1)}\dfrac{\Phi(2n+1)}{\Phi(2n+2r-1)}\dfrac{\Phi(2r-1)}{F(r)} \\
&= \dfrac{\Phi(2r-1)\Phi(2n+2r)F(n)}{\Phi(2n)F(n+2r)}.
\end{align*}
Thus our assertion holds.
\end{proof}

\bibliographystyle{alpha}
\bibliography{catalan}{}
\end{document}